\documentclass[12pt,reqno]{amsart}
\usepackage{amsmath,amssymb,latexsym,esint,cite,mathrsfs}
\usepackage{verbatim,wasysym}
\usepackage[left=2.6cm,right=2.6cm,top=3cm,bottom=3cm]{geometry}
\makeatletter \@addtoreset{equation}{section} \makeatother

\usepackage{graphicx}
\usepackage{subfigure}
\usepackage{graphicx,epic,eepic}
\usepackage{tikz}  
\usepackage[colorlinks=true,urlcolor=blue,
citecolor=red,linkcolor=blue,linktocpage,pdfpagelabels,
bookmarksnumbered,bookmarksopen]{hyperref}
\numberwithin{equation}{section}
\newtheorem{theorem}{Theorem}[section]
\newtheorem{lemma}[theorem]{Lemma}

\newtheorem{proposition}[theorem]{Proposition}
\newtheorem{corollary}[theorem]{Corollary}
\newtheorem{remark}[theorem]{Remark}
\numberwithin{equation}{section}

\allowdisplaybreaks

\begin{document}

\title[On some sharp inequalities]
{On the Stein-Weiss inequalities and higher-order Caffarelli-Kohn-Nirenberg type inequalities: sharp constants, symmetry of extremal functions}

\author[S. Deng]{Shengbing Deng$^{\ast}$}
\address{\noindent Shengbing Deng
\newline
School of Mathematics and Statistics, Southwest University,
Chongqing 400715, People's Republic of China}\email{shbdeng@swu.edu.cn}

\author[X. Tian]{Xingliang Tian}
\address{\noindent Xingliang Tian  \newline
School of Mathematics and Statistics, Southwest University,
Chongqing 400715, People's Republic of China.}\email{xltian@email.swu.edu.cn}

\thanks{$^{\ast}$ Corresponding author}

\thanks{2020 {\em{Mathematics Subject Classification.}} 35J30, 49K30, 26D10}

\thanks{{\em{Key words and phrases.}} Stein-Weiss inequality; Second-order Caffarelli-Kohn-Nirenberg inequality; Rellich-Sobolev inequality; Weighted fourth-order equation; Symmetry of extremal functions}

\allowdisplaybreaks

\begin{abstract}
    In this paper, we first classify all radially symmetry solutions of the following weighted fourth-order equation
    \begin{equation*}
    \Delta(|x|^{-\gamma}\Delta u)=|x|^\gamma u^{\frac{N+4+3\gamma}{N-4-\gamma}},\quad u\geq 0 \quad \mbox{in}\quad \mathbb{R}^N,
    \end{equation*}
    where $N\geq 5$, $-2<\gamma<0$. Then we derive the sharp Stein-Weiss inequality and standard second-order Caffarelli-Kohn-Nirenberg inequality with radially symmetry extremal functions. Moreover, by using standard spherical decomposition, we derive a sharp weighted Rellich-Sobolev inequality. Furthermore, we establish the sharp second-order Caffarelli-Kohn-Nirenberg type inequalities with two variables which have radially symmetry extremal functions. Finally, we derive the weak form Hardy-Rellich inequalities with sharp constants.
\end{abstract}

\vspace{3mm}

\maketitle

\section{{\bfseries Introduction and main results}}\label{sectir}

\subsection{Stein-Weiss inequality}\label{subsectswr}

Let us recall the well-known Hardy-Littlewood-Sobolev inequality \cite{HL28,So63}, see also the work of Lieb \cite{Lieb83} in details in which the author classified extremal functions by using stereographic projection.

    \begin{proposition}\label{prohlsi}
    ({\bfseries Hardy-Littlewood-Sobolev inequality}). Let $r,t>1$ and $0<\lambda<N$ with $\frac{1}{r}+\frac{1}{t}+\frac{\lambda}{N}=2$, $f\in L^r(\mathbb{R}^N)$ and $g\in L^t(\mathbb{R}^N)$. There exists a sharp constant $C(N,r,t,\lambda)>0$, independent of $f$ and $g$, such that
    \begin{equation}\label{hlsi}
    \left|\int_{\mathbb{R}^N}\int_{\mathbb{R}^N}
    \frac{f(x)g(y)}{|x-y|^\lambda}\mathrm{d}y\mathrm{d}x\right|
    \leq C(N,t,\lambda)\|f\|_{L^r(\mathbb{R}^N)}\|g\|_{L^t(\mathbb{R}^N)}.
    \end{equation}
    If $t=r=\frac{2N}{2N-\lambda}$, then
    \begin{align*}
    C(N,t,\lambda)=C_1(N,\lambda)
    =\pi^{\frac{\lambda}{2}}\frac{\Gamma(\frac{N}{2}-\frac{\lambda}{2})}
    {\Gamma(N-\frac{\lambda}{2})}
    \left\{\frac{\Gamma(\frac{N}{2})}{\Gamma(N)}
    \right\}^{-1+\frac{\lambda}{N}},
    \end{align*}
    and there is equality in \eqref{hlsi} if and only if $f\equiv (const.) g$ and
    \begin{align*}
    f(x)=A(c+|x-z|^2)^{-\frac{2N-\lambda}{2}}
    \end{align*}
    for $A\in \mathbb{R}$, $c>0$ and $z\in \mathbb{R}^N$. Furthermore, when $t=2$ and $r=\frac{2N}{3N-2\lambda}$, which requires $N<2\lambda<2N$, then
    \begin{align*}
    C(N,t,\lambda)=C_2(N,\lambda)
    =\pi^{\frac{\lambda}{2}}\frac{\Gamma(\frac{N}{2}-\frac{\lambda}{2})}
    {\Gamma(N-\frac{\lambda}{2})}
    \left\{\frac{\Gamma(\lambda-\frac{N}{2})}
    {\Gamma(\frac{3N}{2}-\lambda)}\right\}^{\frac{1}{2}}
    \left\{\frac{\Gamma(\frac{N}{2})}{\Gamma(N)}
    \right\}^{-1+\frac{\lambda}{N}},
    \end{align*}
    and there is equality in \eqref{hlsi} if and only if
    \begin{align*}
    & f(x)=A(c+|x-z|^2)^{-\frac{3N-2\lambda}{2}},\quad g(x)=B\int_{\mathbb{R}^N}
    \frac{(c+|y-z|^2)^{-\frac{3N-2\lambda}{2}}}
    {|x-y|^{\lambda}}\mathrm{d}y
    \end{align*}
    for $A, B\in \mathbb{R}$, $c>0$ and $z\in \mathbb{R}^N$.
    \end{proposition}

    Hardy and Littlewood also introduce the double weighted inequaliy in $(0,\infty)$, which was later generalized in $\mathbb{R}^N$ by Stein and Weiss in \cite{SW58}. It reads:
    \begin{proposition}\label{proswi}
    ({\bfseries Stein-Weiss inequality}). Let $r,t>1$ and $0<\lambda<N$ with $\frac{1}{r}+\frac{1}{t}+\frac{\lambda+a+b}{N}=2$, $f\in L^r(\mathbb{R}^N)$ and $g\in L^t(\mathbb{R}^N)$. There exists a constant $C(N,t,\lambda,a,b)>0$, independent of $f$ and $g$, such that
    \begin{equation}\label{swi}
    \left|\int_{\mathbb{R}^N}\int_{\mathbb{R}^N}
    \frac{f(x)g(y)}{|x|^a|x-y|^\lambda|y|^b}\mathrm{d}y\mathrm{d}x\right|
    \leq C(N,t,\lambda,a,b)\|f\|_{L^r(\mathbb{R}^N)}\|g\|_{L^t(\mathbb{R}^N)},
    \end{equation}
    under the conditions
    \[
    a+b\geq 0,\quad 1-\frac{1}{r}-\frac{\lambda}{N}<\frac{a}{N}<1-\frac{1}{r}.
    \]
    \end{proposition}

    Lieb \cite{Lieb83} established the existence of extremal functions for \eqref{swi} with sharp constant when $r<\frac{t}{t-1}$ and $a,b\geq 0$, and the author claimed that no nontrivial extremal functions can exist when $r=\frac{t}{t-1}$.
    Furthermore, Beckner \cite{Be08} established a kind of sharp Stein-Weiss inequality which states that, for $f\in L^2(\mathbb{R}^N)$ and $0<\lambda<N$,
    \begin{equation*}
    \left|\int_{\mathbb{R}^N}\int_{\mathbb{R}^N}
    \frac{f(x)f(y)}{|x|^{\frac{N-\lambda}{2}}
    |x-y|^\lambda|y|^{\frac{N-\lambda}{2}}}\mathrm{d}y\mathrm{d}x\right|
    \leq \pi^{\frac{N}{2}}\frac{\Gamma(\frac{N-\lambda}{2})}
    {\Gamma(\frac{\lambda}{2})}
    \left\{\frac{\Gamma(\frac{\lambda}{4})}{\Gamma(\frac{2N-\lambda}{4})}
    \right\}^{2}\|f\|^2_{L^2(\mathbb{R}^N)},
    \end{equation*}
    with no nontrivial extremal functions, in which the author established the equivalence relationship with Pitt's inequality, that is, for $0< \alpha<N$,
    \begin{equation*}
    \int_{\mathbb{R}^N}\frac{|f|^2}{|x|^{\alpha}}\mathrm{d}x
    \leq 2^{-\alpha}\left\{\frac{\Gamma(\frac{N-\alpha}{4})}
    {\Gamma(\frac{N+\alpha}{4})}
    \right\}^{2}\int_{\mathbb{R}^N}|(-\Delta)^{\alpha/4}f|^2\mathrm{d}x.
    \end{equation*}

    Therefore, we want to establish some other sharp Stein-Weiss inequalities as in Proposition \ref{prohlsi}. From the work of Lieb \cite{Lieb83} (see also Chen and Li \cite{CL08}), we notice that the corresponding Euler-Lagrange equations of Stein-Weiss inequality \eqref{swi} are the integral system:
    \begin{eqnarray}\label{swele}
    \left\{ \arraycolsep=1.5pt
       \begin{array}{ll}
        \mu_1rf(x)^{r-1}=\frac{1}{|x|^a}\int_{\mathbb{R}^N}
        \frac{g(y)}{|y|^b|x-y|^\lambda}\mathrm{d}y
        \quad\mbox{in}\quad\mathbb{R}^N,\\[2mm]
        \mu_2tg(x)^{t-1}=\frac{1}{|x|^b}\int_{\mathbb{R}^N}
        \frac{f(y)}{|y|^a|x-y|^\lambda}\mathrm{d}y
        \quad\mbox{in}\quad\mathbb{R}^N,\\[2mm]
        f,g\geq 0\quad\mbox{in}\quad\mathbb{R}^N,
        \end{array}
    \right.
    \end{eqnarray}
    where $\mu_1,\mu_2>0$ are positive parameters. Let $w=c_1f^{r-1}$, $z=c_2g^{t-1}$, $q=\frac{1}{r-1}$, $p=\frac{1}{t-1}$, $pq\neq 1$, and for a proper choice of positive constants $c_1$ and $c_2$, then system \eqref{swele} becomes
    \begin{eqnarray}\label{swelet}
    \left\{ \arraycolsep=1.5pt
       \begin{array}{ll}
        w(x)=\frac{1}{|x|^a}\int_{\mathbb{R}^N}
        \frac{z(y)^p}{|y|^b|x-y|^\lambda}\mathrm{d}y
        \quad\mbox{in}\quad\mathbb{R}^N,\\[2mm]
        z(x)=\frac{1}{|x|^b}\int_{\mathbb{R}^N}
        \frac{w(y)^q}{|y|^a|x-y|^\lambda}\mathrm{d}y
        \quad\mbox{in}\quad\mathbb{R}^N,\\[2mm]
        w,z\geq 0\quad\mbox{in}\quad\mathbb{R}^N.
        \end{array}
    \right.
    \end{eqnarray}
    Similar as the work of Chen, Li and Ou \cite[Theorem 4.1]{CLO06}, we know integral system \eqref{swelet} is equivalent to (multiplies with suitable constants)
    \begin{eqnarray}\label{swelete}
    \left\{ \arraycolsep=1.5pt
       \begin{array}{ll}
        (-\Delta)^{\frac{N-\lambda}{2}}(|x|^aw)=\frac{w^p}{|x|^b}
        \quad\mbox{in}\quad\mathbb{R}^N,\\[2mm]
        (-\Delta)^{\frac{N-\lambda}{2}}(|x|^bz)=\frac{z^q}{|x|^a}
        \quad\mbox{in}\quad\mathbb{R}^N,\\[2mm]
        w,z\geq 0\quad\mbox{in}\quad\mathbb{R}^N.
        \end{array}
    \right.
    \end{eqnarray}
    Let $u(x)=|x|^a w(x)$, $v(x)=|x|^b z(x)$, then \eqref{swelete} becomes
    \begin{eqnarray}\label{sweletet}
    \left\{ \arraycolsep=1.5pt
       \begin{array}{ll}
        (-\Delta)^{\frac{N-\lambda}{2}}u=\frac{v^p}{|x|^{b(p+1)}}
        \quad\mbox{in}\quad\mathbb{R}^N,\\[2mm]
        (-\Delta)^{\frac{N-\lambda}{2}}v=\frac{u^q}{|x|^{a(q+1)}}
        \quad\mbox{in}\quad\mathbb{R}^N,\\[2mm]
        u,v\geq 0\quad\mbox{in}\quad\mathbb{R}^N.
        \end{array}
    \right.
    \end{eqnarray}
    By using the method of moving planes, Jin and Li \cite{JL06} obtain the symmetry result which states that, for $a,b\geq 0$, then positive solution pair $(u,v)$ of \eqref{sweletet} are radially symmetric and decreasing about some point.
    Furthermore, the work of Chen and Li \cite[Theorem 2]{CL08} indicates that if $\lambda=N-2$, $a=b=\frac{N}{p+1}-\frac{N-2}{2}$ and $p=q$ with $b(p+1)<2$, then $u(x)\equiv v(x)$ and they must both assume the form
    $$\phi_c(x)=d\left(c+|x|^{\frac{(N-2)(p-1)}{2}}\right)^{-\frac{2}{p-1}}
    $$
    with some real number $c$ and some constant $d$. In fact, this result gives the extremal functions and also the sharp constant of Stein-Weiss inequality \eqref{swi} in this case.
    Now, let us consider another case:
    \begin{align}\label{condition}
    \lambda=N-2,\quad p=1,\quad b(p+1)=a(q+1)<2,
    \end{align}
    then $a=\frac{b(N-4+2b)}{N-2b}$ and $0\leq b<1$. When $b=0$, Lin \cite{Li98} has classified all solutions of \eqref{sweletet}. We only consider $0<b<1$. Obviously, for the case \eqref{condition} we know \eqref{sweletet} is equivalent to the following weighted fourth-order equation
    \begin{equation}\label{lehe}
    \Delta(|x|^{-\gamma}\Delta u)=|x|^\gamma u^q,\quad u\geq 0 \quad \mbox{in}\quad \mathbb{R}^N,
    \end{equation}
    where $\gamma=-2b\in (-2,0)$, and $q=\frac{N+4+3\gamma}{N-4-\gamma}$. We will classify all solutions of \eqref{lehe}.

    \begin{theorem}\label{thmlehups}
    Assume that $N\geq 5$ and $-2<\gamma<0$. Let $u$ be a smooth solution of equation \eqref{lehe}, then $u$ is radially symmetric and has the following form:
    \begin{align}\label{defula}
    u(x)=U_\lambda(x)= \lambda^{\frac{N-4-\gamma}{2}}U(\lambda x)\quad \mbox{for}\quad \lambda\geq 0,\quad \mbox{with}\quad U(x)=\frac{C_{N,\gamma}}
    {(1+|x|^{2+\gamma})
    ^{\frac{N-4-\gamma}{2+\gamma}}}.
    \end{align}
    Here $C_{N,\gamma}=\left[(N-4-\gamma)(N-2)
    (N+\gamma)(N+2+2\gamma)\right]
    ^{\frac{N-4-\gamma}{4(2+\gamma)}}$.
    \end{theorem}

Therefore, by the results of Chen and Li \cite[Theorem 2]{CL08} and Theorem \ref{thmlehups}, we can classify the extremal functions of Stein-Weiss inequality \eqref{swi} when $t=r$ and $t=2$ separately, which extends the work of Lieb \cite{Lieb83}.

    \begin{theorem}\label{thmswia}
    ({\bfseries Sharp Stein-Weiss inequality}). Assume the conditions in Proposition \ref{proswi} hold with $\lambda=N-2$.
    If $N\geq 3$, $t=r=\frac{2N}{N+2-2b}$ and $0<a=b<1$, then there is equality in \eqref{swi} with sharp constant if and only if $f\equiv (const.) g$ and
    \begin{align*}
    & f(x)=A|x|^{-\frac{b(N+2(1-b))}{N-2(1-b)}}
    \left(c+|x|^{\frac{2(N-2)(1-b)}{N-2(1-b)}}
    \right)^{-\frac{N+2(1-b)}{2(1-b)}},
    \end{align*}
    for $A\in \mathbb{R}$, $c>0$.
    Moreover, if $N\geq 5$, $t=2$ and $r=\frac{N+4-6b}{2(N-2b)}$, with $a=\frac{b(N-4+2b)}{N-2b}$ and $0<b<1$, then there is equality in \eqref{swi} with sharp constant if and only if
    \begin{align*}
    f(x) & =A|x|^{-\frac{b(N+4-2b)}{N-2b}}
    (c+|x|^{2-2b})^{-\frac{N+4-2b}{2-2b}},\\
    g(x) & =B\left[(N-2b)c+(2-2b)|x|^{2-2b}\right]|x|^{-b}
    (c+|x|^{2-2b})
    ^{-\frac{N-2b}{2-2b}},
    \end{align*}
    for $A, B\in \mathbb{R}$, $c>0$.
    \end{theorem}

    Furthermore, as a direct conclusion of Theorem \ref{thmlehups}, we obtain
    \begin{theorem}\label{corolehups}
    Assume that $N\geq 5$ and $-2<\gamma<0$. For all
    \[
    u\in \mathcal{D}^{2,2}_{0,\gamma}(\mathbb{R}^N):=
    \left\{u\in C^\infty_0(\mathbb{R}^N): \int_{\mathbb{R}^N}\frac{|\Delta u|^2}{|x|^{\gamma}} \mathrm{d}x<\infty,\quad \int_{\mathbb{R}^N}
    |x|^{\gamma}|u|^{2^{**}_{0,\gamma}} \mathrm{d}x<\infty\right\},
    \]
    it holds that
    \begin{equation}\label{lehi}
    \int_{\mathbb{R}^N}\frac{|\Delta u|^2}{|x|^{\gamma}} \mathrm{d}x
    \geq \mathcal{S}_{N,\gamma}
    \left(\int_{\mathbb{R}^N}
    |x|^{\gamma}|u|^{2^{**}_{0,\gamma}} \mathrm{d}x\right)^{\frac{2}{2^{**}_{0,\gamma}}},
    \end{equation}
    where $2^{**}_{0,\gamma}:=\frac{2(N+\gamma)}{N-4-\gamma}$, and
    \begin{align}\label{defsclehi}
    \mathcal{S}_{N,\gamma}=\left(\frac{2}{2+\gamma}\right)
    ^{\frac{2(2+\gamma)}{N+\gamma}-4}
    \left(\frac{2\pi^{\frac{N}{2}}}{\Gamma(\frac{N}{2})}\right)
    ^{\frac{2(2+\gamma)}{N+\gamma}}
    \mathcal{B}\left(\frac{2(N+\gamma)}{2+\gamma}\right),
    \end{align}
    where $\mathcal{B}(M)=(M-4)(M-2)M(M+2)
    \left[\Gamma^2(\frac{M}{2})/(2\Gamma(M))\right]^{\frac{4}{M}}$ and $\Gamma$ is the Gamma function.
    Furthermore, the constant $\mathcal{S}_{N,\gamma}$ is sharp and equality holds if and only if $u(x)=AU_\lambda(x)$
    for $A\in\mathbb{R}$ and $\lambda>0$, where $U_\lambda$ is given as in \eqref{defula}.
    \end{theorem}

\subsection{Caffarelli-Kohn-Nirenberg inequality}\label{subsectsckni}

    We notice that inequality \eqref{lehi} is a standard second-order Caffarelli-Kohn-Nirenberg (we write CKN for short) inequality. Let us recall the famous CKN inequality which was firstly introduced in 1984 by Caffarelli, Kohn and Nirenberg in their celebrated work \cite{CKN84} of first-order, then it was extended to higher-order by Lin \cite{Li86}. Here, we only mention the case without interpolation term.
    \vskip0.25cm

    \begin{proposition}\label{prockni}
    ({\bfseries CKN inequality}) There exists a constant $C>0$ such that
    \begin{equation}\label{cknh}
    \||x|^{-a} D^m u\|_{L^p(\mathbb{R}^N)}\geq C \||x|^{-b} D^j u\|_{L^\tau(\mathbb{R}^N)},
    \quad \mbox{for all}\quad u\in C^\infty_0(\mathbb{R}^N),
    \end{equation}
    where $j\geq 0$, $m>0$ are integers, and
    \begin{align*}
    \frac{1}{p}-\frac{a}{N}>0,\quad \frac{1}{\tau}-\frac{b}{N}>0, \quad
    a\leq b\leq a+m-j, \quad
    \frac{1}{\tau}-\frac{b+j}{N}=\frac{1}{p}-\frac{a+m}{N}.
    \end{align*}
    Here
    \[
    |D^m u|:=\left(\sum_{|\mathbf{m}|=m}\left|\frac{\partial^{|\mathbf{m}|} u}{\partial x_1^{m_1}\cdots \partial x_N^{m_N}}\right|^2\right)^{1/2},\quad \mathbf{m}=(m_1,\ldots,m_N)\in \mathbb{N}^N, \quad \mathbb{N}:=\{0,1,2,\ldots\},
    \]
    thus,
    \begin{align}\label{defDuz}
    \||x|^{-a} D^m u\|^p_{L^p(\mathbb{R}^N)}=
    \int_{\mathbb{R}^N}|x|^{-ap}
    \left(\sum_{|\mathbf{m}|=m}\left|\frac{\partial^{|\mathbf{m}|} u}{\partial x_1^{m_1}\cdots \partial x_N^{m_N}}\right|^2\right)^{p/2}
    \mathrm{d}x.
    \end{align}
    In particular, when $j=0$, $m=1$ and $p=2$, it holds
    \begin{equation}\label{ckn1y}
    \int_{\mathbb{R}^N}|x|^{-2a}|D u|^2 \mathrm{d}x
    \geq C\left(\int_{\mathbb{R}^N}|x|^{-b\tau}|u|^\tau \mathrm{d}x\right)^{\frac{2}{\tau}},
    \quad \forall u\in C^\infty_0(\mathbb{R}^N),
    \end{equation}
    where $-\infty<a<a_c:=\frac{N-2}{2}$, $a\leq b\leq a+1$, $\tau=\frac{2N}{N-2(1+a-b)}$. Furthermore, when $j=0$, $m=2$ and $p=2$, it holds
    \begin{equation}\label{ckn2y}
    \int_{\mathbb{R}^N}|x|^{-2a}|D^2 u|^2 \mathrm{d}x
    \geq C\left(\int_{\mathbb{R}^N}|x|^{-b\tau}|u|^\tau \mathrm{d}x\right)^{\frac{2}{\tau}},
    \quad \forall u\in C^\infty_0(\mathbb{R}^N),
    \end{equation}
    where $-\infty<a<\frac{N-4}{2}$, $a\leq b\leq a+2$, $\tau=\frac{2N}{N-2(2+a-b)}$.
    \end{proposition}

    Since $|Du|^2=|\nabla u|^2$, then we rewrite \eqref{ckn1y} as the following classical first-order CKN inequality
    \begin{equation}\label{ckn1}
    \int_{\mathbb{R}^N}|x|^{-2a}|\nabla u|^2 \mathrm{d}x
    \geq C_{\mathrm{CKN1}}\left(\int_{\mathbb{R}^N}|x|^{-b\tau}|u|^\tau \mathrm{d}x\right)^{\frac{2}{\tau}},
    \quad \forall u\in C^\infty_0(\mathbb{R}^N),
    \end{equation}
    which was well studied in last thirty years. When $N\geq 3$, $0\leq a<a_c$ and $a\leq b<a+1$, Chou and Chu \cite{CC93} obtained $C_{\mathrm{CKN1}}$ in \eqref{ckn1} is achieved by explicit radial function using the moving planes method about the uniqueness of nonnegative solutions for related Euler-Lagrange equation. Horiuchi \cite{Ho97} obtained the same conclusion by using symmetrization method for more general $L^p$-case (see also
    \cite{LL17}), which states that for $v(x)=\rho^{-\frac{1}{\tau}}
    u(|x|^{\rho-1}x)$ with $\rho=\frac{N-2}{N-2-2a}$,
    \begin{align}\label{fockncv}
    &\int_{\mathbb{R}^N}|x|^{-b\tau}|u|^\tau \mathrm{d}x
    =\int_{\mathbb{R}^N}|x|^{\frac{(N-2)\tau}{2}-N}|v|^\tau \mathrm{d}x,
    \nonumber\\
    &\int_{\mathbb{R}^N}|x|^{-2a}|\nabla u|^2\mathrm{d}x
    \geq \rho^{-1-\frac{2}{\tau}}\int_{\mathbb{R}^N}|\nabla v|^2 \mathrm{d}x,
    \end{align}
    and the second equality holds if and only if $u$ is radial when $0<a<a_c$, then the sharp constant and extremal functions are obtained from classical Sobolev inequality \cite{Ta76} when $b=a$ and Hardy-Sobolev inequality \cite[Theorem 4.3]{Lieb83} when $a<b<a+1$ which implies $-2<\frac{(N-2)\tau}{2}-N<0$. Furthermore, Dolbeault et al. \cite{DELT09} proved this by using a simpler proof, that is, making the change $v(x)=|x|^{-a}u(x)$ which transforms \eqref{ckn1} into
    \begin{equation*}
    \int_{\mathbb{R}^N}|\nabla v|^2 \mathrm{d}x
    \geq C_{\mathrm{CKN1}}\left(\int_{\mathbb{R}^N}
    \frac{|v|^\tau}{|x|^{(b-a)\tau}} \mathrm{d}x\right)^{\frac{2}{\tau}}
    +a[(N-2)-a]\int_{\mathbb{R}^N}\frac{|v|^2}{|x|^2}\mathrm{d}x,
    \end{equation*}
    then result follows from Schwarz's symmetrization if $0\leq a<a_c$. When $b=a+1$ or $a<0$ and $b=a$, Catrina and Wang \cite{CW01} proved that $C_{\mathrm{CKN1}}$ is not achieved and for other cases it is always achieved. Furthermore, when $a<0$ and $a< b<b_{\mathrm{FS}}(a)$, where
    \begin{align}\label{deffsc}
    b_{\mathrm{FS}}(a):=\frac{N(a_c-a)}{2\sqrt{(a_c-a)^2+N-1}}
    +a-a_c,
    \end{align}
    Felli and Schneider \cite{FS03} proved the extremal function is non-radial by restricting it in the radial space and classifying linearized problem, thus $b_{\mathrm{FS}}(a)$ is usually called {\em Felli-Schneider curve}. Finally, in a celebrated paper, Dolbeault, Esteban and Loss \cite{DEL16} proved an optimal rigidity result by using the so-called {\em carr\'{e} du champ} method that when $a<0$ and $b_{\mathrm{FS}}(a)\leq b<a+1$, the extremal function is symmetry. See previous results as in Figure \ref{F1}. We also refer to \cite{BDMN17,DELM17} for the symmetry and symmetry breaking results about extremal functions of CKN inequality with interpolation term.

    \begin{figure}
\begin{tikzpicture}[scale=4]
		\draw[->,ultra thick](-2,0)--(0,0)node[below right]{$O$}--(0.6,0)node[below]{$a$};
		\draw[->,ultra thick](0,-0.6)--(0,0.8)node[left]{$b$};
        \draw[fill=gray,domain=0:-2]plot(\x,{4*(1-0.47*\x)/
        (2*((1-0.6*\x)^2+3)^0.5)
        +0.47*\x -1})--(-2,-0.6)--(-1.2,-0.6)--(0,0);
        \draw[fill=brown,domain=0:-2]plot(\x,{4*(1-0.47*\x)/
        (2*((1-0.6*\x)^2+3)^0.5)
        +0.47*\x -1})--(-2,-0.5)--(0.5,0.75)--(0.5,0.25)--(0,0);
        \draw[densely dashed](0.5,0)node[below]{$a_c$}--(0.5,0.75);
        \draw[densely dashed](-2,-0.5)--(0.5,0.75);
        \draw[densely dashed](-1.2,-0.6)--(0,0);
        \draw[-,ultra thick](0,-0.6)--(0,0.8);
        \draw[-,ultra thick](-2,0)--(0.6,0);

		\node[left] at(-0.03,0.5){$b=a+1$};
        \node[right] at (-0.4,-0.22){$b=a$};
        \node[right] at (-1.7,-0.48){$b=b_{\mathrm{FS}}(a)$};
        \node[right] at (-0.98,-0.51){$Symmetry\ breaking\ region$};
        \node[right] at (-1.32,0.3){$Symmetry\ region$};
\end{tikzpicture}
\caption{\small On extremal functions of the first-order CKN inequality \eqref{ckn1}. The {\em Felli-Schneider region}, or symmetry breaking region, appears in dark grey defined by $a<0$ and $a<b<b_{\mathrm{FS}}(a)$. And symmetry holds in the brown region defined by $a<0$ and $ b_{\mathrm{FS}}(a)\leq b<a+1$, also $0\leq a<a_c$ and $a\leq b<a+1$.}
\label{F1}
\end{figure}
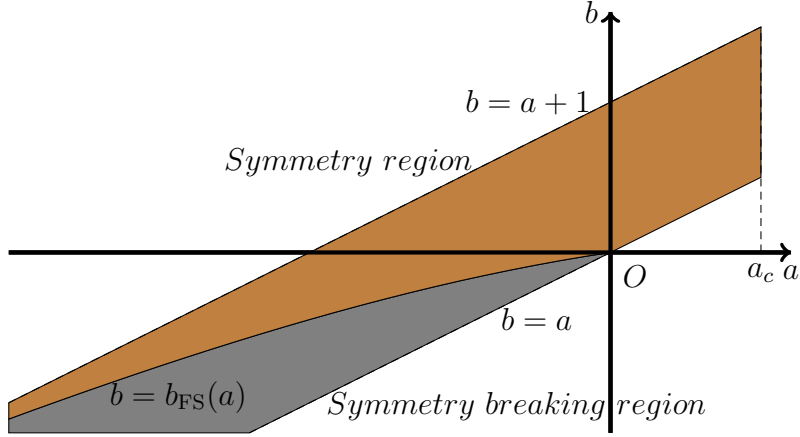

    However, the second-order case \eqref{ckn2y} is complicated. After integration by parts we get $\||x|^{-a}D^2 u\|_{L^2(\mathbb{R}^N)}=\||x|^{-a}\Delta u\|_{L^2(\mathbb{R}^N)}$ unless $a=0$. When $N\geq 5$ and $0\leq a<\frac{N-4}{2}$, Szulkin and Waliullah \cite[Lemma 3.1]{SW12} proved that $\||x|^{-a}D^2 u\|_{L^2(\mathbb{R}^N)}$ is equivalent to $\||x|^{-a}\Delta u\|_{L^2(\mathbb{R}^N)}$. For general $a$, Caldiroli and Musina \cite[Remark 2.3]{CM11} proved that $\||x|^{-a}D^2 u\|_{L^2(\mathbb{R}^N)}$ is equivalent to $\||x|^{-a}\Delta u\|_{L^2(\mathbb{R}^N)}$ if and only if
 \begin{align}\label{defac}
 a\neq -\frac{N}{2}-k\quad\mbox{and}\quad a\neq \frac{N-4}{2}+k, \quad\mbox{for all}\quad k\in\mathbb{N}.
 \end{align}
 Therefore, with addition condition \eqref{defac}, we have that \eqref{ckn2y} is equivalent to the following second-order CKN inequality
    \begin{equation}\label{ckn2}
    \int_{\mathbb{R}^N}|x|^{-2a}|\Delta u|^2 \mathrm{d}x
    \geq C_{\mathrm{CKN2}}\left(\int_{\mathbb{R}^N}|x|^{-b\tau}|u|^\tau \mathrm{d}x\right)^{\frac{2}{\tau}},
    \quad \forall u\in C^\infty_0(\mathbb{R}^N),
    \end{equation}
    which has it's own interests.
    Furthermore, Caldiroli and Musina \cite[Remark 2.3]{CM11} also showed that \eqref{ckn2} does not hold if $a=-\frac{N}{2}-k$ or $\frac{N-4}{2}+k$ for some $k\in\mathbb{N}$, which comes from spectral analysis about the weighted Rellich inequality (see \cite{CM12,MNSS21}): for all $u\in C^\infty_0(\mathbb{R}^N)$,
    \begin{align}\label{wri}
    \int_{\mathbb{R}^N}|x|^{-2a}|\Delta u|^2 \mathrm{d}x
    \geq \inf_{k\in\mathbb{N}}
    \left(k+\frac{N}{2}+a\right)^2\left(k+\frac{N-4}{2}-a\right)^2
    \int_{\mathbb{R}^N}|x|^{-2a-4}|u|^2 \mathrm{d}x.
    \end{align}
    Obviously, our result Theorem \ref{corolehups} is a special case of the second-order CKN inequality \eqref{ckn2}. Now, let us focus on \eqref{ckn2} under condition \eqref{defac}. There are only some partial results about the existence of extremal functions, sharp constants, symmetry or symmetry breaking phenomenon: when $b=a+2$, then \eqref{ckn2} reduces into Rellich type inequality, Caldiroli and Musina \cite{CM12} obtained the explicit form of sharp constant when $b=a+2$ and it is not achieved (see previous results
    \cite{DH98,Mi00} when $-\frac{N}{2}< a<\frac{N-4}{2}$ and $N\geq 3$). While the authors \cite{CM11} proved when $a<b<a+2$ or $a=b$ with energy assumption, the best constant is always achieved in cones, see also
    \cite[Theorem A.2]{MS14}. Furthermore, Szulkin and Waliullah \cite{SW12} proved when $a=b>0$ the best constant is achieved.
    Caldiroli and Cora \cite{CC16} obtained a partial symmetry breaking result when the parameter of pure Rellich term is sufficiently large, we refer to \cite{CM11} in cones, and also \cite{Ya21} with Hardy and Rellich terms. However, as mentioned previous, there are no optimal results about symmetry or symmetry breaking phenomenon.

    Recently, in \cite{DT23-jfa} we obtained the sharp constant and extremal functions of \eqref{ckn2} restricted in radial space, see also \cite{DGT23-jde} jointed work with Grossi. Furthermore, in \cite{DT24}, by showing the equivalent relationship
    \[
    \int_{\mathbb{R}^N} |\mathrm{div}(|x|^\alpha\nabla u)|^2 \mathrm{d}x\sim\int_{\mathbb{R}^N} |x|^{2\alpha}|\Delta u|^2 \mathrm{d}x,\quad \forall u\in C^\infty_0(\mathbb{R}^N),
    \]
    when $N\geq 5$ and $0<\alpha<2$ (this follows from \cite{GG22} which deals with in bounded domain), then we established a second-order weighted inequality
    \[
    \int_{\mathbb{R}^N} |\mathrm{div}(|x|^\alpha\nabla u)|^2 \mathrm{d}x \geq C \left(\int_{\mathbb{R}^N}|u|^{\frac{2N}{N-4+2\alpha}} \mathrm{d}x\right)^{\frac{N-4+2\alpha}{N}},\quad\forall u\in C^\infty_0(\mathbb{R}^N).
    \]
    Moreover, in \cite{DT23-f}, by showing the following equivalent relationship with two variables
    \[
    \int_{\mathbb{R}^N} |x|^{-\beta}|\mathrm{div}(|x|^\alpha\nabla u)|^2 \mathrm{d}x\sim\int_{\mathbb{R}^N} |x|^{2\alpha-\beta}|\Delta u|^2 \mathrm{d}x,\quad \forall u\in C^\infty_0(\mathbb{R}^N),
    \]
    when $N\geq 5$, $\alpha>2-N$ and $\alpha-2\leq \beta\leq \frac{N}{N-2}\alpha$ with also the condition \eqref{defac} by taking $2\alpha-\beta=-2a$ (in fact, we only need require $2\alpha-\beta\neq N+2k$ for all $k\in\mathbb{N}$ since in this case $\frac{2\alpha-\beta}{-2}<\frac{N-4}{2}$), we established a new type second-order CKN inequality
    \begin{equation}\label{cknrs}
    \int_{\mathbb{R}^N}|x|^{-\beta}|\mathrm{div} (|x|^{\alpha}\nabla u)|^2 \mathrm{d}x
    \geq \mathcal{S}_{N,\alpha,\beta}\left(\int_{\mathbb{R}^N}
    |x|^{\beta}|u|^{2^{**}_{\alpha,\beta}} \mathrm{d}x\right)^{\frac{2}{2^{**}_{\alpha,\beta}}},\quad \forall u\in \mathcal{D}^{2,2}_{\alpha,\beta}(\mathbb{R}^N),
    \end{equation}
    where $2^{**}_{\alpha,\beta}:=\frac{2(N+\beta)}{N-4+2\alpha-\beta}$, and $\mathcal{D}^{2,2}_{\alpha,\beta}(\mathbb{R}^N)$ denotes the completion of $C^\infty_0(\mathbb{R}^N)$ with respect to the norm
    \begin{equation*}
    \|u\|_{\mathcal{D}^{2,2}_{\alpha,\beta}(\mathbb{R}^N)}
    =\left(\int_{\mathbb{R}^N}|x|^{-\beta}|\mathrm{div} (|x|^{\alpha}\nabla u)|^2 \mathrm{d}x\right)^{\frac{1}{2}}.
    \end{equation*}
    In fact, \eqref{cknrs} holds for all $N\geq 5$, $\alpha>2-N$ and $\alpha-2\leq \beta\leq \frac{N}{N-2}\alpha$ which does not need the additional condition \eqref{defac} with $2\alpha-\beta=-2a$, and this is different from the work of Caldiroli and Musina
    \cite[Remark 2.3]{CM11}, see Proposition \ref{propgeqd}.

    Classifying solutions to the linearized problem at extremal functions restricted in radial space as Felli and Schneider \cite{FS03}, we obtained a symmetry breaking conclusion in \cite{DT23-f}: no extremal functions of sharp constant $\mathcal{S}_{N,\alpha,\beta}$ in \eqref{cknrs} are radially symmetry if $\alpha>0$ and $\beta_{\mathrm{FS}}(\alpha)<\beta< \frac{N}{N-2}\alpha$, where $\beta_{\mathrm{FS}}(\alpha):=-N+\sqrt{N^2+\alpha^2+2(N-2)\alpha}$. Moreover, we obtained a symmetry result of  extremal functions when $\beta= \frac{N}{N-2}\alpha$ and $\alpha<0$.
    It is worth mentioning that the Euler-Lagrange equation of the new type second-order CKN inequality \eqref{cknrs} is
    \begin{equation*}
    \mathrm{div}(|x|^{\alpha}\nabla(|x|^{-\beta}
    \mathrm{div}(|x|^\alpha\nabla u)))=|x|^\beta|u|^{2^{**}_{\alpha,\beta}-2}u \quad \mbox{in}\quad \mathbb{R}^N,
    \end{equation*}
    which is equivalent to a special weighted Lane-Emden system
    \begin{eqnarray*}
    \left\{ \arraycolsep=1.5pt
       \begin{array}{ll}
        -\mathrm{div}(|x|^\alpha\nabla u)=|x|^{\beta}v\quad \mbox{in}\  \mathbb{R}^N,\\[2mm]
        -\mathrm{div}(|x|^{\alpha}\nabla v)=|x|^\beta|u|^{2^{**}_{\alpha,\beta}-2}u\quad \mbox{in}\  \mathbb{R}^N.
        \end{array}
    \right.
    \end{eqnarray*}
    This is similar to the Euler-Lagrange equation of first-order CKN inequality \eqref{ckn1}, which makes it clear to understand the relationship between the curves $\beta_{\mathrm{FS}}(\alpha)$ and {\em Felli-Schneider curve} $b_{\mathrm{FS}}(a)$ in \eqref{deffsc}. In fact, the curve $\beta_{\mathrm{FS}}(\alpha)$ is nothing else, it is just $b_{\mathrm{FS}}(a)$ by taking $\alpha=-2a$ and $\beta=-b\tau$ with $\tau=\frac{2N}{N-2(1+a-b)}$, moreover, the conditions $a<\frac{N-2}{2}$ and $a\leq b\leq a+1$ imply $\alpha>2-N$ and $\alpha-2\leq \beta\leq \frac{N}{N-2}\alpha$.

    Therefore, it is natural to consider the symmetry region of  extremal functions for \eqref{cknrs}. In fact, with the help of our previous result Theorem \ref{corolehups}, we obtain the symmetry region: $(\alpha,\beta)\in (2-N,0]\times (\alpha-2, \frac{N}{N-2}\alpha]$. The first step is for us to establish the following weighted Rellich-Sobolev inequality, which extends the work of Dan et al. \cite{DMY20} to weighted one.

    \begin{theorem}\label{thmwrsi}
    Assume that $N\geq 5$, $-2<\gamma\leq 0$ and $\gamma<\mu<N-4$. For all $u\in \mathcal{D}^{2,2}_{0,\gamma}(\mathbb{R}^N)$,
    \begin{align}\label{wrsi}
    & \int_{\mathbb{R}^N}\frac{|\Delta u|^2}{|x|^{\gamma}} \mathrm{d}x
    -C_{\gamma,\mu,1}\int_{\mathbb{R}^N}\frac{|\nabla u|^2}{|x|^{\gamma+2}} \mathrm{d}x
    +C_{\gamma,\mu,2}\int_{\mathbb{R}^N}\frac{|u|^2}{|x|^{\gamma+4}} \mathrm{d}x
    \nonumber
    \\
    & \geq \left(\frac{N-4-\mu}{N-4-\gamma}\right)^{3+\frac{2}{2^{**}_{0,\gamma}}}
    \mathcal{S}_{N,\gamma}
    \left(\int_{\mathbb{R}^N}
    |x|^{\gamma}|u|^{2^{**}_{0,\gamma}} \mathrm{d}x\right)^{\frac{2}{2^{**}_{0,\gamma}}},
    \end{align}
    where $2^{**}_{0,\gamma}:=\frac{2(N+\gamma)}{N-4-\gamma}$, $\mathcal{S}_{N,\gamma}$ is given as in \eqref{defsclehi} and
    \begin{align}\label{RSisc}
    C_{\gamma,\mu,1}& =\frac{N^2-4N+8+\gamma^2+4\gamma}{2(N-4-\gamma)^2}(\mu-\gamma)
    [2(N-4-\gamma)-(\mu-\gamma)];
    \nonumber\\
    C_{\gamma,\mu,2}& =\frac{N^2-4N+8+\gamma^2+4\gamma}{8(N-4-\gamma)^2}(\mu-\gamma)^2
    [2(N-4-\gamma)-(\mu-\gamma)]^2
    \nonumber\\
    &\quad-\frac{(\mu-\gamma)^2[2(N-2)-(\mu-\gamma)]^2}{16}
    \nonumber\\
    &\quad-\frac{(\mu-\gamma)[2(N-2)-(\mu-\gamma)]
    (N-4-\mu)(2+\gamma)}{4}.
    \end{align}
    Moreover, $\left(\frac{N-4-\mu}{N-4-\gamma}\right)^{3+\frac{2}{2^{**}_{0,\gamma}}}
    \mathcal{S}_{N,\gamma}$ is sharp and equality in \eqref{wrsi} holds if and only if
    \begin{align*}
    u(x)=A|x|^{-\frac{\mu-\gamma}{2}}
    \left(\lambda+|x|^{\frac{(2+\gamma)(N-4-\mu)}{N-4-\gamma}}
    \right)^{-{\frac{N-4-\gamma}{2+\gamma}}},
    \end{align*}
    for $A\in\mathbb{R}$ and $\lambda>0$.
    \end{theorem}

    \begin{remark}\label{remkspwrsi} \rm
    When $\gamma=0$, \eqref{wrsi} reduces into
    \cite[Theorem 1.6]{DMY20}. The key step of proving Theorem \ref{thmwrsi} is the change of variable
    \begin{align*}
    u(x)=|x|^{\vartheta}v(|x|^{\zeta-1}x),\quad \mbox{with}\quad \vartheta=\frac{\gamma-\mu}{2},\ \zeta=\frac{N-4-\mu}{N-4-\gamma}.
    \end{align*}
    By using standard spherical decomposition then we will prove
    \begin{align*}
    & \int_{\mathbb{R}^N}\frac{|\Delta u|^2}{|x|^{\gamma}} \mathrm{d}x
    -C_{\gamma,\mu,1}\int_{\mathbb{R}^N}\frac{|\nabla u|^2}{|x|^{\gamma+2}} \mathrm{d}x
    +C_{\gamma,\mu,2}\int_{\mathbb{R}^N}\frac{|u|^2}{|x|^{\gamma+4}} \mathrm{d}x
    \geq \zeta^3\int_{\mathbb{R}^N}\frac{|\Delta v|^2}{|x|^{\gamma}} \mathrm{d}x,
    \end{align*}
    and the equality holds if and only if $u$ is radially symmetry. Furthermore,
    \begin{align*}
    \int_{\mathbb{R}^N}
    |x|^{\gamma}|u|^{2^{**}_{0,\gamma}} \mathrm{d}x
    =\zeta^{-1}\int_{\mathbb{R}^N}
    |x|^{\gamma}|v|^{2^{**}_{0,\gamma}} \mathrm{d}x.
    \end{align*}
    Thus the result directly follows from Theorem \ref{corolehups}.
    \end{remark}

    Then inspired by \cite{Ho97,LL17} and also \cite{DELT09}, following the arguments as those in our recent work \cite{DT23-f}, we can establish the sharp second-order CKN type inequality \eqref{cknrs} with radially symmetry extremal functions.

    \begin{theorem}\label{thmcknrs}
    Assume that $N\geq 5$, $2-N<\alpha\leq 0$ and $\alpha-2<\beta\leq \frac{N}{N-2}\alpha$ (except for $\alpha=\beta=0$). The sharp constant $\mathcal{S}_{N,\alpha,\beta}$ in \eqref{cknrs} is defined as
    \begin{equation*}
    \mathcal{S}_{N,\alpha,\beta}:=\inf_{u\in \mathcal{D}^{2,2}_{\alpha,\beta}(\mathbb{R}^N)\setminus\{0\}}
    \frac{\int_{\mathbb{R}^N}|x|^{-\beta}|\mathrm{div} (|x|^{\alpha}\nabla u)|^2 \mathrm{d}x}{\left(\int_{\mathbb{R}^N}
    |x|^{\beta}|u|^{2^{**}_{\alpha,\beta}} \mathrm{d}x\right)^{\frac{2}{2^{**}_{\alpha,\beta}}}},
    \end{equation*}
    with the explicit form
    \begin{align}\label{bccknrs}
    \mathcal{S}_{N,\alpha,\beta}=\left(\frac{2}{2+\beta-\alpha}\right)
    ^{\frac{2(2+\beta-\alpha)}{N+\beta}-4}
    \left(\frac{2\pi^{\frac{N}{2}}}{\Gamma(\frac{N}{2})}\right)
    ^{\frac{2(2+\beta-\alpha)}{N+\beta}}
    \mathcal{B}\left(\frac{2(N+\beta)}{2+\beta-\alpha}\right),
    \end{align}
    where $\mathcal{B}$ is as in Theorem \ref{corolehups}.
    Furthermore, $\mathcal{S}_{N,\alpha,\beta}$ can be achieved if and only if by
    \begin{align*}
    u(x)=A(\lambda+|x|^{2+\beta-\alpha})
    ^{-\frac{N-4+2\alpha-\beta}{2+\beta-\alpha}},
    \end{align*}
    for $A\in\mathbb{R}\setminus\{0\}$ and $\lambda>0$.
    \end{theorem}

    \begin{remark}\label{remkspnckn}\rm
    It is obvious that for $N\geq 5$, $2-N<\alpha\leq 0$ and $\alpha-2<\beta\leq \frac{N}{N-2}\alpha$, condition \eqref{defac} holds by taking $2\alpha-\beta=-2a$. In fact, we will show the second-order Caffarelli-Kohn-Nirenberg type inequality \eqref{cknrs} holds which does not need the additional condition \eqref{defac} with $2\alpha-\beta=-2a$, see Proposition \ref{propgeqd}.

    When $\alpha=\beta=0$, \eqref{bccknrs} reduces into the classical second-order Sobolev inequality (see \cite{EFJ90,Li98,Li85-1,Va93}) and extremal functions with also translation invariance. When $\alpha<0$ and $\beta=\frac{N}{N-2}\alpha$, \eqref{bccknrs} reduces into our recent work \cite[Theorem 1.6]{DT23-f}.
    The key step of proving Theorem \ref{thmcknrs} is the change of variable
    \begin{align*}
    u(x)=|x|^{\eta}v(x) \quad\mbox{with}\quad
    \eta=-\frac{\alpha (N-4+2\alpha-\beta)}{2(N-2+\alpha)}.
    \end{align*}
    The choice of $\eta$ is ensuring  $\beta-2(\alpha+\eta)=\beta+2^{**}_{\alpha,\beta}\cdot\eta$. Then $2^{**}_{\alpha,\beta}=2^{**}_{0,\gamma}$, where $\gamma =\beta-2(\alpha+\eta)=\frac{\beta(N-2)-\alpha N}{N-2+\alpha}\in (-2,0]$.
    By using standard spherical decomposition we prove
    \begin{align*}
    \int_{\mathbb{R}^N}|x|^{-\beta}|\mathrm{div} (|x|^{\alpha}\nabla u)|^2 \mathrm{d}x
    \geq \int_{\mathbb{R}^N}\frac{|\Delta v|^2}{|x|^{\gamma}} \mathrm{d}x
    -C_{\gamma,\mu,1}\int_{\mathbb{R}^N}\frac{|\nabla v|^2}{|x|^{\gamma+2}} \mathrm{d}x
    +C_{\gamma,\mu,2}\int_{\mathbb{R}^N}\frac{|v|^2}{|x|^{\gamma+4}} \mathrm{d}x,
    \end{align*}
    and the equality holds if and only if $u$ is radially symmetry (the assumption $\alpha\leq 0$ will play a crucial role), where $\mu=\frac{N-4-\gamma}{2-N}\alpha+\gamma$. Furthermore,
    \begin{align*}
    \int_{\mathbb{R}^N}
    |x|^{\beta}|u|^{2^{**}_{\alpha,\beta}} \mathrm{d}x
    =\int_{\mathbb{R}^N}
    |x|^{\gamma}|v|^{2^{**}_{0,\gamma}} \mathrm{d}x.
    \end{align*}
    Then the result directly follows from Theorem \ref{thmwrsi}. The purpose of selecting such $\eta$ is to ensure that $\beta+\eta\cdot 2^{**}_{\alpha,\beta}= \gamma=-(2\alpha-\beta+2\eta)$.

    In fact, we can also obtain the result by using the direct change of variable,
    \begin{align}\label{dcv}
    u(x)=|x|^{\eta+\vartheta}w(|x|^{\zeta-1}x),
    \end{align}
    where $\eta$ is given as previous, and $\vartheta, \zeta$ are given in Remark \ref{remkspwrsi}. It is easy to verify that $\eta+\vartheta=0$ and $\zeta=1+\frac{\alpha}{N-2}$. Thus, \eqref{dcv} is
    \[
    u(x)=w(|x|^{\frac{\alpha}{N-2}}x),
    \]
    then the previous statements and Remark \ref{remkspwrsi} indicate
    \begin{align*}
    \int_{\mathbb{R}^N}|x|^{-\beta}|\mathrm{div} (|x|^{\alpha}\nabla u)|^2 \mathrm{d}x
    \geq \zeta^3\int_{\mathbb{R}^N}\frac{|\Delta w|^2}{|x|^{\gamma}} \mathrm{d}x,
    \end{align*}
    and the equality holds if and only if $u$ is radially symmetry. Furthermore,
    \begin{align*}
    \int_{\mathbb{R}^N}
    |x|^{\beta}|u|^{2^{**}_{\alpha,\beta}} \mathrm{d}x
    =\zeta^{-1}\int_{\mathbb{R}^N}
    |x|^{\gamma}|w|^{2^{**}_{0,\gamma}} \mathrm{d}x.
    \end{align*}
    Thus, the result also follows from Theorem \ref{corolehups} when $-2<\gamma<0$ and \cite{EFJ90,Li85-1,Va93} when $\gamma=0$. Indeed, this transformation method is similar to \cite{Ho97,LL17} as in \eqref{fockncv}.
    \end{remark}

    Furthermore, comparing our symmetry result of Theorem \ref{thmcknrs} and symmetry breaking result in \cite{DT23-f} to the ones in \cite{DEL16,FS03} shown as in Figure \ref{F1}, we give the following conjecture.

\vskip0.25cm

    {\em {\bf Conjecture:} the extremal functions in \eqref{cknrs} are also radially symmetry if $\alpha>0$ and $\alpha-2< \beta\leq \beta_{\mathrm{FS}}(\alpha)$}. See Figure \ref{F2}.

    \vskip0.25cm

    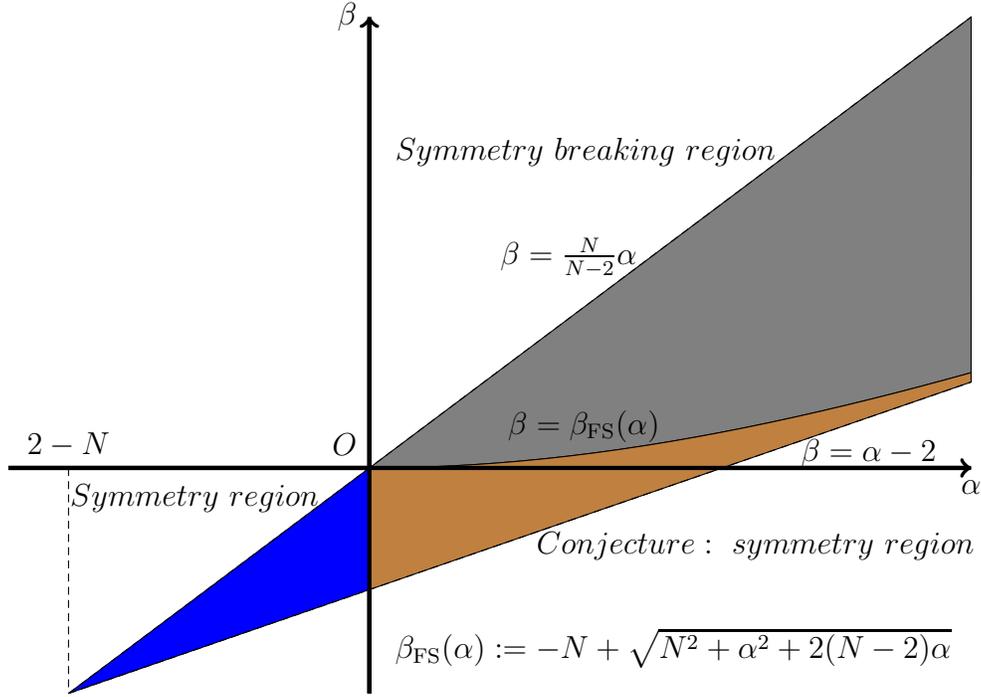
\begin{figure}[ht]
    \begin{tikzpicture}[scale=4]
		\draw[->,ultra thick](-1.2,0)--(0,0)node[above left]{$O$}--(2,0)node[below]{$\alpha$};
		\draw[->,ultra thick](0,-0.75)--(0,1.5)node[left]{$\beta$};
\draw[fill=gray,domain=0:2]plot(\x,{-1+((3/7*\x)^2+1^2)^0.5})
--(2,1.5)--(0,0);
\draw[fill=blue,domain=-1:1](-1,-0.75)--(0,0)--(0,-34/84)--(-1,-0.75);
\draw[fill=brown,domain=0:2]plot(\x,{-1+((3/7*\x)^2+1^2)^0.5})
--(2,2/7)--(34/29,0);
\draw[fill=brown,domain=-1:0](0,0)--(34/29,0)--(0,-34/84)--(0,0);
        \draw[densely dashed](-1,-3/4)--(2,1.5);
        \draw[densely dashed](-1,-3/4)--(2,2/7);
        \draw[densely dashed](-1,-0)node[above]{$2-N$}--(-1,-0.75);
        \draw[-,ultra thick](-1.2,0)--(2,0);
        \draw[-,ultra thick](0,-0.75)--(0,1.5);
		\node[left] at(1.92,0.05){$\beta=\alpha-2$};
        \node[right] at (0.4,0.7){$\beta=
        \frac{N}{N-2}\alpha$};
        \node[left] at (1,0.14){$\beta=\beta_{\mathrm{FS}}(\alpha)$};
        \node[right] at (0.05,-0.6){$\beta_{\mathrm{FS}}(\alpha):=
        -N+\sqrt{N^2+\alpha^2+2(N-2)\alpha}$};
        \node[right] at (0.05,1.05){$Symmetry\ breaking\ region$};
        \node[right] at (-1.03,-0.1){$Symmetry\ region$};
        \node[right] at (0.52,-0.26){$Conjecture:\ symmetry\ region$};
\end{tikzpicture}
\caption{\small On extremal functions of the new second-order CKN inequality \eqref{cknrs}. Symmetry breaking region appears in dark grey defined by $\alpha>0$ and $\beta_{\mathrm{FS}}(\alpha)<\beta<\frac{N}{N-2}\alpha$. We have proved that the symmetry holds in the blue region defined by $2-N<\alpha\leq 0$ and $\alpha-2< \beta\leq \frac{N}{N-2}\alpha$, and we conjecture that the symmetry also holds in the brown region defined by $\alpha>0$ and $\alpha-2< \beta\leq \beta_{\mathrm{FS}}(\alpha)$.}
\label{F2}
\end{figure}

    Next, we give the symmetry result of a singular second-order CKN type inequality which was established in our recent work \cite{DT24-2}: when $N\geq 5$, $\alpha>2-N$ and $\frac{N-4}{N-2}\alpha-4 \leq \beta\leq \alpha -2$ with also the condition \eqref{defac} by taking $2\alpha-\beta=-2a$, then for all $u\in \mathcal{D}^{2,2}_{\alpha,\beta}(\mathbb{R}^N\setminus\{0\})$,
    \begin{align}\label{cknrss}
    \int_{\mathbb{R}^N}|x|^{-\beta}|\mathrm{div} (|x|^{\alpha}\nabla u)|^2 \mathrm{d}x
    \geq \overline{\mathcal{S}}_{N,\alpha,\beta}
    \left(\int_{\mathbb{R}^N}|x|^{\xi}
    |u|^{\bar{2}^{**}_{\alpha,\beta}} \mathrm{d}x\right)^{\frac{2}{\bar{2}^{**}_{\alpha,\beta}}},
    \end{align}
    for some constant $\overline{\mathcal{S}}_{N,\alpha,\beta}>0$, where
    \begin{align*}
    \bar{2}^{**}_{\alpha,\beta}:=\frac{2(N+\xi)}{N+2\alpha-\beta-4}
    \quad \mbox{with}\quad (N+\beta)(N+\xi)=(N+2\alpha-\beta-4)^2.
    \end{align*}
    In fact, we will establish the equivalence relationship between \eqref{cknrs} and \eqref{cknrss} shown as in Lemma \ref{lemertwo}, thus from Proposition \ref{propgeqd} we know that \eqref{cknrss} also holds without the additional condition \eqref{defac} by taking $2\alpha-\beta=-2a$. As an affiliated result of Theorem \ref{thmcknrs}, we have

    \begin{theorem}\label{thmcknrss}
    Assume that $N\geq 5$, $2-N<\alpha\leq 0$ and $\frac{N-4}{N-2}\alpha-4 \leq \beta<\alpha -2$ (except for $\alpha=0$ and $\beta=-4$). The sharp constant $\overline{\mathcal{S}}_{N,\alpha,\beta}$ in \eqref{cknrss} is defined as
    \begin{align*}
    \overline{\mathcal{S}}_{N,\alpha,\beta}:=
    \inf_{u\in \mathcal{D}^{2,2}_{\alpha,\beta}(\mathbb{R}^N\setminus\{0\})
    \setminus\{0\}}
    \frac{\int_{\mathbb{R}^N}|x|^{-\beta}|\mathrm{div} (|x|^{\alpha}\nabla u)|^2 \mathrm{d}x}
    {\left(\int_{\mathbb{R}^N}|x|^{\xi}
    |u|^{\bar{2}^{**}_{\alpha,\beta}} \mathrm{d}x\right)^{\frac{2}{\bar{2}^{**}_{\alpha,\beta}}}},
    \end{align*}
    with the explicit form
    \begin{equation}\label{defsr}
    \overline{\mathcal{S}}_{N,\alpha,\beta}
    =\left(\frac{2}{\alpha-\beta-2}\right)
    ^{\frac{2(\alpha-\beta-2)}{N+2\alpha-\beta-4}-4}
    \left(\frac{2\pi^{\frac{N}{2}}}{\Gamma(\frac{N}{2})}\right)
    ^{\frac{2(\alpha-\beta-2)}{N+2\alpha-\beta-4}}
    \mathcal{B}\left(\frac{2(N+2\alpha-\beta-4)}{\alpha-\beta-2}\right),
    \end{equation}
    where $\mathcal{B}$ is given as in Theorem \ref{corolehups}.
    Furthermore, $\overline{\mathcal{S}}_{N,\alpha,\beta}$ can be achieved if and only if by
    \begin{align*}
    u(x)=A|x|^{2+\beta-\alpha}(1+|x|^{\alpha-\beta-2})
    ^{-\frac{N+\beta}{\alpha-\beta-2}},
    \end{align*}
    for $A\in\mathbb{R}\setminus\{0\}$ and $\lambda>0$.
    \end{theorem}

    \begin{remark}\label{remefrnr}\rm
    When $\alpha=0$ and $\beta=-4$, \eqref{cknrss} reduces into classical second-order Sobolev inequality by taking $u(x)=|x|^2v(x)$ (see \cite{DT24-2}), so the extremal functions are also ``translation invariant".

    Furthermore, let us give some comments about the reachability and non-reachability of \eqref{cknrss}. In fact, we proved in \cite{DT24-2} that when $\beta=\alpha-2$, or $\beta=\frac{N-4}{N-2}\alpha-4$ and $\alpha>0$, the singular second-order CKN type inequality \eqref{cknrss} is strict for any nonzero functions, while for the other cases equality can always be achieved. Therefore, by the equivalence relationship between \eqref{cknrs} and \eqref{cknrss} (see Lemma \ref{lemertwo}), we have that, when $\beta=\alpha-2$, or $\beta=\frac{N}{N-2}\alpha$ and $\alpha>0$, the second-order CKN type inequality \eqref{cknrs} is also strict for any nonzero functions, while for the other cases equality can always be achieved.
    \end{remark}

\subsection{Hardy-Rellich inequality}\label{subsecthri}

    Finally, let us recall the following inequality:
    \begin{align}\label{hri}
    \int_{\mathbb{R}^N}|x|^{-2a}|\Delta u|^2 \mathrm{d}x
    \geq C_*
    \int_{\mathbb{R}^N}|x|^{-2a-2}|\nabla u|^2 \mathrm{d}x,\quad \forall u\in C^\infty_0(\mathbb{R}^N),
    \end{align}
    for some $C_*\geq 0$, where $-\infty<a<\frac{N-2}{2}$. If $a$ satisfies \eqref{defac}, then \eqref{hri} is equivalent to the second-order CKN inequality \eqref{cknh} and $C_*>0$. \eqref{hri} is usually named as the weighted Hardy-Rellich inequality, thanks to Tertikas and Zographopoulos \cite{TZ07} studied this problem and when $N\geq 5$ and $0\leq a<\frac{N-4}{2}$, the authors obtained the sharp constant with the explicit form
    \begin{align*}
    C_*=\inf_{k\in\mathbb{N}}\frac{\left(\frac{(N-4-2a)(N+2a)}{4}
    +k(N-2+k)\right)^2}{(\frac{N-4-2a}{2})^2+k(N-2+k)}.
    \end{align*}
    Furthermore, Ghoussoub and Moradifam in \cite[Theorem 6.1]{GM11} studied the general case $N\geq 1$ and $a\leq \frac{N-2}{2}$ under different conditions for $a$, see also \cite[Corollary 5.9]{DM22} for $-\frac{N}{2}<a\leq \frac{N-4}{2}$. It is obvious that when $N\geq 5$ and $a=0$, Tertikas and Zographopoulos \cite{TZ07} obtained the sharp constant $C_*=\frac{N^2}{4}$, while in the cases $N=3$ or $4$ the sharp constant does not satisfies the general formula. In fact,  Cazacu \cite{Ca20} proved the sharp constant $C_*$ is strictly less than $\frac{N^2}{4}$ when $N=3$ or $4$, that is, $C_*=3$ when $N=4$ and $C_*=\frac{25}{36}$ when $N=3$. In present paper, we are concerned about a weak form of \eqref{hri} which reads that:

    \begin{theorem}\label{thmhriw}
    Let $N\geq 1$, then for each $-\infty<a<\frac{N-2}{2}$, it holds that
    \begin{align}\label{hriw}
    \int_{\mathbb{R}^N}|x|^{-2a}|\Delta u|^2 \mathrm{d}x
    \geq C(N,a)
    \int_{\mathbb{R}^N}|x|^{-2a-4}(x\cdot \nabla u)^2 \mathrm{d}x,\quad \forall u\in C^\infty_0(\mathbb{R}^N),
    \end{align}
    with sharp constant
    \begin{align*}
    C(N,a)=\left(\frac{N+2a}{2}\right)^2
    \end{align*}
    if $\frac{-2-\sqrt{N^2-2N+2}}{2}\leq a\leq \frac{-2+\sqrt{N^2-2N+2}}{2}$ for $N\geq 2$, or $\frac{-2-\sqrt{5}}{2}\leq a< -\frac{1}{2}$ for $N=1$, and
    \begin{align*}
    C(N,a)=\left(\frac{2}{N-4-2a}\right)^2\inf\limits_{k\in\mathbb{N}}
    \left(k+\frac{N}{2}+a\right)^2\left(k+\frac{N-4}{2}-a\right)^2
    \end{align*}
    if $a<\frac{-2-\sqrt{N^2-2N+2}}{2}$ or $\frac{-2+\sqrt{N^2-2N+2}}{2}<a<\frac{N-2}{2}$ for $N\geq 2$, or $a<\frac{-2-\sqrt{5}}{2}$ for $N=1$.
    In particular, for each $N\geq 3$,
    \begin{align}\label{hriw0}
    \int_{\mathbb{R}^N}|\Delta u|^2 \mathrm{d}x
    \geq \frac{N^2}{4}
    \int_{\mathbb{R}^N}|x|^{-4}(x\cdot \nabla u)^2 \mathrm{d}x,\quad \forall u\in C^\infty_0(\mathbb{R}^N),
    \end{align}
    with sharp constant $\frac{N^2}{4}$.
    \end{theorem}

    It is not difficult to verify (see \eqref{defcnas2} in Remark \ref{remwfhri} directly) that our result \eqref{hriw} implies the weighted Rellich inequality \eqref{wri} of the work of Caldiroli and Musina \cite{CM12}, by using the Hardy type inequality:
    \begin{align*}
    \int_{\mathbb{R}^N}|x|^{-2a-4}(x\cdot \nabla u)^2 \mathrm{d}x
    \geq \left(\frac{N-4-2a}{2}\right)^2
    \int_{\mathbb{R}^N}|x|^{-2a-4}|u|^2 \mathrm{d}x,\quad \forall u\in C^\infty_0(\mathbb{R}^N).
    \end{align*}
    Moreover, when $a=\frac{N-4}{2}$, Ghoussoub and Moradifam in \cite[Theorem 6.1]{GM11} proved the sharp constant in \eqref{hri} is $C_*=\min\{(N-2)^2,N-1\}$, then our conclusion about \eqref{hriw} indicates $C\left(N,\frac{N-4}{2}\right)=(N-2)^2>C_*$ if $N=1$ or $N\geq 4$, while $C\left(N,\frac{N-4}{2}\right)=C_*$ if $N=2$ or $3$.
    Furthermore, our weak form Hardy-Rellich inequality \eqref{hriw0} admits consistent formula for sharp constant when $N=3$ or $4$, and $C\left(N,0\right)=\frac{N^2}{4}>C_*$, which shows a different phenomenon comparing with Cazacu \cite{Ca20}.

    \begin{remark}\label{remer}\rm
    From the sharp weak form weighted Hardy-Rellich inequality \eqref{hriw}, it is not difficult to verify that for $N\geq 5$ and $-N<\alpha-2\leq \beta\leq \frac{N}{N-2}\alpha$, the norms $\int_{\mathbb{R}^N} |x|^{-\beta}|\mathrm{div}(|x|^\alpha\nabla u)|^2 \mathrm{d}x$ and $\int_{\mathbb{R}^N} |x|^{2\alpha-\beta}|\Delta u|^2 \mathrm{d}x$ are equivalent if and only if \eqref{defac} holds with $2\alpha-\beta=-2a$. In fact, for all $u\in C^\infty_0(\mathbb{R}^N)$,
    \begin{align*}
    \int_{\mathbb{R}^N}|x|^{-\beta}|\mathrm{div} (|x|^{\alpha}\nabla u)|^2\mathrm{d}x
    = & \int_{\mathbb{R}^N}|x|^{2\alpha-\beta}|\Delta u|^2\mathrm{d}x
    + \alpha (N-4+2\alpha-\beta)\int_{\mathbb{R}^N}
    |x|^{2\alpha-2-\beta}|\nabla u|^2\mathrm{d}x
    \nonumber \\
    & + \alpha(2\beta-3\alpha+4) \int_{\mathbb{R}^N}|x|^{2\alpha-4-\beta}(x\cdot\nabla u)^2\mathrm{d}x,
    \end{align*}
    see \eqref{ckn2nfe}, then
    \begin{align*}
    \int_{\mathbb{R}^N}|x|^{-\beta}|\mathrm{div} (|x|^{\alpha}\nabla u)|^2\mathrm{d}x
    \leq & \int_{\mathbb{R}^N}|x|^{2\alpha-\beta}|\Delta u|^2\mathrm{d}x
    + \alpha (N-\alpha+\beta)\int_{\mathbb{R}^N}
    |x|^{2\alpha-4-\beta}(x\cdot \nabla u)^2\mathrm{d}x
    \end{align*}
    if $\alpha<0$, and
    \begin{align*}
    \int_{\mathbb{R}^N}|x|^{-\beta}|\mathrm{div} (|x|^{\alpha}\nabla u)|^2\mathrm{d}x
    \geq & \int_{\mathbb{R}^N}|x|^{2\alpha-\beta}|\Delta u|^2\mathrm{d}x
    + \alpha (N-\alpha+\beta)\int_{\mathbb{R}^N}
    |x|^{2\alpha-4-\beta}(x\cdot \nabla u)^2\mathrm{d}x
    \end{align*}
    if $\alpha>0$. Note that $\frac{2\alpha-\beta}{-2}<\frac{N-4}{2}
    <\frac{-2+\sqrt{N^2-2N+2}}{2}$, then we can split it into two cases: $\frac{-2-\sqrt{N^2-2N+2}}{2}\leq \frac{2\alpha-\beta}{-2}<\frac{N-4}{2}$ and $\frac{2\alpha-\beta}{-2}<\frac{-2-\sqrt{N^2-2N+2}}{2}$, thus from Theorem \ref{thmhriw} we can deduce the conclusion holds. In fact, from \eqref{defcnas2} in Remark \ref{remwfhri}, the conclusion holds directly.
    \end{remark}

\subsection{Structure of the paper}\label{subsect:structrue}
    The paper is organized as follows: in Section \ref{sectss}, we prove the solutions of \eqref{lehe} are unique thus can be classified, then we establish two cases of sharp Stein-Weiss inequalities \eqref{swi} with explicit form of extremal functions when $t=r$ and $t=2$ separately, and also establish a standard sharp second-order CKN inequality \eqref{lehi} with radially symmetry extremal functions. Section \ref{sectswrsi} is devoted to establishing the sharp weighted Rellich-Sobolev type inequality \eqref{wrsi} by using standard spherical decomposition, and we also establish a sharp weighted Hardy-Rellich inequality \eqref{shri} by taking $\mu\uparrow N-4$ in \eqref{wrsi}. In Section \ref{sectcknps}, taking the transformation $u(x)=|x|^{\eta}v(x)$ with suitable $\eta$, we deduce the sharp CKN type inequality \eqref{thmcknrs} with radially symmetry extremal functions by using the result of Theorem \ref{thmwrsi}. Then in Section \ref{sectcknpss}, we establish an equivalent relationship of two second-order CKN type inequalities, then we obtain a sharp singular CKN type inequality \eqref{cknrss} also with radially symmetry extremal functions. Finally, we derive the sharp weak form weighted Hardy-Rellich inequality \eqref{hriw} by standard spherical decomposition and give the proof of Theorem \ref{thmhriw}.

\section{{\bfseries Symmetry of solutions}}\label{sectss}

    In this section, we will classify all solutions of equation \eqref{lehe}. Let us recall a uniqueness result established by Chen and Li \cite{CL08} for the following weighted system:
    \begin{eqnarray}\label{wlescl}
    \left\{ \arraycolsep=1.5pt
       \begin{array}{ll}
        -\Delta u=\frac{u^p}{|x|^{b(p+1)}}\quad \mbox{in}\quad  \mathbb{R}^N,\\[2mm]
        -\Delta v=\frac{v^q}{|x|^{a(q+1)}}\quad \mbox{in}\quad  \mathbb{R}^N,\\[2mm]
        u,v\geq 0\quad\mbox{in}\quad \mathbb{R}^N,
        \end{array}
    \right.
    \end{eqnarray}
    where $0<p,q<\infty$, $a, b\geq 0$, $\frac{a}{N}<\frac{1}{q+1}<\frac{N-2+a}{N}$, $\frac{1}{p+1}+\frac{1}{q+1}=\frac{N-2+a+b}{N}$, and
    \begin{align*}
    b(p+1)<2,\quad a(q+1)<2,
    \end{align*}
    which states that the solutions of the system \eqref{wlescl} are unique in the sense that if $(u_1,v_1)$ and $(u_2,v_2)$ are any two pairs of solutions with $u_1(0)=u_2(0)$, then $v_1(0)=v_2(0)$ and hence $(u_1(x),v_1(x))\equiv (u_2(x),v_2(x))$.

\vskip0.25cm

\noindent{\bf \em Proof of Theorem \ref{thmlehups}}.
    Note that equation \eqref{lehe} is equivalent to the following Hardy-Lane-Emden system
    \begin{eqnarray}\label{lehs}
    \left\{ \arraycolsep=1.5pt
       \begin{array}{ll}
        -\Delta u=|x|^\gamma v\quad \mbox{in}\quad  \mathbb{R}^N,\\[2mm]
        -\Delta v=|x|^\gamma u^q\quad \mbox{in}\quad  \mathbb{R}^N,\\[2mm]
        u\geq 0\quad\mbox{in}\quad \mathbb{R}^N.
        \end{array}
    \right.
    \end{eqnarray}
    By using the maximum principle, from $-\Delta v=|x|^\gamma u^q$ and $u\geq 0$ we know $v$ is also nonnegative. Then let us consider system \eqref{wlescl} with
    \[
    p=1,\quad b(p+1)=a(q+1),
    \]
    then $a=\frac{b(N-4+2b)}{N-2b}$ and $0\leq b<1$. Let $b=-\frac{\gamma}{2}\in (0,1)$, then we obtain the solutions of the system \eqref{lehs} are also unique, therefore equation \eqref{lehe} admits a unique solution in the sense of Chen and Li \cite{CL08}.

    Now, let us show the solution of equation \eqref{lehe} must be the form $U_\lambda$ as in \eqref{defula}. On the one hand, it is obvious that $U_\lambda$ is a solution of \eqref{lehe} for each $\lambda\geq 0$, see our recent work \cite{DT23-jfa}. On the other hand, assume that $u(x)$ is a solution of \eqref{lehe}, then we can always choose nonnegative $\lambda$ such that
    \[
    u(0)=U_\lambda(0),
    \]
    thus by the uniqueness result of Chen and Li \cite{CL08}, we must have $u(x)\equiv U_\lambda(x)$. This completes the proof of the theorem.
    \qed

\vskip0.25cm

\noindent{\bf \em Proof of Theorem \ref{thmswia}}. Let us consider system \eqref{sweletet} when $\lambda=N-2$. From the uniqueness result of Chen and Li \cite[Theorem 2]{CL08} where $a=b=\frac{N}{p+1}-\frac{N-2}{2}$ and $p=q$ with $b(p+1)<2$, we can directly deduce the sharp Stein-Weiss inequality \eqref{swi} when $t=r=\frac{2N}{N+2-a-b}$, that is, there is equality in \eqref{swi} with sharp constant if and only if $f\equiv (const.) g$ and
    \begin{align*}
    & f(x)=A|x|^{-\frac{b(N+2(1-b))}{N-2(1-b)}}
    \left(c+|x|^{\frac{2(N-2)(1-b)}{N-2(1-b)}}
    \right)^{-\frac{N+2(1-b)}{2(1-b)}},
    \end{align*}
    for $A\in \mathbb{R}$, $c>0$.

Furthermore, from the uniqueness result of Theorem \ref{thmlehups} where $p=1$ and $b(p+1)=a(q+1)<2$, we can also directly deduce the sharp Stein-Weiss inequality \eqref{swi} when $t=2$ with $a=\frac{b(N-4+2b)}{N-2b}$ and $0<b<1$ which implies $r=\frac{N+4-6b}{2(N-2b)}$, that is, there is equality in \eqref{swi} with sharp constant if and only if
    \begin{align*}
    f(x) & =A|x|^{-\frac{b(N+4-2b)}{N-2b}}
    (c+|x|^{2-2b})^{-\frac{N+4-2b}{2-2b}},\\
    g(x) & =-B|x|^{b}\Delta \left((c+|x|^{2-2b})
    ^{-\frac{N-4+2b}{2-2b}}\right)
    =B'|x|^{-b}\int_{\mathbb{R}^N}
    \frac{(c+|y|^{2-2b})^{-\frac{N+4-2b}{2-2b}}}
    {|y|^{\frac{2Nb}{N-2b}}|x-y|^{N-2}}\mathrm{d}y\\
    & =B'\left[(N-2b)c+(2-2b)|x|^{2-2b}\right]|x|^{-b}
    (c+|x|^{2-2b})
    ^{-\frac{N-2b}{2-2b}},
    \end{align*}
    for $A, B\in \mathbb{R}$, $c>0$. Putting the extremal functions into Stein-Weiss inequality\eqref{swi}, we can derive the explicit form of sharp constants. Now, the proof of Theorem \ref{thmswia} is completed.
\qed

\vskip0.25cm

\noindent{\bf \em Proof of Theorem \ref{corolehups}}. From \cite[Theorem A.2 (i)]{MS14} (see also \cite[Theorem 1.2]{CM11}), we know equality \eqref{lehi} with sharp constant holds for some nontrivial functions. It is well known that extremal functions of \eqref{lehi} are always solutions (multiplies some suitable constant) of the following related Euler-Lagrange equation:
    \begin{equation}\label{lehele}
    \Delta(|x|^{-\gamma}\Delta u)=|x|^\gamma |u|^{2^{**}_{0,\gamma}-2}u,\quad u\in \mathcal{D}^{2,2}_{0,\gamma}(\mathbb{R}^N).
    \end{equation}
    Furthermore, nontrivial extremal functions of \eqref{lehi} are always least energy solutions (multiplies some suitable constant) of equation \eqref{lehele}.
    In fact, the energy functional of \eqref{lehele} is defined as
    \begin{align}\label{deffe}
    \Phi(u)=\frac{1}{2}\int_{\mathbb{R}^N}\frac{|\Delta u|^2}{|x|^{\gamma}} \mathrm{d}x
    -\frac{1}{2^{**}_{0,\gamma}}\int_{\mathbb{R}^N}
    |x|^{\gamma}|u|^{2^{**}_{0,\gamma}} \mathrm{d}x,
    \quad u\in \mathcal{D}^{2,2}_{0,\gamma}(\mathbb{R}^N).
    \end{align}
    Define the Nehari manifold
    \begin{align*}
    \mathcal{N}:=\left\{u\in \mathcal{D}^{2,2}_{0,\gamma}(\mathbb{R}^N) \backslash\{0\}: \langle \Phi'(u),u\rangle=0\right\}.
    \end{align*}
    We know that the least energy of \eqref{lehele} is defined as $\inf\{\Phi(u):\ u\in \mathcal{D}^{2,2}_{0,\gamma}(\mathbb{R}^N) \backslash\{0\},\ \Phi'(u)=0\}$ which is large than $\inf\limits_{u\in \mathcal{N}}\Phi(u)$.
    For $v\in \mathcal{N}$ we have
    \begin{align*}
    \Phi(v)
    & =\left(\frac{1}{2}-\frac{1}{2^{**}_{0,\gamma}}\right)
    \int_{\mathbb{R}^N}\frac{|\Delta v|^2}{|x|^{\gamma}} \mathrm{d}x
    = \left(\frac{1}{2}-\frac{1}{2^{**}_{0,\gamma}}\right)
    \left(\frac{\int_{\mathbb{R}^N}\frac{|\Delta v|^2}{|x|^{\gamma}} \mathrm{d}x}
    {\left(\int_{\mathbb{R}^N}
    |x|^{\gamma}|v|^{2^{**}_{0,\gamma}} \mathrm{d}x\right)^{\frac{2}{2^{**}_{0,\gamma}}}}
    \right)^{\frac{2^{**}_{0,\gamma}}{2^{**}_{0,\gamma}-2}}
    \\
    & \geq\left(\frac{1}{2}-\frac{1}{2^{**}_{0,\gamma}}\right)
    \mathcal{S}_{N,\gamma}^{\frac{2^{**}_{0,\gamma}}{2^{**}_{0,\gamma}-2}}
    = \Phi(cu),
    \end{align*}
    for any minimizer $u$ of \eqref{lehi} with some suitable constant $c$ such that $cu$ satisfies \eqref{lehele}, therefore $cu$ is also the least energy solution.

    Next, we show that the least energy solutions of \eqref{lehele} can not change sign, which implies the extremal functions of \eqref{lehi} also can not change sign. If not, we suppose that $w$ is a least energy solution of \eqref{lehele} which is changing sign. Define
    \[
    w^+(x)=\max\{w(x),0\},\quad w^-(x)=\min\{w(x),0\}.
    \]
    Note that $w^+(x)$ and $w^-(x)$ are also nontrivial solutions of \eqref{lehele}. Therefore,
    \begin{align*}
    \left(\frac{1}{2}-\frac{1}{2^{**}_{0,\gamma}}\right)
    \mathcal{S}_{N,\gamma}^{\frac{2^{**}_{0,\gamma}}{2^{**}_{0,\gamma}-2}}
    & =\Phi(w)=\Phi(w^+)+\Phi(w^-)\geq 2\left(\frac{1}{2}-\frac{1}{2^{**}_{0,\gamma}}\right)
    \mathcal{S}_{N,\gamma}^{\frac{2^{**}_{0,\gamma}}{2^{**}_{0,\gamma}-2}},
    \end{align*}
    which leads to a contradiction.

    Then from the uniqueness result of Theorem \ref{thmlehups}, we get the extremal functions of \eqref{lehi} must be the form $cU_\lambda(x)$
    for $c\in\mathbb{R}$ and $\lambda>0$, where $U_\lambda(x)=\lambda^{\frac{N-4-\gamma}{2}}U(\lambda x)$ is given in \eqref{defula}. Therefore, putting $U$ into \eqref{lehi} as a test function, we can directly obtain the explicit form of sharp constant as
    \begin{align*}
    \mathcal{S}_{N,\gamma}=\left(\frac{2}{2+\gamma}\right)
    ^{\frac{2(2+\gamma)}{N+\gamma}-4}
    \left(\frac{2\pi^{\frac{N}{2}}}{\Gamma(\frac{N}{2})}\right)
    ^{\frac{2(2+\gamma)}{N+\gamma}}
    \mathcal{B}\left(\frac{2(N+\gamma)}{2+\gamma}\right),
    \end{align*}
    where $\mathcal{B}(M)=(M-4)(M-2)M(M+2)
    \left[\Gamma^2(\frac{M}{2})/(2\Gamma(M))\right]^{\frac{4}{M}}$, and $\Gamma$ is the Gamma function.
    Now, the proof of Theorem \ref{corolehups} is completed.
    \qed

\section{{\bfseries Sharp weighted Rellich-Sobolev inequality}}\label{sectswrsi}

Let us recall a crucial Rellich-Sobolev type inequality with explicit form of extremal functions which was established by Dan, Ma and Yang \cite{DMY20} as the following.

    \begin{theorem}\label{thmrsi}
    Let $N\geq 5$ and $0<\mu<N-4$. For all $u\in \mathcal{D}^{2,2}_0(\mathbb{R}^N)=\{u\in C^\infty_0(\mathbb{R}^N): \int_{\mathbb{R}^N}|\Delta u|^2 \mathrm{d}x<\infty,\ \int_{\mathbb{R}^N}|u|^{\frac{2N}{N-4}} \mathrm{d}x<\infty\}$,
    \begin{align*}
    & \int_{\mathbb{R}^N}|\Delta u|^2 \mathrm{d}x
    -C_{\mu,1}\int_{\mathbb{R}^N}\frac{|\nabla u|^2}{|x|^2} \mathrm{d}x
    +C_{\mu,2}\int_{\mathbb{R}^N}\frac{|u|^2}{|x|^4} \mathrm{d}x
    \\
    & \geq \left(1-\frac{\mu}{N-4}\right)^{4-\frac{4}{N}}\mathcal{S}_0
    \left(\int_{\mathbb{R}^N}|u|^{\frac{2N}{N-4}} \mathrm{d}x\right)^\frac{N-4}{N},
    \end{align*}
    where $\mathcal{S}_0$ is the sharp constant of classical second-order Sobolev inequality, and
    \begin{align*}
    C_{\mu,1}:& =\frac{N^2-4N+8}{2(N-4)^2}\mu[2(N-4)-\mu];
    \\
    C_{\mu,2}:& =\frac{N^2}{16(N-4)^2}\mu^2[2(N-4)-\mu]^2
    -\frac{N-2}{2}\mu[2(N-4)-\mu].
    \end{align*}
    Moreover, equality holds if and only if
    \begin{align*}
    u(x)=A|x|^{-\frac{\mu}{2}}
    \left(\lambda+|x|^{2(1-\frac{\mu}{N-4})}
    \right)^{-\frac{N-4}{2}},
    \end{align*}
    for $A\in\mathbb{R}$ and $\lambda>0$.
    \end{theorem}

    Let us make a brief comment about the proof of Theorem \ref{thmrsi}. The key step is making the change of variable
    \begin{align*}
    u(x)=|x|^{-\frac{\mu}{2}}v(|x|^{-\frac{\mu}{N-4}}x).
    \end{align*}
    Then by using standard spherical decomposition which needs lots of careful calculations,
    \begin{align*}
    \int_{\mathbb{R}^N}|\Delta u|^2 \mathrm{d}x
    -C_{\mu,1}\int_{\mathbb{R}^N}\frac{|\nabla u|^2}{|x|^2} \mathrm{d}x
    +C_{\mu,2}\int_{\mathbb{R}^N}\frac{|u|^2}{|x|^4} \mathrm{d}x\geq \left(1-\frac{\mu}{N-4}\right)^3\int_{\mathbb{R}^N}|\Delta v|^2 \mathrm{d}x,
    \end{align*}
    and the equality holds if and only if $v$ is radially symmetry, so does $u$. Furthermore,
    \begin{align*}
    \int_{\mathbb{R}^N}|u|^{\frac{2N}{N-4}} \mathrm{d}x
    =\left(1-\frac{\mu}{N-4}\right)^{-1}\int_{\mathbb{R}^N}
    |v|^{\frac{2N}{N-4}} \mathrm{d}x.
    \end{align*}
    Then the result follows directly from the second-order Sobolev inequality.
    In this section, following the arguments as those established by Dan et al. \cite{DMY20}, we will extend the result of Theorem \ref{thmrsi} to weighted one shown as in Theorem \ref{thmwrsi}.

    For simplicity, we set $r=|x|$. A function $u$ is said to be radial if $u$ depends only on $r$. In terms of the polar coordinate we write the Laplacian $\Delta$ as
    \[
    \Delta=\frac{\partial^2}{\partial r^2}+\frac{N-1}{r}\frac{\partial}{\partial r}+\frac{1}{r^2}\Delta_{\mathbb{S}^{N-1}},
    \]
    where $\Delta_{\mathbb{S}^{N-1}}$ is the Laplace-Beltrami operator on unit sphere $\mathbb{S}^{N-1}$. Firstly, we have the following:

    \begin{lemma}\label{lemre}
    Assume $N\geq 5$, $-2<\gamma\leq 0$ and $\gamma<\mu<N-4$. Let $t=r^{\frac{N-4-\mu}{N-4-\gamma}}$, then for all $u\in C^\infty_0(\mathbb{R}^N)$,
    \begin{align}\label{re}
    & \int_{\mathbb{R}^N}|x|^{-\mu}\left|\Delta u
    -\frac{(N-2)(\mu-\gamma)}{(N-4-\gamma)r}\frac{\partial u}{\partial r}
    -\frac{(\mu-\gamma)[2(N-4)-\gamma-\mu]}{(N-4-\gamma)^2}\frac{1}{r^2}
    \Delta_{\mathbb{S}^{N-1}} u\right|^2 \mathrm{d}x
    \nonumber\\
    & = \left(\frac{N-4-\mu}{N-4-\gamma}\right)^3
    \int_{\mathbb{S}^{N-1}}\int^\infty_0\left|
    \frac{\partial^2u}{\partial t^2}+\frac{N-1}{t}\frac{\partial u}{\partial t}
    +\frac{1}{t^2}
    \Delta_{\mathbb{S}^{N-1}} u\right|^2 t^{N-1-\gamma}\mathrm{d}t\mathrm{d}\sigma.
    \end{align}
    \end{lemma}

    \begin{proof}
    Set $t=r^{\zeta}$. Then
    \[
    \frac{\partial u}{\partial r}=\zeta r^{\zeta-1}\frac{\partial u}{\partial t},\quad \frac{\partial^2 u}{\partial r^2}=\zeta(\zeta-1)r^{\zeta-2}\frac{\partial u}{\partial t}+\zeta^2 r^{2\zeta-2}\frac{\partial^2 u}{\partial t^2}.
    \]
    A simple calculation shows, for $\zeta>0$ and $\alpha_1,\alpha_2\in\mathbb{R}$,
    \begin{align*}
    & \int_{\mathbb{R}^N}|x|^{-\mu}\left|\Delta u
    -\frac{\alpha_1}{r}\frac{\partial u}{\partial r}
    -\frac{\alpha_2}{r^2}
    \Delta_{\mathbb{S}^{N-1}} u\right|^2 \mathrm{d}x
    \\
    & =
    \int_{\mathbb{S}^{N-1}}\int^\infty_0\left|
    \frac{\partial^2 u}{\partial r^2}+\frac{N-1-\alpha_1}{r}\frac{\partial u}{\partial r}
    +\frac{1-\alpha_2}{r^2}
    \Delta_{\mathbb{S}^{N-1}} u\right|^2 r^{N-1-\mu}\mathrm{d}r\mathrm{d}\sigma \\
    & = \zeta^4\int_{\mathbb{S}^{N-1}}\int^\infty_0\left|
    \frac{\partial^2 u}{\partial t^2}+\frac{N+\zeta-2-\alpha_1}{\zeta r}\frac{\partial u}{\partial t}
    +\frac{1-\alpha_2}{\zeta^2r^2}
    \Delta_{\mathbb{S}^{N-1}} u\right|^2 r^{N-1-\mu}\mathrm{d}r\mathrm{d}\sigma \\
    & = \zeta^3\int_{\mathbb{S}^{N-1}}\int^\infty_0\left|
    \frac{\partial^2 u}{\partial t^2}+\frac{N+\zeta-2-\alpha_1}{\zeta t}\frac{\partial u}{\partial t}
    +\frac{1-\alpha_2}{\zeta^2r^2}
    \Delta_{\mathbb{S}^{N-1}} u\right|^2 t^{\frac{N+3\zeta-4-\mu}{\zeta}}\mathrm{d}t\mathrm{d}\sigma.
    \end{align*}
    Choosing $\zeta$, $\alpha_1$, $\alpha_2$ such that
    \[
    \frac{N+\zeta-2-\alpha_1}{\zeta}=N-1,\quad \frac{1-\alpha_2}{\zeta^2}=1,\quad \frac{N+3\zeta-4-\mu}{\zeta}=N-1-\gamma.
    \]
    We obtain
    \[
    \zeta=\frac{N-4-\mu}{N-4-\gamma},\quad \alpha_1=\frac{(N-2)(N-\gamma)}{N-4-\gamma},\quad \alpha_2=\frac{(\mu-\gamma)[2(N-4)-\gamma-\mu]}{(N-4-\gamma)^2}.
    \]
    Therefore, \eqref{re} holds.
    \end{proof}

    A key step of proving Theorem \ref{thmwrsi} is the following inequality:
    \begin{lemma}\label{lemtl}
    Assume $N\geq 5$, $-2<\gamma\leq 0$ and $\gamma< \mu<N-4$. Then for all $u\in C^\infty_0(\mathbb{R}^N)$,
    \begin{align}\label{tl}
    & \int_{\mathbb{R}^N}|x|^{-\mu}\left|\Delta u
    -\frac{(N-2)(\mu-\gamma)}{(N-4-\gamma)r}\frac{\partial u}{\partial r}
    -\frac{(\mu-\gamma)[2(N-4)-\gamma-\mu]}{(N-4-\gamma)^2}\frac{1}{r^2}
    \Delta_{\mathbb{S}^{N-1}} u\right|^2 \mathrm{d}x
    \nonumber\\
    & \leq \int_{\mathbb{R}^N}\frac{|\Delta (|x|^{\vartheta}u)|^2}{|x|^{\gamma}} \mathrm{d}x
    -C_{\gamma,\mu,1}\int_{\mathbb{R}^N}\frac{|\nabla (|x|^{\vartheta}u)|^2}{|x|^{\gamma+2}} \mathrm{d}x
    +C_{\gamma,\mu,2}\int_{\mathbb{R}^N}\frac{||x|^{\vartheta}u|^2}
    {|x|^{\gamma+4}} \mathrm{d}x,
    \end{align}
    where $\vartheta=\frac{\gamma-\mu}{2}$, and $C_{\gamma,\mu,1}$, $C_{\gamma,\mu,2}$ are given in \eqref{RSisc}. Furthermore, equality holds if and only if $u$ is radially symmetry.
    \end{lemma}

    Before proving Lemma \ref{lemtl}, we need the following lemma established in \cite[Lemma 2.3]{DMY20}:
    \begin{lemma}\label{lemtle}
    Let $N\geq 5$, $\mu<N-4$ and $f\in C^\infty_0([0,\infty))$. There holds, for $\beta_i\in\mathbb{R}$, $1\leq i\leq 4$,
    \begin{align*}
    & \int^\infty_0\left(f''(r)+\frac{\beta_1}{r}f'(r)
    +\frac{\beta_2}{r^2}f(r)\right)^2 r^{N-1-\mu}\mathrm{d}r
    -\int^\infty_0\left(f''(r)+\frac{\beta_3}{r}f'(r)
    +\frac{\beta_4}{r^2}f(r)\right)^2 r^{N-1-\mu}\mathrm{d}r
    \\
    & =A_1\int^\infty_0\left(f'(r)\right)^2 r^{N-3-\mu}\mathrm{d}r
    +A_2\int^\infty_0f^2(r)r^{N-5-\mu}\mathrm{d}r,
    \end{align*}
    where
    \begin{align*}
    & A_1=\beta_1^2-\beta_3^2-(\beta_1-\beta_3)(N-2-\mu)-2(\beta_2-\beta_4);
    \\
    & A_2=\beta_2^2-\beta_4^2+(\beta_2-\beta_4)(N-3-\mu)(N-4-\mu)
    -(\beta_1\beta_2-\beta_3\beta_4)(N-4-\mu).
    \end{align*}
    \end{lemma}

    \vskip0.25cm
\noindent{\em \bfseries Proof of Lemma \ref{lemtl}.} Firstly, we decompose $u$ as follows:
    \begin{equation}\label{defvd}
    u(x)=u(r,\sigma)=\sum^{\infty}_{k=0}f_k(r)\Psi_k(\sigma),
    \end{equation}
    where $r=|x|$, $\sigma=\frac{x}{|x|}\in \mathbb{S}^{N-1}$, and
    \begin{equation*}
    f_k(r)=\int_{\mathbb{S}^{N-1}}u(r,\sigma)\Psi_k(\sigma)\mathrm{d}\sigma.
    \end{equation*}
    Here $\Psi_k(\sigma)$ denotes the $k$-th spherical harmonic, i.e., it satisfies
    \begin{equation}\label{deflk}
    -\Delta_{\mathbb{S}^{N-1}}\Psi_k=c_k \Psi_k,
    \end{equation}
    where $c_k$ is the $k$-th eigenvalue of $-\Delta_{\mathbb{S}^{N-1}}$ satisfying $c_k=k(N-2+k)$, $k=0,1,2,\ldots$
    and
    \[\mathrm{Ker}(\Delta_{\mathbb{S}^{N-1}}+c_k)
    =\mathbb{Y}_k(\mathbb{R}^N)|_{\mathbb{S}^{N-1}},
    \]
    where $\mathbb{Y}_k(\mathbb{R}^N)$ is the space of all homogeneous harmonic polynomials of degree $k$ in $\mathbb{R}^N$. It is standard that $c_0=0$ and the corresponding eigenfunction of (\ref{deflk}) is the constant function. Furthermore, the functions $f_k(r)\in C^\infty_0([0,\infty))$ satisfying $f_k(r)=O(r^k)$ and $f_k'(r)=O(r^{k-1})$ as $r\to 0$. We refer to
    \cite[Section 2.2]{TZ07} for details.

    For simplicity, we let
    \[
    \int_{\mathbb{S}^{N-1}}|\Psi_k|^2\mathrm{d}\sigma=1,\quad \mbox{for all}\quad k\geq 0,
    \]
    so that
    \begin{align*}
    & \int_{\mathbb{R}^{N}}|\Delta u|^2\mathrm{d}x
    =\sum^\infty_{k=0}\int^\infty_0\left(\Delta f_k-\frac{c_k}{r^2}f_k\right)^2 r^{N-1}\mathrm{d}r,
    \\
    & \int_{\mathbb{R}^{N}}|\nabla u|^2\mathrm{d}x
    =\sum^\infty_{k=0}\int^\infty_0\left(|f_k'|^2
    +\frac{c_k}{r^2}f_k^2\right) r^{N-1}\mathrm{d}r.
    \end{align*}
    Therefore,
    \begin{align*}
    & \int_{\mathbb{R}^N}|x|^{-\mu}\left|\Delta u
    -\frac{(N-2)(\mu-\gamma)}{(N-4-\gamma)r}\frac{\partial u}{\partial r}
    -\frac{(\mu-\gamma)[2(N-4)-\gamma-\mu]}{(N-4-\gamma)^2}\frac{1}{r^2}
    \Delta_{\mathbb{S}^{N-1}} u\right|^2 \mathrm{d}x
    \\
    & = \sum^\infty_{k=0}\int^\infty_0\left\{f_k''
    +\left[N-1-\frac{(N-2)(\mu-\gamma)}{(N-4-\gamma)}\right]
    \frac{1}{r}f_k'
    -\left(\frac{N-4-\mu}{N-4-\gamma}\right)^2
    \frac{c_k}{r^2}f_k\right\}^2 r^{N-1-\mu}\mathrm{d}r,
    \end{align*}
    and
    \begin{align*}
    & \int_{\mathbb{R}^N}\frac{|\Delta (|x|^{\vartheta}u)|^2}{|x|^{\gamma}} \mathrm{d}x
    -C_{\gamma,\mu,1}\int_{\mathbb{R}^N}\frac{|\nabla (|x|^{\vartheta}u)|^2}{|x|^{\gamma+2}} \mathrm{d}x
    +C_{\gamma,\mu,2}\int_{\mathbb{R}^N}\frac{||x|^{\vartheta}u|^2}
    {|x|^{\gamma+4}} \mathrm{d}x\\
    & = \sum^\infty_{k=0}\int^\infty_0\left[\Delta(r^\vartheta f_k)
    -\frac{c_k}{r^{2}}(r^\vartheta f_k)\right]^2 r^{N-1-\gamma}\mathrm{d}r
    \\
    &\quad -C_{\gamma,\mu,1} \sum^\infty_{k=0}\int^\infty_0\left[|(r^\vartheta f_k)'|^2
    +\frac{c_k}{r^{2}}(r^\vartheta f_k)^2\right]r^{N-3-\gamma}\mathrm{d}r
    + C_{\gamma,\mu,2}\sum^\infty_{k=0}\int^\infty_0(r^\vartheta f_k)^2 r^{N-5-\gamma}\mathrm{d}r.
    \end{align*}
    Here $\vartheta=\frac{\gamma-\mu}{2}$. To finish the proof, it is enough to show
    \begin{align}\label{lte}
    & \int^\infty_0\left|\Delta(r^{\vartheta}f_0)\right|^2 r^{N-1-\gamma}\mathrm{d}r
    -C_{\gamma,\mu,1}\int^\infty_0|(r^{\vartheta}f_0)'|^2
    r^{N-3-\gamma}\mathrm{d}r
    + C_{\gamma,\mu,2}\int^\infty_0(r^{\vartheta}f_0)^2 r^{N-5-\gamma}\mathrm{d}r
    \nonumber\\
    & = \int^\infty_0\left\{f_0''
    +\left[N-1-\frac{(N-2)(\mu-\gamma)}{(N-4-\gamma)}\right]
    \frac{1}{r}f_0'\right\}^2 r^{N-1-\mu}\mathrm{d}r,
    \end{align}
    and for $k\geq 1$ and $f_k\in C^\infty_0([0,\infty))\setminus\{0\}$,
    \begin{align}\label{ltl}
    & \int^\infty_0\left\{f_k''
    +\left[N-1-\frac{(N-2)(\mu-\gamma)}{(N-4-\gamma)}\right]
    \frac{1}{r}f_k'
    -\left(\frac{N-4-\mu}{N-4-\gamma}\right)^2
    \frac{c_k}{r^2}f_k\right\}^2 r^{N-1-\mu}\mathrm{d}r
    \nonumber\\
    & <\int^\infty_0\left[\Delta(r^\vartheta f_k)
    -\frac{c_k}{r^{2}}(r^\vartheta f_k)\right]^2 r^{N-1-\gamma}\mathrm{d}r
    -C_{\gamma,\mu,1} \int^\infty_0\left[|(r^\vartheta f_k)'|^2
    +\frac{c_k}{r^{2}}(r^\vartheta f_k)^2\right]r^{N-3-\gamma}\mathrm{d}r
    \nonumber\\
    &\quad + C_{\gamma,\mu,2}\int^\infty_0(r^\vartheta f_k)^2 r^{N-5-\gamma}\mathrm{d}r.
    \end{align}

    Firstly, we prove \eqref{lte}. By direct calculations we have
    \begin{align}\label{ltec}
    & \int^\infty_0\left|\Delta(r^\vartheta f_0)\right|^2 r^{N-1-\gamma}\mathrm{d}r
    -C_{\gamma,\mu,1}\int^\infty_0|(r^\vartheta f_0)'|^2r^{N-3-\gamma}\mathrm{d}r
    + C_{\gamma,\mu,2}\int^\infty_0(r^\vartheta f_0)^2 r^{N-5-\gamma}\mathrm{d}r
    \nonumber\\
    & = \int^\infty_0\left[f_0''+\frac{N-1+2\vartheta }{r}f_0'
    +\frac{\vartheta (N-2+\vartheta )}{r^2}f_0\right]^2 r^{N-1-\mu}\mathrm{d}r
    -C_{\gamma,\mu,1}\int^\infty_0|f_0'|^2r^{N-3-\mu}\mathrm{d}r
    \nonumber\\
    &\quad + \left[C_{\gamma,\mu,2}+\vartheta (N-4-\mu-\vartheta )C_{\gamma,\mu,1}\right]
    \int^\infty_0|f_0|^2 r^{N-5-\mu}\mathrm{d}r.
    \end{align}
    Furthermore, as in Lemma \ref{lemtle}, we take
    \[
    \beta_1=N-1-\frac{(N-2)(\mu-\gamma)}{(N-4-\gamma)},\quad \beta_2=0,\quad \beta_3=N-1+2\vartheta ,\quad \beta_4=\vartheta (N-2+\vartheta ),
    \]
    then,
    \begin{align}\label{ltet}
    & \int^\infty_0\left\{f_0''
    +\left[N-1-\frac{(N-2)(\mu-\gamma)}{(N-4-\gamma)}\right]
    \frac{1}{r}f_0'\right\}^2 r^{N-1-\mu}\mathrm{d}r
    \nonumber\\
    &\quad -\int^\infty_0\left[f_0''+\frac{N-1+2\vartheta }{r}f_0'
    +\frac{\vartheta (N-2+\vartheta )}{r^2}f_0\right]^2 r^{N-1-\mu}\mathrm{d}r
    \nonumber\\
    & = A_1\int^\infty_0|f_0'|^2 r^{N-3-\mu}\mathrm{d}r
    +A_2\int^\infty_0|f_0|^2 r^{N-5-\mu}\mathrm{d}r,
    \end{align}
    where
    \begin{align*}
    & A_1=2\vartheta (N-2+\vartheta )
    -\left[\frac{(N-2)(\mu-\gamma)}{N-4-\gamma}+2\vartheta \right]
    \left[N+2\vartheta +\mu-\frac{(N-2)(\mu-\gamma)}{N-4-\gamma}\right],
    \\ 
    & A_2=\vartheta (N-2+\vartheta )(N-4-\mu)(2+\mu+2\vartheta )-\vartheta ^2(N-2+\vartheta )^2.
    \end{align*}
    It is easy to verify that
    \begin{align*}
    C_{\gamma,\mu,1}=-A_1,\quad C_{\gamma,\mu,2}=A_2-\vartheta (N-4-\mu-\vartheta )C_{\gamma,\mu,1},
    \end{align*}
    where $C_{\gamma,\mu,1}$, $C_{\gamma,\mu,2}$ are given in \eqref{RSisc}, then by \eqref{ltec} and \eqref{ltet},
    \begin{align*}
    & \int^\infty_0\left\{f_0''
    +\left[N-1-\frac{(N-2)(\mu-\gamma)}{(N-4-\gamma)}\right]
    \frac{1}{r}f_0'\right\}^2 r^{N-1-\mu}\mathrm{d}r
    -\int^\infty_0\left|\Delta(r^\vartheta f_0)\right|^2 r^{N-1-\gamma}\mathrm{d}r
    \\
    & = -C_{\gamma,\mu,1}\int^\infty_0|(r^\vartheta f_0)'|^2r^{N-3-\gamma}\mathrm{d}r
    + C_{\gamma,\mu,2}\int^\infty_0(r^\vartheta f_0)^2 r^{N-5-\gamma}\mathrm{d}r,
    \end{align*}
    thus \eqref{lte} holds.

    Secondly, we prove \eqref{ltl} holds for all $k\geq 1$. By Lemma \ref{lemtle} we have
    \begin{align}\label{ltle}
    & \int^\infty_0\left\{f_k''
    +\left[N-1-\frac{(N-2)(\mu-\gamma)}{(N-4-\gamma)}\right]
    \frac{1}{r}f_k'
    -\left(\frac{N-4-\mu}{N-4-\gamma}\right)^2
    \frac{c_k}{r^2}f_k\right\}^2 r^{N-1-\mu}\mathrm{d}r
    \nonumber\\
    & =\int^\infty_0\left\{f_k''
    +\left[N-1-\frac{(N-2)(\mu-\gamma)}{(N-4-\gamma)}\right]
    \frac{1}{r}f_k'\right\}^2 r^{N-1-\mu}\mathrm{d}r
    \nonumber\\
    &\quad + 2\left(\frac{N-4-\mu}{N-4-\gamma}\right)^2
    \int^\infty_0|f_k'|^2 r^{N-3-\mu}\mathrm{d}r
    \nonumber\\
    &\quad +\left(\frac{N-4-\mu}{N-4-\gamma}\right)^4
    \left[c_k^2+(2+\gamma)(N-4-\gamma)c_k\right]
    \int^\infty_0|f_k|^2 r^{N-5-\mu}\mathrm{d}r,
    \end{align}
    and
    \begin{align}\label{ltle2}
    & \int^\infty_0\left[\Delta(r^\vartheta f_k)
    -\frac{c_k}{r^{2}}(r^\vartheta f_k)\right]^2 r^{N-1-\gamma}\mathrm{d}r
    \nonumber\\
    & =\int^\infty_0\left[f_k''+\frac{N-1+2\vartheta }{r}f_k'
    -\frac{c_k -\vartheta (N-2+\vartheta )}{r^{2}}f_k\right]^2 r^{N-1-\mu}\mathrm{d}r
    \nonumber\\
    & = \int^\infty_0\left[f_k''+\frac{N-1+2\vartheta }{r}f_k'
    +\frac{\vartheta (N-2+\vartheta )}{r^{2}}f_k\right]^2 r^{N-1-\mu}\mathrm{d}r
    +2c_k\int^\infty_0|f_k'|^2 r^{N-3-\mu}\mathrm{d}r
    \nonumber\\
    & \quad + c_k\left[c_k-2\vartheta (N-2+\vartheta )+(N-4-\mu)(2+2\vartheta +\mu)\right]
    \int^\infty_0|f_k|^2 r^{N-5-\mu}\mathrm{d}r.
    \end{align}
    Therefore, by \eqref{ltle}-\eqref{ltle2} and \eqref{lte},
    \begin{align*}
    & \int^\infty_0\left[\Delta(r^\vartheta f_k)
    -\frac{c_k}{r^{2}}(r^\vartheta f_k)\right]^2 r^{N-1-\gamma}\mathrm{d}r
    -C_{\gamma,\mu,1} \int^\infty_0\left[|(r^\vartheta f_k)'|^2
    +\frac{c_k}{r^{2}}(r^\vartheta f_k)^2\right]r^{N-3-\gamma}\mathrm{d}r
    \\
    &\quad + C_{\gamma,\mu,2}\int^\infty_0|f_k|^2 r^{N-5-\gamma+2\vartheta }\mathrm{d}r
    \\
    &\quad-\int^\infty_0\left\{f_k''
    +\left[N-1-\frac{(N-2)(\mu-\gamma)}{(N-4-\gamma)}\right]
    \frac{1}{r}f_k'
    -\left(\frac{N-4-\mu}{N-4-\gamma}\right)^2
    \frac{c_k}{r^2}f_k\right\}^2 r^{N-1-\mu}\mathrm{d}r
    \\
    & =2\left[1-\left(\frac{N-4-\mu}{N-4-\gamma}\right)^2\right]c_k
    \int^\infty_0|f_k'|^2 r^{N-3-\mu}\mathrm{d}r
    +A_{3,k}\int^\infty_0|f_k|^2 r^{N-5-\mu}\mathrm{d}r,
    \end{align*}
    where
    \begin{align*}
    A_{3,k}
    & = c_k\left[c_k-2\vartheta (N-2+\vartheta )+(N-4-\mu)(2+2\vartheta +\mu)\right]
    \\
    &\quad -\left(\frac{N-4-\mu}{N-4-\gamma}\right)^4
    \left[c_k^2+(2+\gamma)(N-4-\gamma)c_k\right]
    -C_{\gamma,\mu,1}\cdot c_k
    \\
    & = \left[1-\left(\frac{N-4-\mu}{N-4-\gamma}\right)^4\right]
    c_k\left[c_k+(2+\gamma)(N-4-\gamma)\right]
    \\
    &\quad +\left[(\mu-\gamma)(N-4-\gamma)-\frac{(\mu-\gamma)^2}{2}
    -C_{\gamma,\mu,1}\right]c_k.
    \end{align*}
    To finish the proof, since $\gamma< \mu<N-4$, it is enough to show $A_{3,k}>0$ for all $k\geq 1$. Set $\zeta=\frac{N-4-\mu}{N-4-\gamma}$. Then $\zeta\in (0,1)$ and
    \begin{align*}
    C_{\gamma,\mu,1}
    & =\frac{N^2-4N+8+\gamma^2+4\gamma}{2(N-4-\gamma)^2}(\mu-\gamma)
    [2(N-4-\gamma)-(\mu-\gamma)]
    \\
    & =\frac{N^2-4N+8+\gamma^2+4\gamma}{2}(1-\zeta^2).
    \end{align*}
    Therefore,
    \begin{align*}
    A_{3,k}
    & = (1-\zeta^4)
    c_k\left[c_k+(2+\gamma)(N-4-\gamma)\right]
    \\
    &\quad +\left[\frac{2(1-\zeta)-(1-\zeta)^2}{2}(N-4-\gamma)^2
    -\frac{N^2-4N+8+\gamma^2+4\gamma}{2}(1-\zeta^2)\right]c_k
    \\
    & = (1-\zeta^4)
    c_k\left[c_k+(2+\gamma)(N-4-\gamma)\right]
    -(1-\zeta^2)(N-2)(2+\gamma)c_k
    \\
    & = (1-\zeta^2)\left\{\zeta^2\left[c_k+(2+\gamma)(N-4-\gamma)\right]
    +[c_k-(2+\gamma)^2]\right\}c_k>0.
    \end{align*}
    Here we use the facts that $N\geq 5$, $c_k=k(N-2+k)\geq 4$ for $k\geq 1$, and $-2<\gamma\leq 0$. The proof of Lemma \ref{lemtl} is thereby completed.
\qed

\vskip0.25cm

Now, we are ready to establish the weighted Rellich-Sobolev inequality shown as in Theorem \ref{thmwrsi}.

\vskip0.25cm

\noindent{\bf \em Proof of Theorem \ref{thmwrsi}}. Let $t=r^{\frac{N-4-\mu}{N-4-\gamma}}$ with $r=|x|$, and
    \[
    v(x)=u(|x|^{\frac{N-4-\mu}{N-4-\gamma}-1}x).
    \]
 By Lemma \ref{lemre} and Theorem \ref{corolehups},
\begin{align}\label{retb}
    & \int_{\mathbb{R}^N}|x|^{-\mu}\left|\Delta u
    -\frac{(N-2)(\mu-\gamma)}{(N-4-\gamma)r}\frac{\partial u}{\partial r}
    -\frac{(\mu-\gamma)[2(N-4)-\gamma-\mu]}{(N-4-\gamma)^2}\frac{1}{r^2}
    \Delta_{\mathbb{S}^{N-1}} u\right|^2 \mathrm{d}x
    \nonumber\\
    & = \left(\frac{N-4-\mu}{N-4-\gamma}\right)^3
    \int_{\mathbb{S}^{N-1}}\int^\infty_0\left|
    \frac{\partial^2u}{\partial t^2}+\frac{N-1}{t}\frac{\partial u}{\partial t}
    +\frac{1}{t^2}
    \Delta_{\mathbb{S}^{N-1}} u\right|^2 t^{N-1-\gamma}\mathrm{d}t\mathrm{d}\sigma
    \nonumber\\
    & = \left(\frac{N-4-\mu}{N-4-\gamma}\right)^3
    \int_{\mathbb{S}^{N-1}}\int^\infty_0\left|
    \frac{\partial^2 v}{\partial t^2}+\frac{N-1}{t}\frac{\partial v}{\partial t}
    +\frac{1}{t^2}
    \Delta_{\mathbb{S}^{N-1}} v\right|^2 t^{N-1-\gamma}\mathrm{d}t\mathrm{d}\sigma
    \nonumber\\
    & \geq \left(\frac{N-4-\mu}{N-4-\gamma}\right)^3\mathcal{S}_{N,\gamma}
    \left(\int_{\mathbb{S}^{N-1}}\int^\infty_0
    |v|^{2^{**}_{0,\gamma}} t^{N-1+\gamma}\mathrm{d}t\mathrm{d}\sigma
    \right)^{\frac{2}{2^{**}_{0,\gamma}}}
    \nonumber\\
    & = \left(\frac{N-4-\mu}{N-4-\gamma}\right)^{3+\frac{2}{2^{**}_{0,\gamma}}}
    \mathcal{S}_{N,\gamma}
    \left(\int_{\mathbb{S}^{N-1}}\int^\infty_0
    |u|^{2^{**}_{0,\gamma}} r^{\frac{(N-4-\mu)(N+\gamma)}{N-4-\gamma}-1}\mathrm{d}r\mathrm{d}\sigma
    \right)^{\frac{2}{2^{**}_{0,\gamma}}}
    \nonumber\\
    & = \left(\frac{N-4-\mu}{N-4-\gamma}\right)^{3+\frac{2}{2^{**}_{0,\gamma}}}
    \mathcal{S}_{N,\gamma}
    \left(\int_{\mathbb{R}^N}
    |u|^{2^{**}_{0,\gamma}} |x|^{\frac{(\gamma-\mu)2^{**}_{0,\gamma}}{2}+\gamma}\mathrm{d}x
    \right)^{\frac{2}{2^{**}_{0,\gamma}}}.
    \end{align}
    Therefore, by Lemma \ref{lemtl} and \eqref{retb},
    \begin{align*}
    & \int_{\mathbb{R}^N}\frac{|\Delta (|x|^\vartheta u)|^2}{|x|^{\gamma}} \mathrm{d}x
    -C_{\gamma,\mu,1}\int_{\mathbb{R}^N}\frac{|\nabla (|x|^\vartheta u)|^2}{|x|^{\gamma+2}} \mathrm{d}x
    +C_{\gamma,\mu,2}\int_{\mathbb{R}^N}\frac{||x|^\vartheta u|^2}{|x|^{\gamma+4}} \mathrm{d}x
    \\
    & \geq \int_{\mathbb{R}^N}|x|^{-\mu}\left|\Delta u
    -\frac{(N-2)(\mu-\gamma)}{(N-4-\gamma)r}\frac{\partial u}{\partial r}
    -\frac{(\mu-\gamma)[2(N-4)-\gamma-\mu]}{(N-4-\gamma)^2}\frac{1}{r^2}
    \Delta_{\mathbb{S}^{N-1}} u\right|^2 \mathrm{d}x
    \\
    & \geq \left(\frac{N-4-\mu}{N-4-\gamma}\right)^{3+\frac{2}{2^{**}_{0,\gamma}}}
    \mathcal{S}_{N,\gamma}
    \left(\int_{\mathbb{R}^N}
    |x|^{\gamma}||x|^\vartheta u|^{2^{**}_{0,\gamma}}\mathrm{d}x
    \right)^{\frac{2}{2^{**}_{0,\gamma}}},
    \end{align*}
    where $\vartheta =\frac{\gamma-\mu}{2}$. Replacing $|x|^\vartheta u$ by $u$, we get the weighted Rellich-Sobolev type inequality \eqref{wrsi}.

    Also by Lemma \ref{lemtl} and Theorem \ref{corolehups}, we know the extremal functions are radially symmetry. Therefore, we have equality in \eqref{wrsi} if and only if
    \begin{align*}
    u(x)=c|x|^{\frac{\gamma-\mu}{2}}
    \left(\lambda+t^{2+\gamma}
    \right)^{-{\frac{N-4-\gamma}{2+\gamma}}}=c|x|^{\frac{\gamma-\mu}{2}}
    \left(\lambda+|x|^{\frac{(2+\gamma)(N-4-\mu)}{N-4-\gamma}}
    \right)^{-{\frac{N-4-\gamma}{2+\gamma}}},
    \end{align*}
    for $A\in\mathbb{R}$ and $\lambda>0$. Now, the proof of Theorem \ref{thmwrsi} is completed.
\qed

\vskip0.25cm

    Passing $\mu\uparrow N-4$ in \eqref{wrsi}, then we obtain a sharp inequality involving Hardy and Rellich terms, which reads that:
    \begin{corollary}\label{coroshri}
    Assume that $N\geq 5$ and $-2<\gamma\leq 0$. For all $u\in \mathcal{D}^{2,2}_{0,\gamma}(\mathbb{R}^N)$,
    \begin{align}\label{shri}
    \int_{\mathbb{R}^N}\frac{|\Delta u|^2}{|x|^{\gamma}} \mathrm{d}x
    +\left(\frac{N-4-\gamma}{2}\right)^4
    \int_{\mathbb{R}^N}\frac{|u|^2}{|x|^{\gamma+4}} \mathrm{d}x
    \geq \frac{N^2-4N+8+\gamma^2+4\gamma}{2}\int_{\mathbb{R}^N}\frac{|\nabla u|^2}{|x|^{\gamma+2}} \mathrm{d}x.
    \end{align}
    The constant $\frac{N^2-4N+8+\gamma^2+4\gamma}{2}$ is sharp and the inequality is strict for any nonzero $u$.
    \end{corollary}

    \begin{proof}
    Passing $\mu\uparrow N-4$ in \eqref{wrsi}, we directly obtain \eqref{shri}.
    Next, we will show the constant $\frac{N^2-4N+8+\gamma^2+4\gamma}{2}$ is sharp in the sense of
    \begin{align*}
    \frac{N^2-4N+8+\gamma^2+4\gamma}{2}
    =\inf_{u\in \mathcal{D}^{2,2}_{0,\gamma}(\mathbb{R}^N)\setminus\{0\}}
    \frac{\int_{\mathbb{R}^N}\frac{|\Delta u|^2}{|x|^{\gamma}} \mathrm{d}x
    +\left(\frac{N-4-\gamma}{2}\right)^4
    \int_{\mathbb{R}^N}\frac{|u|^2}{|x|^{\gamma+4}} \mathrm{d}x}
    {\int_{\mathbb{R}^N}\frac{|\nabla u|^2}{|x|^{\gamma+2}} \mathrm{d}x}.
    \end{align*}

    We notice that the work of Catrina and Wang
    \cite[Theorem 1.1 (ii)]{CW01} indicates the following weighted Hardy inequality
    \begin{align}\label{defwhi}
    \int_{\mathbb{R}^N}\frac{|\nabla u|^2}{|x|^{\gamma+2}} \mathrm{d}x
    \geq \left(\frac{N-4-\gamma}{2}\right)^2
    \int_{\mathbb{R}^N}\frac{|u|^2}{|x|^{\gamma+4}} \mathrm{d}x,\quad \forall u\in C^\infty_0(\mathbb{R}^N),
    \end{align}
    for $N\geq 3$ and $\gamma<N-4$, the constant $\left(\frac{N-4-\gamma}{2}\right)^2$ is sharp and the inequality is strict for any nonzero $u$. Furthermore, the work of Ghoussoub and Moradifam
    \cite[Theorem 6.1 (3)]{GM11} (see also \cite[Corollary 5.9]{DM22} for more general case) indicates the following weighted Hardy-Rellich inequality
    \begin{align}\label{defwhri}
    \int_{\mathbb{R}^N}\frac{|\Delta u|^2}{|x|^{\gamma}} \mathrm{d}x
    \geq \left(\frac{N+\gamma}{2}\right)^2
    \int_{\mathbb{R}^N}\frac{|\nabla u|^2}{|x|^{\gamma+2}} \mathrm{d}x,\quad \forall u\in C^\infty_0(\mathbb{R}^N),
    \end{align}
    for $N\geq 2$ and $\frac{-(N+4)-2\sqrt{N^2-N+1}}{3}\leq\gamma\leq \min\left\{N-2,\frac{-(N+4)+2\sqrt{N^2-N+1}}{3}\right\}$, the constant $\left(\frac{N+\gamma}{2}\right)^2$ is sharp. Note that
    \begin{align*}
    & \frac{-(N+4)+2\sqrt{N^2-N+1}}{3}>0\quad\mbox{holds for all}\quad N>2+2\sqrt{2};
    \\
    & \frac{-(N+4)-2\sqrt{N^2-N+1}}{3}<-2\quad\mbox{holds for all}\quad N\geq 1.
    \end{align*}
    Therefore, when $N\geq 5$ and $-2<\gamma\leq0$, we can always obtain the sharp weighted Hardy inequality \eqref{defwhi} and weighted Hardy-Rellich inequality \eqref{defwhri} in $\mathcal{D}^{2,2}_{0,\gamma}(\mathbb{R}^N)$, due to $\mathcal{D}^{2,2}_{0,\gamma}(\mathbb{R}^N)$ is dense in $C^\infty_0(\mathbb{R}^N)$. Then,
    \begin{align*}
    & \inf_{u\in \mathcal{D}^{2,2}_{0,\gamma}(\mathbb{R}^N)\setminus\{0\}}
    \frac{\int_{\mathbb{R}^N}\frac{|\Delta u|^2}{|x|^{\gamma}} \mathrm{d}x
    +\left(\frac{N-4-\gamma}{2}\right)^4
    \int_{\mathbb{R}^N}\frac{|u|^2}{|x|^{\gamma+4}} \mathrm{d}x}
    {\int_{\mathbb{R}^N}\frac{|\nabla u|^2}{|x|^{\gamma+2}} \mathrm{d}x}
    \\
    & \leq \inf_{u\in \mathcal{D}^{2,2}_{0,\gamma}(\mathbb{R}^N)\setminus\{0\}}
    \frac{\int_{\mathbb{R}^N}\frac{|\Delta u|^2}{|x|^{\gamma}} \mathrm{d}x
    +\left(\frac{N-4-\gamma}{2}\right)^2
    \int_{\mathbb{R}^N}\frac{|\nabla u|^2}{|x|^{\gamma+2}} \mathrm{d}x}
    {\int_{\mathbb{R}^N}\frac{|\nabla u|^2}{|x|^{\gamma+2}} \mathrm{d}x}
    \\
    & \leq \left(\frac{N+\gamma}{2}\right)^2+\left(\frac{N-4-\gamma}{2}\right)^2
    =\frac{N^2-4N+8+\gamma^2+4\gamma}{2},
    \end{align*}
    which shows the constant $\frac{N^2-4N+8+\gamma^2+4\gamma}{2}$ in \eqref{shri} is sharp. Moreover, the infimum in above can not be achieved due to the weighted Hardy inequality \eqref{defwhi} is strict, therefore the inequality in \eqref{shri} is also strict for any nonzero $u$. Now, the proof of Corollary \ref{coroshri} is completed.
\end{proof}

\section{{\bfseries Sharp second-order CKN type inequality}}\label{sectcknps}
Firstly, we show the second-order CKN type inequality \eqref{cknrs} holds which does not need the additional condition \eqref{defac} with $2\alpha-\beta=-2a$.

\begin{proposition}\label{propgeqd}
Assume that $N\geq 5$, $-N<\alpha-2\leq \beta\leq \frac{N\alpha}{N-2}$. Then there is a constant $\mathcal{S}>0$ such that
\begin{equation}\label{ckn2nf}
    \int_{\mathbb{R}^N}|x|^{-\beta}|\mathrm{div} (|x|^{\alpha}\nabla u)|^2 \mathrm{d}x
    \geq \mathcal{S}\left(\int_{\mathbb{R}^N}
    |x|^{\beta}|u|^{2^{**}_{\alpha,\beta}} \mathrm{d}x\right)^{\frac{2}{2^{**}_{\alpha,\beta}}},\quad \forall u\in C^\infty_0(\mathbb{R}^N).
    \end{equation}
\end{proposition}

\begin{proof}
We know from \cite[Proposition A.1]{DT23-f} which states that there is a constant $\mathfrak{C}=\mathfrak{C}(N,\alpha,\beta)>0$ such that
\begin{align}\label{neq}
\int_{\mathbb{R}^N}|x|^{2\alpha-\beta}|\Delta u|^2 \mathrm{d}x
\leq \mathfrak{C} \int_{\mathbb{R}^N}|x|^{-\beta}|\mathrm{div} (|x|^{\alpha}\nabla u)|^2 \mathrm{d}x,\quad \forall u\in C^\infty_0(\mathbb{R}^N).
\end{align}
For the readers' convenience, here we give the proof of \eqref{neq}. For $u\in C^\infty_0(\mathbb{R}^N)$, let
\begin{align}\label{neqer}
w:=-|x|^{-\frac{\beta}{2}}\mathrm{div} (|x|^{\alpha}\nabla u),
\end{align}
then we see that
\begin{align}\label{neqer1}
\int_{\mathbb{R}^N}|w|^2\mathrm{d}x
=\int_{\mathbb{R}^N}|x|^{-\beta}|\mathrm{div} (|x|^{\alpha}\nabla u)|^2\mathrm{d}x.
\end{align}
It follows from \eqref{neqer} that
\begin{align*}
|x|^{\alpha-\frac{\beta}{2}}\Delta u=-w-\alpha|x|^{\alpha-\frac{\beta}{2}-2}(x\cdot \nabla u),
\end{align*}
thus
\begin{align}\label{neqer2}
\int_{\mathbb{R}^N}|x|^{2\alpha-\beta}|\Delta u|^2\mathrm{d}x
& = \int_{\mathbb{R}^N}|w|^2\mathrm{d}x
+2\alpha\int_{\mathbb{R}^N}|x|^{\alpha-\frac{\beta}{2}-2}(x\cdot \nabla u)w\mathrm{d}x
\nonumber\\
& \quad +\alpha^2\int_{\mathbb{R}^N} |x|^{2\alpha-\beta-4}(x\cdot \nabla u)^2\mathrm{d}x.
\end{align}
Taking $\phi=|x|^{\alpha-\beta-2} u$ as a test function into \eqref{neqer} which is equivalent to $|x|^{\frac{\beta}{2}}w=-\mathrm{div} (|x|^{\alpha}\nabla u)$,
\[
\int_{\mathbb{R}^N}|x|^{\frac{\beta}{2}}w \phi\mathrm{d}x
=\int_{\mathbb{R}^N}|x|^{\alpha}\nabla u\cdot \nabla \phi\mathrm{d}x,
\]
therefore,
\begin{align}\label{neqere}
\int_{\mathbb{R}^N} |x|^{\alpha-\frac{\beta}{2}-2}wu\mathrm{d}x
& =\int_{\mathbb{R}^N} |x|^{2\alpha-\beta-2}|\nabla u|^2\mathrm{d}x+(\alpha-\beta-2)\int|x|^{2\alpha-\beta-4}
\left(x\cdot \nabla \left(\frac{|u|^2}{2}\right)\right)\mathrm{d}x
\nonumber\\
& =\frac{(2+\beta-\alpha)(N+2\alpha-\beta-4)}{2}
\int_{\mathbb{R}^N}|x|^{2\alpha-\beta-4}|u|^2\mathrm{d}x
\nonumber\\
& \quad +\int_{\mathbb{R}^N} |x|^{2\alpha-\beta-2}|\nabla u|^2\mathrm{d}x.
\end{align}
Note that the assumptions imply $\frac{(2+\beta-\alpha)(N+2\alpha-\beta-4)}{2}\geq 0$,
then it follows from \eqref{neqere} and the H\"{o}lder inequality that
\begin{align*}
\int_{\mathbb{R}^N} |x|^{2\alpha-\beta-2}|\nabla u|^2\mathrm{d}x
\leq \int_{\mathbb{R}^N} |x|^{\alpha-\frac{\beta}{2}-2}wu\mathrm{d}x
\leq \left(\int_{\mathbb{R}^N} |w|^2\mathrm{d}x\right)^{\frac{1}{2}}\left(\int_{\mathbb{R}^N} |x|^{2\alpha-\beta-4}|u|^2\mathrm{d}x\right)^{\frac{1}{2}}.
\end{align*}
Since the assumptions also imply $-\frac{2\alpha-\beta-2}{2}<\frac{N-2}{2}$, by using the weighted Hardy inequality (see \cite{CW01}) we obtain
\[
\int_{\mathbb{R}^N}|x|^{2\alpha-\beta-4}|u|^2\mathrm{d}x
\leq E \int_{\mathbb{R}^N}|x|^{2\alpha-\beta-2}|\nabla u|^2\mathrm{d}x,
\]
for some $E>0$ independent of $u$ (in fact, from \cite{CW01} we know the sharp constant is $E=(\frac{N-4+2\alpha-\beta}{2})^2$), then we see that
\begin{align*}
\int_{\mathbb{R}^N} |x|^{2\alpha-\beta-2}|\nabla u|^2\mathrm{d}x
\leq E\int_{\mathbb{R}^N} |w|^2\mathrm{d}x.
\end{align*}
By the Young inequality and $(x\cdot\nabla u)^2\leq |x|^2|\nabla u|^2$, it follows from \eqref{neqer1} and \eqref{neqer2} that
\begin{align*}
\int_{\mathbb{R}^N}|x|^{2\alpha-\beta}|\Delta u|^2\mathrm{d}x
& \leq \int_{\mathbb{R}^N}|w|^2\mathrm{d}x
+|\alpha|\left(\int_{\mathbb{R}^N}|x|^{2\alpha-\beta-2}|\nabla u|^2\mathrm{d}x
+\int_{\mathbb{R}^N}|w|^2\mathrm{d}x
\right)
\nonumber\\
& \quad +\alpha^2\int_{\mathbb{R}^N} |x|^{2\alpha-\beta-2}|\nabla u|^2\mathrm{d}x
\\
& \leq \left[1+|\alpha|+E(|\alpha|+\alpha^2)\right]\int_{\mathbb{R}^N}
|x|^{-\beta}|\mathrm{div} (|x|^{\alpha}\nabla u)|^2\mathrm{d}x.
\end{align*}
Therefore, we obtain inequality \eqref{neq} with $\mathfrak{C}=1+|\alpha|+E(|\alpha|+\alpha^2)$.

When $2-N<\alpha\leq 0$, it is easy to verify that $\frac{2\alpha-\beta}{-2}<\frac{N-4}{2}$ and $\frac{2\alpha-\beta}{-2}>-\frac{N}{2}$ thus the condition \eqref{defac} always holds with $2\alpha-\beta=-2a$.
Combining with \eqref{neq} and the second-order CKN inequality \eqref{ckn2} so that \eqref{ckn2nf} holds.

Now, we only need to consider $\alpha>0$. By a direct calculation we have $\mathrm{div} (|x|^{\alpha}\nabla u)=|x|^\alpha\Delta u+\alpha |x|^{\alpha-2}(x\cdot\nabla u)$, then
    \begin{align}\label{ckn2nfe}
    \int_{\mathbb{R}^N}|x|^{-\beta}|\mathrm{div} (|x|^{\alpha}\nabla u)|^2\mathrm{d}x
    & = \int_{\mathbb{R}^N}|x|^{2\alpha-\beta}|\Delta u|^2\mathrm{d}x
    + 2\alpha \int_{\mathbb{R}^N}|x|^{2\alpha-2-\beta}\Delta u(x\cdot\nabla u)\mathrm{d}x
    \nonumber \\
    & \quad + \alpha^2 \int_{\mathbb{R}^N}|x|^{2\alpha-4-\beta}(x\cdot\nabla u)^2\mathrm{d}x\nonumber \\
    & = \int_{\mathbb{R}^N}|x|^{2\alpha-\beta}|\Delta u|^2\mathrm{d}x
    + \alpha (N-4+2\alpha-\beta)\int_{\mathbb{R}^N}
    |x|^{2\alpha-2-\beta}|\nabla u|^2\mathrm{d}x
    \nonumber \\
    & \quad+ \alpha(2\beta-3\alpha+4) \int_{\mathbb{R}^N}|x|^{2\alpha-4-\beta}(x\cdot\nabla u)^2\mathrm{d}x,
    \end{align}
    thanks to
    \begin{align*}
    \int_{\mathbb{R}^N}|x|^{2\alpha-2-\beta}\Delta u(x\cdot\nabla u)\mathrm{d}x
    & = \frac{N-4+2\alpha-\beta}{2}\int_{\mathbb{R}^N}
    |x|^{2\alpha-2-\beta}|\nabla u|^2\mathrm{d}x
    \\
    & \quad + (\beta-2\alpha+2) \int_{\mathbb{R}^N}|x|^{2\alpha-4-\beta}(x\cdot\nabla u)^2\mathrm{d}x,
    \end{align*}
    see \eqref{psny3}.
    Since $N-4+2\alpha-\beta\geq N-4+2\alpha-\frac{N\alpha}{N-2}=\frac{(N-4)(N-2+\alpha)}{N-2}>0$, if $2\beta-3\alpha+4\geq 0$,
     \begin{align*}
    \int_{\mathbb{R}^N}|x|^{-\beta}|\mathrm{div} (|x|^{\alpha}\nabla u)|^2\mathrm{d}x
    \geq \int_{\mathbb{R}^N}|x|^{2\alpha-\beta}|\Delta u|^2\mathrm{d}x
    + \alpha (N-4+2\alpha-\beta)\int_{\mathbb{R}^N}
    |x|^{2\alpha-2-\beta}|\nabla u|^2\mathrm{d}x,
    \end{align*}
    then we can deduce \eqref{ckn2nf} from the first-order CKN inequality \eqref{ckn1} due to $-\frac{2\alpha-2-\beta}{2}< \frac{N-2}{2}$.

    It remains only the case $\alpha>0$ and $2\beta-3\alpha+4<0$. Thanks to $(x\cdot\nabla u)^2\leq |x|^2|\nabla u|^2$, then it follows from \eqref{ckn2nfe} that
    \begin{align*}
    \int_{\mathbb{R}^N}|x|^{-\beta}|\mathrm{div} (|x|^{\alpha}\nabla u)|^2\mathrm{d}x
    \geq & \int_{\mathbb{R}^N}|x|^{2\alpha-\beta}|\Delta u|^2\mathrm{d}x
    + \alpha (N-\alpha+\beta)\int_{\mathbb{R}^N}
    |x|^{2\alpha-2-\beta}|\nabla u|^2\mathrm{d}x.
    \end{align*}
    In this case, we have $N-\alpha+\beta\geq N-\alpha+\alpha-2=N-2>0$, then the first-order CKN inequality \eqref{ckn1} also indicates \eqref{ckn2nf} holds. Now, the proof is completed.
\end{proof}

    Now, we establish the sharp second-order CKN type inequality \eqref{cknrs} when $2-N<\alpha< 0$ and $\alpha-2<\beta\leq \frac{N}{N-2}\alpha$, which has radially symmetry extremal functions.

\vskip0.25cm

\noindent{\bf \em Proof of Theorem \ref{thmcknrs}}. Note that when $\alpha=\beta=0$, then \eqref{cknrs} reduces to classical second-order Sobolev inequality (see \cite{EFJ90,Li85-1,Va93}). Furthermore, if $\alpha=0$ and $-2<\beta<0$, then \eqref{cknrs} reduces to second CKN inequality \eqref{lehi} obtained in Theorem \ref{corolehups}. Therefore, we only need to consider the case $2-N<\alpha<0$.

    For each $u\in \mathcal{D}^{2,2}_{\alpha,\beta}(\mathbb{R}^N)$, let us make the change
    \[
    u(x)=|x|^{\eta}v(x) \quad\mbox{with}\quad
    \eta=-\frac{\alpha (N-4+2\alpha-\beta)}{2(N-2+\alpha)}.
    \]
    Let
    \begin{align}\label{defgm}
    \gamma:=\beta -2(\alpha+\eta)=\frac{\beta(N-2)-\alpha N}{N-2+\alpha}.
    \end{align}
    From the assumptions $2-N<\alpha<0$ and $\alpha-2<\beta\leq \frac{N}{N-2}\alpha$, we have $-2<\gamma\leq 0$.

    A direct calculation indicates
    \begin{align}\label{psle}
    \int_{\mathbb{R}^N}
    |x|^{\beta}|u|^{2^{**}_{\alpha,\beta}} \mathrm{d}x
    =\int_{\mathbb{R}^N}
    |x|^{\gamma}|v|^{2^{**}_{0,\gamma}} \mathrm{d}x,
    \end{align}
    and
    \begin{align}\label{psny}
    & \int_{\mathbb{R}^N}|x|^{-\beta}|\mathrm{div} (|x|^{\alpha}\nabla u)|^2 \mathrm{d}x
    \nonumber\\
    & = \int_{\mathbb{R}^N}
    |x|^{2(\eta+\alpha)-\beta}[\Delta v +(2\eta+\alpha)|x|^{-2}(x\cdot \nabla v)
    +\eta(N+\alpha+\eta-2)|x|^{-2}v]^2 \mathrm{d}x
    \nonumber\\
    & = \int_{\mathbb{R}^N}\frac{|\Delta v|^2}{|x|^\gamma} \mathrm{d}x
    +\eta^2(N+\alpha+\eta-2)^2
    \int_{\mathbb{R}^N}\frac{v^2}{|x|^{\gamma+4}}\mathrm{d}x
    +(2\eta+\alpha)^2\int_{\mathbb{R}^N}\frac{(x\cdot \nabla v)^2}{|x|^{\gamma+4}}\mathrm{d}x
    \nonumber\\
    &\quad +2(2\eta+\alpha)\int_{\mathbb{R}^N}\frac{(x\cdot \nabla v)\Delta v}{|x|^{\gamma+2}}\mathrm{d}x
    +2\eta(N+\alpha+\eta-2)
    \int_{\mathbb{R}^N}\frac{v\Delta v}{|x|^{\gamma+2}}\mathrm{d}x
    \nonumber\\
    &\quad +2(2\eta+\alpha)\eta(N+\alpha+\eta-2)
    \int_{\mathbb{R}^N}\frac{v(x\cdot \nabla v)}{|x|^{\gamma+4}}\mathrm{d}x.
    \end{align}
    Now, we calculate those integrals in last term as the following:
    \begin{align}\label{psny1}
    \int_{\mathbb{R}^N}\frac{v(x\cdot \nabla v)}{|x|^{\gamma+4}}\mathrm{d}x
    & =\frac{1}{2}\int_{\mathbb{R}^N}\frac{x\cdot \nabla (v^2)}{|x|^{\gamma+4}}\mathrm{d}x
    =-\frac{1}{2}\int_{\mathbb{R}^N}v^2\mathrm{div}
    \left(\frac{x}{|x|^{\gamma+4}}\right)\mathrm{d}x
    \nonumber\\
    & = -\frac{N-4-\gamma}{2}\int_{\mathbb{R}^N}\frac{v^2}{|x|^{\gamma+4}}
    \mathrm{d}x,
    \end{align}
    and
    \begin{align}\label{psny2}
    \int_{\mathbb{R}^N}\frac{v\Delta v}{|x|^{\gamma+2}}\mathrm{d}x
    & =-\int_{\mathbb{R}^N}\nabla v\cdot\nabla\left(\frac{v}{|x|^{\gamma+2}}\right)\mathrm{d}x
    = (\gamma+2)\int_{\mathbb{R}^N}\frac{v(x\cdot \nabla v)}{|x|^{\gamma+4}}\mathrm{d}x
    -\int_{\mathbb{R}^N}\frac{|\nabla v|^2}{|x|^{\gamma+2}}
    \mathrm{d}x
    \nonumber\\
    & = -\frac{(N-4-\gamma)(\gamma+2)}{2}
    \int_{\mathbb{R}^N}\frac{v^2}{|x|^{\gamma+4}}
    \mathrm{d}x
    -\int_{\mathbb{R}^N}\frac{|\nabla v|^2}{|x|^{\gamma+2}}
    \mathrm{d}x,
    \end{align}
    furthermore,
    \begin{align}\label{psny3}
    \int_{\mathbb{R}^N}\frac{(x\cdot \nabla v)\Delta v}{|x|^{\gamma+2}}\mathrm{d}x
    & =-\int_{\mathbb{R}^N}\nabla v\cdot\nabla\left(\frac{x\cdot \nabla v}{|x|^{\gamma+2}}\right)\mathrm{d}x
    \nonumber\\
    & = -\sum^N_{i,j=1}
    \int_{\mathbb{R}^N}\left[
    \frac{\delta_{ij}}{|x|^{\gamma+2}}\frac{\partial v}{\partial x_i}
    \frac{\partial v}{\partial x_j}
    +\frac{x_j}{|x|^{\gamma+2}}\frac{\partial^2 v}{\partial x_i\partial x_j}
    \frac{\partial v}{\partial x_i}
    -\frac{(\gamma+2)x_ix_j}{|x|^{\gamma+4}}\frac{\partial v}{\partial x_i}
    \frac{\partial v}{\partial x_j}
    \right]
    \mathrm{d}x
    \nonumber\\
    & = -\int_{\mathbb{R}^N}\frac{|\nabla v|^2}{|x|^{\gamma+2}}
    \mathrm{d}x
    +\frac{N-2-\gamma}{2}\int_{\mathbb{R}^N}\frac{|\nabla v|^2}{|x|^{\gamma+2}}
    \mathrm{d}x
    + (\gamma+2)\int_{\mathbb{R}^N}\frac{(x\cdot\nabla v)^2}{|x|^{\gamma+4}}
    \mathrm{d}x
    \nonumber\\
    & = \frac{N-4-\gamma}{2}\int_{\mathbb{R}^N}\frac{|\nabla v|^2}{|x|^{\gamma+2}}
    \mathrm{d}x
    + (\gamma+2)\int_{\mathbb{R}^N}\frac{(x\cdot\nabla v)^2}{|x|^{\gamma+4}}
    \mathrm{d}x.
    \end{align}
    Therefore, putting \eqref{psny1}, \eqref{psny2} and \eqref{psny3} into \eqref{psny} we have
    \begin{align*}
    & (2\eta+\alpha)^2\int_{\mathbb{R}^N}\frac{(x\cdot \nabla v)^2}{|x|^{\gamma+4}}\mathrm{d}x
    +2(2\eta+\alpha)\int_{\mathbb{R}^N}\frac{(x\cdot \nabla v)\Delta v}{|x|^{\gamma+2}}\mathrm{d}x
    \nonumber\\
    &\quad +2\eta(N+\alpha+\eta-2)
    \int_{\mathbb{R}^N}\frac{v\Delta v}{|x|^{\gamma+2}}\mathrm{d}x
    +2(2\eta+\alpha)\eta(N+\alpha+\eta-2)
    \int_{\mathbb{R}^N}\frac{v(x\cdot \nabla v)}{|x|^{\gamma+4}}\mathrm{d}x
    \\
    &= [(2\eta+\alpha)(N-4-\gamma)-2\eta(N+\alpha+\eta-2)]
    \int_{\mathbb{R}^N}\frac{|\nabla v|^2}{|x|^{\gamma+2}}\mathrm{d}x
    \\
    &\quad+ (2\eta+\alpha)(2\eta+\alpha+2\gamma+4)
    \int_{\mathbb{R}^N}\frac{(x\cdot\nabla v)^2}{|x|^{\gamma+4}} \mathrm{d}x\\
    &\quad -\eta(N+\alpha+\eta-2)(N-4-\gamma)(2\eta+\alpha+\gamma+2)
    \int_{\mathbb{R}^N}\frac{v^2}{|x|^{\gamma+4}}\mathrm{d}x
    \\
    & \geq [(2\eta+\alpha)(N+2\eta+\alpha+\gamma)-2\eta(N+\alpha+\eta-2)]
    \int_{\mathbb{R}^N}\frac{|\nabla v|^2}{|x|^{\gamma+2}}\mathrm{d}x
    \\
    &\quad-\eta(N+\alpha+\eta-2)(N-4-\gamma)(2\eta+\alpha+\gamma+2)
    \int_{\mathbb{R}^N}\frac{v^2}{|x|^{\gamma+4}}\mathrm{d}x,
    \end{align*}
    due to $(x\cdot\nabla v)^2\leq |x|^2|\nabla v|^2$ and
    $$(2\eta+\alpha)(2\eta+\alpha+2\gamma+4)
    =\frac{\alpha(2-\alpha+\beta)^2[2(N-2)+\alpha]}{(N-2+\alpha)^2}<0
    $$
    for $\beta>\alpha-2$ and $2-N<\alpha<0$ (note that the assumption $\alpha<0$ plays a crucial role), furthermore, the equality holds if and only if $v$ is radial.
    Thus from \eqref{psny} we have
    \begin{align}\label{psnb}
    & \int_{\mathbb{R}^N}|x|^{-\beta}|\mathrm{div} (|x|^{\alpha}\nabla u)|^2 \mathrm{d}x
    \geq \int_{\mathbb{R}^N}\frac{|\Delta v|^2}{|x|^\gamma} \mathrm{d}x
    \nonumber\\
    &\quad+ \left[\eta^2(N+\alpha+\eta-2)^2
    -\eta(N+\alpha+\eta-2)(N-4-\gamma)(2\eta+\alpha+\gamma+2)\right]
    \int_{\mathbb{R}^N}\frac{v^2}{|x|^{\gamma+4}}\mathrm{d}x
    \nonumber\\
    &\quad+ [(2\eta+\alpha)(N+2\eta+\alpha+\gamma)-2\eta(N+\alpha+\eta-2)]
    \int_{\mathbb{R}^N}\frac{|\nabla v|^2}{|x|^{\gamma+2}}\mathrm{d}x,
    \end{align}
    and the equality holds if and only if $v$ is radial (so does $u$ due to $u(x)=|x|^\eta v(x)$).

    Now, let us make the change
    \begin{align}\label{defmu}
    \alpha=\frac{(2-N)(\mu-\gamma)}{N-4-\gamma},
    \end{align}
    which implies $\gamma<\mu<N-4$ due to $2-N<\alpha<0$. Since $\eta=-\frac{\alpha (N-4+2\alpha-\beta)}{2(N-2+\alpha)}$, then with tedious calculations, it can be verified that
    \begin{align}\label{vecgm1}
    (2\eta+\alpha)(N+2\eta+\alpha+\gamma)-2\eta(N+\alpha+\eta-2)
    =-C_{\gamma,\mu,1},
    \end{align}
    and
    \begin{align}\label{vecgm2}
    \eta^2(N+\alpha+\eta-2)^2
    -\eta(N+\alpha+\eta-2)(N-4-\gamma)(2\eta+\alpha+\gamma+2)
    =C_{\gamma,\mu,2},
    \end{align}
    where $C_{\gamma,\mu,1}$ and $C_{\gamma,\mu,2}$ are given in \eqref{RSisc}. We verify these two identities in Remark \ref{remvecgm}. Thus, from \eqref{psle}, \eqref{psnb} and Theorem \ref{thmwrsi}, we deduce
    \begin{align}\label{2cknp}
    \int_{\mathbb{R}^N}|x|^{-\beta}|\mathrm{div} (|x|^{\alpha}\nabla u)|^2 \mathrm{d}x
    & \geq \int_{\mathbb{R}^N}\frac{|\Delta v|^2}{|x|^{\gamma}} \mathrm{d}x
    -C_{\gamma,\mu,1}\int_{\mathbb{R}^N}\frac{|\nabla v|^2}{|x|^{\gamma+2}} \mathrm{d}x
    +C_{\gamma,\mu,2}\int_{\mathbb{R}^N}\frac{|v|^2}{|x|^{\gamma+4}} \mathrm{d}x
    \nonumber\\
    &\geq \left(\frac{N-4-\mu}{N-4-\gamma}\right)^{3+\frac{2}{2^{**}_{0,\gamma}}}
    \mathcal{S}_{N,\gamma}
    \left(\int_{\mathbb{R}^N}
    |x|^{\gamma}|v|^{2^{**}_{0,\gamma}} \mathrm{d}x\right)^{\frac{2}{2^{**}_{0,\gamma}}}
    \nonumber\\
    &= \left(\frac{N-4-\mu}{N-4-\gamma}\right)^{3+\frac{2}{2^{**}_{0,\gamma}}}
    \mathcal{S}_{N,\gamma}
    \left(\int_{\mathbb{R}^N}
    |x|^{\beta}|u|^{2^{**}_{\alpha,\beta}} \mathrm{d}x
    \right)^{\frac{2}{2^{**}_{\alpha,\beta}}}.
    \end{align}
    Note that the first equality holds if $v$ is radial, and the second equality holds if and only if $v(x)=A|x|^{-\frac{\mu-\gamma}{2}}
    \left(\lambda+|x|^{\frac{(2+\gamma)(N-4-\mu)}{N-4-\gamma}}
    \right)^{-{\frac{N-4-\gamma}{2+\gamma}}}$ for $A\in\mathbb{R}$ and $\lambda>0$. Therefore, \eqref{2cknp} indicates that for all $u\in \mathcal{D}^{2,2}_{\alpha,\beta}(\mathbb{R}^N)$,
    \begin{align*}
    \int_{\mathbb{R}^N}|x|^{-\beta}|\mathrm{div} (|x|^{\alpha}\nabla u)|^2 \mathrm{d}x
    \geq \left(\frac{N-4-\mu}{N-4-\gamma}\right)^{3+\frac{2}{2^{**}_{0,\gamma}}}
    \mathcal{S}_{N,\gamma}
    \left(\int_{\mathbb{R}^N}
    |x|^{\beta}|u|^{2^{**}_{\alpha,\beta}} \mathrm{d}x
    \right)^{\frac{2}{2^{**}_{\alpha,\beta}}}.
    \end{align*}
    moreover, the constant $\left(\frac{N-4-\mu}{N-4-\gamma}\right)^{3+\frac{2}{2^{**}_{0,\gamma}}}
    \mathcal{S}_{N,\gamma}$ is sharp and equality holds if and only if
    \[
    u(x)=A|x|^\eta|x|^{-\frac{\mu-\gamma}{2}}
    \left(\lambda+|x|^{\frac{(2+\gamma)(N-4-\mu)}{N-4-\gamma}}
    \right)^{-{\frac{N-4-\gamma}{2+\gamma}}}
    =A(\lambda+|x|^{2+\beta-\alpha})
    ^{-\frac{N-4+2\alpha-\beta}{2+\beta-\alpha}}.
    \]
    Thus, we obtain the sharp second-order CKN type inequality \eqref{cknrs} with
    \begin{align*}
    \mathcal{S}_{N,\alpha,\beta}
    & =\left(\frac{N-4-\mu}{N-4-\gamma}\right)^{3+\frac{2}{2^{**}_{0,\gamma}}}
    \mathcal{S}_{N,\gamma}
    \\
    & = \left(\frac{N-4-\mu}{N-4-\gamma}\right)^{3+\frac{2}{2^{**}_{0,\gamma}}}
    \left(\frac{2}{2+\gamma}\right)
    ^{\frac{2(2+\gamma)}{N+\gamma}-4}
    \left(\frac{2\pi^{\frac{N}{2}}}{\Gamma(\frac{N}{2})}\right)
    ^{\frac{2(2+\gamma)}{N+\gamma}}
    \mathcal{B}\left(\frac{2(N+\gamma)}{2+\gamma}\right)
    \\
    & =\left(\frac{2}{2+\beta-\alpha}\right)
    ^{\frac{2(2+\beta-\alpha)}{N+\beta}-4}
    \left(\frac{2\pi^{\frac{N}{2}}}{\Gamma(\frac{N}{2})}\right)
    ^{\frac{2(2+\beta-\alpha)}{N+\beta}}
    \mathcal{B}\left(\frac{2(N+\beta)}{2+\beta-\alpha}\right),
    \end{align*}
    where $\mathcal{B}(M)=(M-4)(M-2)M(M+2)
    \left[\Gamma^2(\frac{M}{2})/(2\Gamma(M))\right]^{\frac{4}{M}}$ and $\Gamma$ is the Gamma function. Here, by the choices of $\gamma$ in \eqref{defgm} and $\mu$ in \eqref{defmu}, it is easy to verify that
    \begin{align*}
    & 3+\frac{2}{2^{**}_{0,\gamma}}
    =4-\frac{2(2+\gamma)}{N+\gamma}
    =4-\frac{2(2+\beta-\alpha)}{N+\beta},
    \\
    & \frac{2}{2+\gamma}\cdot\frac{N-4-\gamma}{N-4-\mu}
    =\frac{2}{2+\beta-\alpha},
    \\
    & \frac{2(N+\gamma)}{2+\gamma}=\frac{2(N+\beta)}{2+\beta-\alpha}.
    \end{align*}
    Now, the proof of Theorem \ref{thmcknrs} is completed.
\qed

\begin{remark}\label{remvecgm}\rm
In this remark, we will show the two identities \eqref{vecgm1} and \eqref{vecgm2}. Keeping in mind that:
    \begin{align*}
    \eta & =-\frac{\alpha (N-4+2\alpha-\beta)}{2(N-2+\alpha)},
    \\
    \gamma & =\frac{\beta(N-2)-\alpha N}{N-2+\alpha},
    \\
    \alpha & =\frac{(2-N)(\mu-\gamma)}{N-4-\gamma}\Leftrightarrow
    \mu=\frac{N-4-\gamma}{2-N}\alpha+\gamma.
    \end{align*}

    Firstly, we prove \eqref{vecgm1}. Putting $\eta$ and $\gamma$ into formulas,
    \begin{align*}
    & (2\eta+\alpha)(N+2\eta+\alpha+\gamma)-2\eta(N+\alpha+\eta-2)
    \nonumber\\
    & = \left[-\frac{\alpha (N-4+2\alpha-\beta)}{N-2+\alpha}+\alpha\right]
    \left[N-\frac{\alpha (N-4+2\alpha-\beta)}{N-2+\alpha}+\alpha
        +\frac{\beta(N-2)-\alpha N}{N-2+\alpha}\right]
    \nonumber\\
    & \quad + \frac{\alpha (N-4+2\alpha-\beta)}{N-2+\alpha}
    \left[N+\alpha-2-\frac{\alpha (N-4+2\alpha-\beta)}{2(N-2+\alpha)}\right]
    \nonumber\\
    & = \frac{\alpha(2+\beta-\alpha)\left[-\alpha^2+(2+\beta)\alpha
    +(N-2)(N+\beta)\right]}
    {(N-2+\alpha)^2}
    \nonumber\\
    & \quad +\frac{\alpha(N-4+2\alpha-\beta)
    \left[(3N-4+\beta)\alpha+2(N-2)^2\right]}
    {2(N-2+\alpha)^2}
    \nonumber\\
    & = \frac{\alpha}
    {2(N-2+\alpha)^2}\Big\{2(2+\beta-\alpha)[-\alpha^2+(2+\beta)\alpha
    +(N-2)(N+\beta)]
    \nonumber\\
    & \quad\quad+(N-4+2\alpha-\beta)[(3N-4+\beta)\alpha+2(N-2)^2]
    \Big\}
    \nonumber\\
    & = \frac{\alpha}
    {2(N-2+\alpha)^2}\Big\{2\alpha^3+2(3N-8-\beta)\alpha^2
    +[\beta^2-4(N-3)\beta+5N^2-28N+40]\alpha
    \nonumber\\
    & \quad\quad
    +2(N-2)[\beta^2+4\beta+N^2-4N+8]
    \Big\},
    \end{align*}
    furthermore, also putting $\mu$ into formulas,
    \begin{align*}
    C_{\gamma,\mu,1}
    & =\frac{N^2-4N+8+\gamma^2+4\gamma}{2(N-4-\gamma)^2}(\mu-\gamma)
    [2(N-4-\gamma)-(\mu-\gamma)]
    \nonumber\\
    & = -\frac{(N-2)^2+(\gamma+2)^2}{2(N-2)^2}\alpha
    [2(N-2)+\alpha]
    \nonumber\\
    & = \frac{\alpha[2(N-2)+\alpha]
    \left[(N-2+\alpha)^2+(2+\beta-\alpha)^2\right]}{2(N-2+\alpha)^2}
    \nonumber\\
    & = \frac{-\alpha}
    {2(N-2+\alpha)^2}\Big\{2\alpha^3+2(3N-8-\beta)\alpha^2
    +[\beta^2-4(N-3)\beta+5N^2-28N+40]\alpha
    \nonumber\\
    & \quad\quad
    +2(N-2)[\beta^2+4\beta+N^2-4N+8]
    \Big\}.
    \end{align*}
    Thus \eqref{vecgm1} holds.

    Secondly, we prove \eqref{vecgm2}. Putting $\eta$ and $\gamma$ into formulas,
    \begin{align*}
    & \eta^2(N+\alpha+\eta-2)^2
    -\eta(N+\alpha+\eta-2)(N-4-\gamma)(2\eta+\alpha+\gamma+2)
    \nonumber\\
    & = \frac{\alpha(N-4+2\alpha-\beta)^2}{16(N-2+\alpha)^2}
    \left[
    (3N-4+\beta)\alpha+2(N-2)^2
    \right]
    \nonumber\\
    & \quad\quad\times\left[
    -(N-4-\beta)\alpha^2-2(N-2)(N-6-2\beta)\alpha+4(N-2)^2(\beta+2)
    \right]
    \nonumber\\
    & = \frac{-\alpha(N-4+2\alpha-\beta)^2}{16(N-2+\alpha)^2}
    \Big\{
    (3N-4+\beta)(N-4-\beta)\alpha^3
    \nonumber\\
    & \quad\quad-4(N-2)[\beta^2+(3N-2)\beta-2(N^2-7N+8)]\alpha^2
    \nonumber\\
    & \quad\quad
    -4(N-2)^2\left[\beta^2+(5N-6)\beta-(N^2-14N+20)\alpha
    \right]
    -8(N-2)^4(\beta+2)
    \Big\}.
    \end{align*}
    Furthermore, by the choices of $\gamma$ and $\mu$,
    \begin{align*}
    C_{\gamma,\mu,2}& =\frac{N^2-4N+8+\gamma^2+4\gamma}{8(N-4-\gamma)^2}(\mu-\gamma)^2
    [2(N-4-\gamma)-(\mu-\gamma)]^2
    \\
    &\quad-\frac{(\mu-\gamma)^2[2(N-2)-(\mu-\gamma)]^2}{16}
    \\
    &\quad-\frac{(\mu-\gamma)[2(N-2)-(\mu-\gamma)]
    (N-4-\mu)(2+\gamma)}{4}
    \\
    & = \frac{\alpha(N-4-\gamma)^2}{16(N-2)^4}
    \Big\{
    2(N^2-4N+8+\gamma^2+4\gamma)\alpha[2(N-2)+\alpha]^2
    \\
    &\quad\quad+[2(N-2)^2+(N-4-\gamma)\alpha]
    \\
    &\quad\quad\quad
    \times\left\{
    -[2(N-2)^2+(N-4-\gamma)\alpha]\alpha
    +4(N-2)(N-2+\alpha)(2+\gamma)
    \right\}
    \Big\}.
    \end{align*}
    Note that
    \begin{align*}
    \frac{(N-4-\gamma)^2}{(N-2)^4}=\frac{[(N-2)(N-4+2\alpha-\beta)]^2}
    {(N-2+\alpha)^2(N-2)^4}=\frac{(N-4+2\alpha-\beta)^2}{(N-2+\alpha)^2},
    \end{align*}
    and
    \begin{align*}
    N^2-4N+8+\gamma^2+4\gamma
    =\frac{(N-2)^2}{(N-2+\alpha)^2}
    \left[
    2\alpha^2+2(N-4-\beta)\alpha+N^2-4N+8+\beta^2+4\beta
    \right],
    \end{align*}
    moreover,
    \begin{align*}
    & [2(N-2)^2+(N-4-\gamma)\alpha]\left\{
    [2(N-2)^2+(N-4-\gamma)\alpha]\alpha
    -4(N-2)(N-2+\alpha)(2+\gamma)
    \right\}
    \\
    & = \frac{(N-2)^2}{(N-2+\alpha)^2}
    \left[
    2\alpha^2+(3N-8-\beta)\alpha+2(N-2)^2
    \right]
    \\
    & \quad\quad \times
    \left[
    2\alpha^3+(7N-16-\beta)\alpha^2+2(N-2)(3N-10-2\beta)\alpha
    -4(N-2)^2(2+\beta)
    \right].
    \end{align*}
    Therefore,
    \begin{align*}
    C_{\gamma,\mu,2}
    & = \frac{\alpha(N-4+2\alpha-\beta)^2}{16(N-2+\alpha)^2}
    \Big\{
    2\alpha[\alpha^2+4(N-2)\alpha+4(N-2)^2]
    \\
    &\quad\quad\quad
    \times[2\alpha^2+2(N-4-\beta)\alpha+N^2-4N+8+\beta^2+4\beta]
    \\
    &\quad\quad
    -[2\alpha^2+(3N-8-\beta)\alpha+2(N-2)^2]
    \\
    &\quad\quad\quad
    \times[2\alpha^3+(7N-16-\beta)\alpha^2
    +2(N-2)(3N-10-2\beta)\alpha
    -4(N-2)^2(2+\beta)]
    \Big\}
    \\
    & = \frac{-\alpha(N-4+2\alpha-\beta)^2}{16(N-2+\alpha)^2}
    \Big\{
    (3N-4+\beta)(N-4-\beta)\alpha^3
    \nonumber\\
    & \quad\quad-4(N-2)[\beta^2+(3N-2)\beta-2(N^2-7N+8)]\alpha^2
    \nonumber\\
    & \quad\quad
    -4(N-2)^2\left[\beta^2+(5N-6)\beta-(N^2-14N+20)\alpha
    \right]
    -8(N-2)^4(\beta+2)
    \Big\}.
    \end{align*}
    Thus \eqref{vecgm2} holds.
\end{remark}

\section{{\bfseries Sharp singular second-order CKN type inequality}}\label{sectcknpss}

At first, we establish the equivalence relationship of two second-order CKN type inequalities \eqref{cknrs} and \eqref{cknrss}.

    \begin{lemma}\label{lemertwo}
    Assume that $N\geq 5$, $\alpha>2-N$ and $\frac{N-4}{N-2}\alpha-4 \leq \beta\leq \alpha -2$. For each $u\in \mathcal{D}^{2,2}_{\alpha,\beta}(\mathbb{R}^N\setminus\{0\})$, let $u(x)=|x|^{\eta}v(x)$ with $\eta=2+\beta-\alpha$, then
    \begin{align}\label{ertwo}
    \int_{\mathbb{R}^N}|x|^{\xi}
    |u|^{\bar{2}^{**}_{\alpha,\beta}} \mathrm{d}x
    & =\int_{\mathbb{R}^N}|x|^{\bar{\beta}}
    |v|^{2^{**}_{\alpha,\bar{\beta}}} \mathrm{d}x,
    \\ \label{ertwo2}
    \int_{\mathbb{R}^N}|x|^{-\beta}|\mathrm{div} (|x|^{\alpha}\nabla u)|^2 \mathrm{d}x
    & = \int_{\mathbb{R}^N}|x|^{-\bar{\beta}}|\mathrm{div} (|x|^{\alpha}\nabla v)|^2 \mathrm{d}x,
    \end{align}
    where $\bar{2}^{**}_{\alpha,\beta}
    =\frac{2(N+\xi)}{N+2\alpha-\beta-4}$ with $\xi=\frac{(N+2\alpha-\beta-4)^2}{N+\beta}-N$, and $\bar{\beta}=2\alpha-\beta-4$.
    \end{lemma}

    \begin{proof}
    By the choices of $\eta$ and $\bar{\beta}$, it is easy to verify that
    \begin{align*}
    \bar{2}^{**}_{\alpha,\beta}
    & =\frac{2(N+\xi)}{N+2\alpha-\beta-4}
    =\frac{2(N+2\alpha-\beta-4)}{N+\beta}
    =2^{**}_{\alpha,\bar{\beta}}
    =\frac{2(N+\bar{\beta})}{N+2\alpha-\bar{\beta}-4},
    \\
    \bar{\beta}
    & =\xi+\eta\cdot\bar{2}^{**}_{\alpha,\beta},
    \end{align*}
    thus \eqref{ertwo} holds.

    Next, we prove \eqref{ertwo2}. Since $u(x)=|x|^{\eta}v(x)$ with $\eta=2+\beta-\alpha$, then as in the proof of Theorem \ref{thmcknrs} we have
    \begin{align}\label{ertwo2t}
    & \int_{\mathbb{R}^N}|x|^{-\beta}|\mathrm{div} (|x|^{\alpha}\nabla u)|^2 \mathrm{d}x
    \nonumber\\
    & = \int_{\mathbb{R}^N}
    |x|^{2(\eta+\alpha)-\beta}[\Delta v +(2\eta+\alpha)|x|^{-2}(x\cdot \nabla v)
    +\eta(N+\alpha+\eta-2)|x|^{-2}v]^2 \mathrm{d}x
    \nonumber\\
    & = \int_{\mathbb{R}^N}\frac{|\Delta v|^2}{|x|^\gamma} \mathrm{d}x
    +(2\eta+\alpha)(2\eta+\alpha+2\gamma+4)
    \int_{\mathbb{R}^N}\frac{(x\cdot \nabla v)^2}{|x|^{\gamma+4}}\mathrm{d}x
    \nonumber\\
    &\quad +\left[(2\eta+\alpha)(N-4-\gamma)
    -2\eta(N+\alpha+\eta-2)\right]\int_{\mathbb{R}^N}\frac{|\nabla v|^2}{|x|^{\gamma+2}}\mathrm{d}x
    \nonumber\\
    &\quad +\left[\eta^2(N+\alpha+\eta-2)^2
    -\eta(N+\alpha+\eta-2)(N-4-\gamma)(2\eta+\alpha+\gamma+2)\right]
    \int_{\mathbb{R}^N}\frac{|v|^2}{|x|^{\gamma+4}}\mathrm{d}x,
    \end{align}
    where $\gamma=\beta-2(\eta+\alpha)$, and
    \begin{align}\label{ertwo2t2}
    \int_{\mathbb{R}^N}|x|^{-\bar{\beta}}|\mathrm{div} (|x|^{\alpha}\nabla v)|^2 \mathrm{d}x
    & = \int_{\mathbb{R}^N}\frac{|\Delta v|^2}{|x|^{\bar{\beta}-2\alpha}} \mathrm{d}x
    +\alpha(2\bar{\beta}-3\alpha+4)
    \int_{\mathbb{R}^N}\frac{(x\cdot \nabla v)^2}{|x|^{\bar{\beta}-2\alpha+4}}\mathrm{d}x
    \nonumber\\
    &\quad +\alpha(N-4+2\alpha-\bar{\beta})\int_{\mathbb{R}^N}\frac{|\nabla v|^2}{|x|^{\bar{\beta}-2\alpha+2}}\mathrm{d}x.
    \end{align}
    It is easy to verify that
    \begin{align*}
    & \gamma=\beta-2(\eta+\alpha)=\bar{\beta}-2\alpha,
    \\
    &(2\eta+\alpha)(2\eta+\alpha+2\gamma+4)
    =\alpha(2\bar{\beta}-3\alpha+4),
    \\
    & (2\eta+\alpha)(N-4-\gamma)
    -2\eta(N+\alpha+\eta-2)=\alpha(N-4+2\alpha-\bar{\beta}),
    \\
    & \eta^2(N+\alpha+\eta-2)^2
    -\eta(N+\alpha+\eta-2)(N-4-\gamma)(2\eta+\alpha+\gamma+2)=0,
    \end{align*}
    therefore, \eqref{ertwo2} holds.
    \end{proof}

    \begin{remark}\label{remefk}\rm
    The key step of showing above equivalent form is choosing suitable $\eta$. The idea is that, comparing \eqref{ertwo2t} and \eqref{ertwo2t2}, we notice that it must satisfy
    \[
    \eta^2(N+\alpha+\eta-2)^2
    -\eta(N+\alpha+\eta-2)(N-4-\gamma)(2\eta+\alpha+\gamma+2)=0,
    \]
    where $\gamma=\beta-2(\eta+\alpha)$, then
    \[
    \eta_1=0,\quad \eta_2=2-N-\alpha,\quad \eta_3=2+\beta-\alpha,\quad \eta_4=\beta-2\alpha+4-N.
    \]
    Furthermore, in order to establish \eqref{ertwo2}, we need $\bar{2}^{**}_{\alpha,\beta}=2^{**}_{\alpha,\bar{\beta}}$, that is,
    \[
    \frac{2(N+2\alpha-\beta-4)}{N+\beta}
    =\frac{2(N+\bar{\beta})}{N+2\alpha-\bar{\beta}-4},
    \]
    which implies $\bar{\beta}=2\alpha-\beta-4$. Then by $\beta-2(\eta+\alpha)=\bar{\beta}-2\alpha$, we should choose $\eta=2+\beta-\alpha$.
    \end{remark}

    Now, based on the symmetry result of Theorem \ref{thmcknrs}, we are ready to show the symmetry of extremal functions for the singular second-order CKN type inequality \eqref{cknrss}.

\vskip0.25cm

\noindent{\bf \em Proof of Theorem \ref{thmcknrss}}. Since $\frac{N-4}{N-2}\alpha-4 \leq \beta<\alpha -2$, then for $\bar{\beta}=2\alpha-\beta-4$ we have
\[
\alpha-2<\bar{\beta}\leq \frac{N}{N-2}\alpha.
\]
Then by Lemma \ref{lemertwo} and Theorem \ref{thmcknrs}, we can directly obtain the singular second-order CKN type inequality \eqref{cknrss} with sharp constant
    \begin{align*}
    \overline{\mathcal{S}}_{N,\alpha,\beta}
    & =\mathcal{S}_{N,\alpha,\bar{\beta}}
    \\
    & =\left(\frac{2}{2+\bar{\beta}-\alpha}\right)
    ^{\frac{2(2+\bar{\beta}-\alpha)}{N+\bar{\beta}}-4}
    \left(\frac{2\pi^{\frac{N}{2}}}{\Gamma(\frac{N}{2})}\right)
    ^{\frac{2(2+\bar{\beta}-\alpha)}{N+\bar{\beta}}}
    \mathcal{B}\left(\frac{2(N+\bar{\beta})}{2+\bar{\beta}-\alpha}\right)
    \\
    & =\left(\frac{2}{\alpha-\beta-2}\right)
    ^{\frac{2(\alpha-\beta-2)}{N+2\alpha-\beta-4}-4}
    \left(\frac{2\pi^{\frac{N}{2}}}{\Gamma(\frac{N}{2})}\right)
    ^{\frac{2(\alpha-\beta-2)}{N+2\alpha-\beta-4}}
    \mathcal{B}\left(\frac{2(N+2\alpha-\beta-4)}{\alpha-\beta-2}\right),
    \end{align*}
    where $\mathcal{B}(M)=(M-4)(M-2)M(M+2)
    \left[\Gamma^2(\frac{M}{2})/(2\Gamma(M))\right]^{\frac{4}{M}}$, and equality in \eqref{cknrss} holds if and only if
    \begin{align*}
    u(x)=A|x|^\eta(\lambda+|x|^{2+\bar{\beta}-\alpha})
    ^{-\frac{N-4+2\alpha-\bar{\beta}}{2+\bar{\beta}-\alpha}}
    =A|x|^{2+\beta-\alpha}(1+|x|^{\alpha-\beta-2})
    ^{-\frac{N+\beta}{\alpha-\beta-2}},
    \end{align*}
    for $A\in\mathbb{R}$ and $\lambda>0$, where $\eta=2+\beta-\alpha$. Now, the proof of Theorem \ref{thmcknrss} is completed.
\qed

\section{{\bfseries Weak form Hardy-Rellich inequality}}\label{sectwhri}

In this section, we establish the weak form weighted Hardy-Rellich inequality \eqref{hriw} with explicit sharp constant, by following the arguments as those in \cite{GM11,TZ07}, and also \cite{Ca20}.

 \vskip0.25cm

\noindent{\bf \em Proof of Theorem \ref{thmhriw}.} The proof follows in several steps as follows.

 \vskip0.25cm
\noindent{\bfseries Step \uppercase\expandafter{\romannumeral 1}: Spherical coordinates}
 \vskip0.25cm
    For each $u\in C^\infty_0(\mathbb{R}^N)$, let us make the standard spherical decomposition of $u$ as in the proof of Lemma \ref{lemtl}, namely
    \begin{equation}\label{Ppwhl2defvdp}
    u(x)=u(r,\sigma)=\sum^{\infty}_{k=0}f_k(r)\Psi_k(\sigma),
    \end{equation}
    where $r=|x|$ and $\sigma=\frac{x}{|x|}\in \mathbb{S}^{N-1}$, and
    \begin{equation*}
    f_k(r)=\int_{\mathbb{S}^{N-1}}u(r,\sigma)\Psi_k(\sigma)
    \mathrm{d}\sigma.
    \end{equation*}
    Here $\Psi_k(\sigma)$ denotes the $k$-th spherical harmonic. It is well known that
    \begin{align}\label{Ppwhl2deflklwp}
    \Delta (f_k(r)\Psi_k(\sigma))
    =\left(f''_k+\frac{N-1}{r}f'_k
    -\frac{c_k}{r^2}f_k\right)\Psi_k.
    \end{align}
    It is easy to verify that
    \begin{equation*}
    \frac{\partial (f_k(r)\Psi_k(\sigma))}{\partial x_i}=f'_k\frac{x_i}{r}\Psi_k+\varphi_k
    \frac{\partial\Psi_k}{\partial \sigma_l}\frac{\partial\sigma_l}{\partial x_i},
    \end{equation*}
    hence
    \begin{equation}\label{Ppwhl2deflklnp}
    \begin{split}
    x\cdot\nabla (f_k(r)\Psi_k(\sigma))=\sum^{N}_{i=1}x_i\frac{\partial (f_k(r)\Psi_k(\sigma))}{\partial x_i}=f'_kr\Psi_k+f_k\frac{\partial\Psi_k}{\partial \sigma_l}\sum^{N}_{i=i}\frac{\partial\sigma_l}{\partial x_i}x_i=f'_kr\Psi_k,
    \end{split}
    \end{equation}
    due to $\sum^N_{i=1}\frac{\partial\sigma_l}{\partial x_i}x_i=0$ for every $l$. Furthermore, the functions $f_k(r)\in C^\infty_0([0,\infty))$ satisfying $f_k(r)=O(r^k)$ and $f_k'(r)=O(r^{k-1})$ as $r\to 0$. For simplicity, we also let
    \[
    \int_{\mathbb{S}^{N-1}}|\Psi_k|^2\mathrm{d}\sigma=1,\quad \mbox{for all}\quad k\geq 0,
    \]
    so that
    \begin{align}\label{sc1}
    \int_{\mathbb{R}^N}|x|^{-2a-4}(x\cdot \nabla u)^2 \mathrm{d}x
    =\sum^\infty_{k=0}\int^\infty_0(f_k')^2 r^{N-3-2a}\mathrm{d}r,
    \end{align}
    and
    \begin{align}\label{sc2}
    \int_{\mathbb{R}^{N}}|x|^{-2a}|\Delta u|^2\mathrm{d}x
    & =\sum^\infty_{k=0}\int^\infty_0\left(f''_k+\frac{N-1}{r}f'_k
    -\frac{c_k}{r^2}f_k\right)^2 r^{N-1-2a}\mathrm{d}r
    \nonumber\\
    & =\sum^\infty_{k=0}\bigg[
    \int^\infty_0\left(f_k''\right)^2 r^{N-1-2a}\mathrm{d}r
    \nonumber\\&\quad\quad +[(1+2a)(N-1)+2c_k]\int^\infty_0\left(f_k'\right)^2 r^{N-3-2a}\mathrm{d}r
    \nonumber\\&\quad\quad +c_k[c_k+2(1+a)(N-4-2a)]\int^\infty_0f_k^2 r^{N-5-2a}\mathrm{d}r
    \bigg].
    \end{align}
    We refer to \cite[Section 2.2]{TZ07} for details, and also \cite[Appendix B]{GM11}.

 \vskip0.25cm
\noindent{\bfseries Step \uppercase\expandafter{\romannumeral 2}: Weighted one-dimensional Hardy and Rellich inequalities}
 \vskip0.25cm

 We recall the following Hardy and Rellich type inequalities with sharp constants:
    \begin{align}\label{odwri}
    \int^\infty_0\left(f_k''\right)^2 r^{N-1-2a}\mathrm{d}r
    \geq & \left(\frac{N-2-2a}{2}\right)^2\int^\infty_0\left(f_k'\right)^2 r^{N-3-2a}\mathrm{d}r,\quad \forall k\geq 0,
    \\ \label{odwhi}
    \int^\infty_0\left(f_k'\right)^2 r^{N-3-2a}\mathrm{d}r
    \geq & \left(\frac{N-4-2a}{2}\right)^2\int^\infty_0f_k^2 r^{N-5-2a}\mathrm{d}r,\quad \forall k\geq 1.
    \end{align}
    For the sake of clarity let us give a few lines proof of both \eqref{odwri} and \eqref{odwhi}. Indeed, for each $v\in C^\infty_0(\mathbb{R}^N)$, from the identity
    \begin{align*}
    \int_{\mathbb{R}^N}|x|^{-2-2c}x\cdot \nabla v^2\mathrm{d}x
    =-(N-2-2c)\int_{\mathbb{R}^N}|x|^{-2-2c}v^2\mathrm{d}x
    =2\int_{\mathbb{R}^N}v|x|^{-2-2c}(x\cdot \nabla v)\mathrm{d}x,
    \end{align*}
    for any $c\in\mathbb{R}$, by Cauchy-Schwarz inequality we have
    \begin{align*}
    \int_{\mathbb{R}^N}|x|^{-2-2c}(x\cdot \nabla v)^2\mathrm{d}x
    \geq \left(\frac{N-2-2c}{2}\right)^2
    \int_{\mathbb{R}^N}|x|^{-2-2c}v^2\mathrm{d}x.
    \end{align*}
    Thus, let $v(x)=f_k'(|x|)$ with $c=a$ we deduce \eqref{odwri}, and let $v(x)=f_k(|x|)$ with $c=a+1$ we deduce \eqref{odwhi}. Next, we show the constants in \eqref{odwri} and \eqref{odwhi} are optimal. If $a\neq \frac{N-4}{2}$, it is enough to show that the family of functions
    \begin{align}\label{deftf}
    \left\{f_\epsilon(r):=r^{-(N-4-2a)/2-\epsilon}g(r)\right\}_{\epsilon>0}
    \end{align}
    is a minimizing sequence for $\left(\frac{N-2-2a}{2}\right)^2$ in \eqref{odwri} and also for $\left(\frac{N-4-2a}{2}\right)^2$ in \eqref{odwhi}, where $g$ is a sufficiently smooth function satisfying $g(r)=0$ if $0\leq r\leq 1$ and $g(r)=1$ if $r\geq 2$. To resume, one can easily check that
    \begin{align}\label{fe1}
    \frac{\int^\infty_0\left(f_\epsilon''\right)^2 r^{N-1-2a}\mathrm{d}r}
    {\int^\infty_0\left(f_\epsilon'\right)^2 r^{N-3-2a}\mathrm{d}r}
    & = \frac{\left(-\frac{N-4-2a}{2}-\epsilon\right)^2
    \left(-\frac{N-2-2a}{2}-\epsilon\right)^2
    \left(\frac{2^{-2\epsilon}}{2\epsilon}+O(1)\right)+O(1)}
    {\left(-\frac{N-4-2a}{2}-\epsilon\right)^2
    \left(\frac{2^{-2\epsilon}}{2\epsilon}+O(1)\right)+O(1)}
    \nonumber\\ & \to \left(\frac{N-2-2a}{2}\right)^2,\quad \mbox{as}\ \epsilon\to 0,
    \end{align}
    and
    \begin{align}\label{fehi}
    \frac{\int^\infty_0\left(f_\epsilon'\right)^2 r^{N-3-2a}\mathrm{d}r}
    {\int^\infty_0f_\epsilon^2 r^{N-5-2a}\mathrm{d}r}
    & = \frac{\left(-\frac{N-4-2a}{2}-\epsilon\right)^2
    \left(\frac{2^{-2\epsilon}}{2\epsilon}+O(1)\right)+O(1)}
    {\frac{2^{-2\epsilon}}{2\epsilon}+O(1)}
    \nonumber\\ & \to \left(\frac{N-4-2a}{2}\right)^2,\quad \mbox{as}\ \epsilon\to 0.
    \end{align}
    Note that \eqref{fehi} also holds if $a=\frac{N-4}{2}$.
    While, for the sharpness of inequality \eqref{odwri}, if $a=\frac{N-4}{2}$ then the nontermination in \eqref{fe1} is $0/0$, so we need to consider another sequence
    \begin{align}\label{deftfe}
    \left\{\bar{f}_\epsilon(r):=\int^\infty_r h_\epsilon(t)\mathrm{d}t\right\}_{\epsilon>0}\quad \mbox{with}\quad h_\epsilon(t)=t^{-1-\epsilon}g(t),
    \end{align}
    and it is enough to show that the family of functions $\{\bar{f}_\epsilon\}_{\epsilon>0}$ is a minimizing sequence for $\left(\frac{N-2-2a}{2}\right)^2=1$ in \eqref{odwri}, that is
    \begin{align}\label{fe1e}
    \frac{\int^\infty_0\left(\bar{f}_\epsilon''\right)^2 r^{3}\mathrm{d}r}
    {\int^\infty_0\left(\bar{f}_\epsilon'\right)^2 r\mathrm{d}r}
    & = \frac{\int^\infty_0\left(h_\epsilon'\right)^2 r^{3}\mathrm{d}r}
    {\int^\infty_0 h_\epsilon^2 r\mathrm{d}r}
    = \frac{\left(-1-\epsilon\right)^2
    \left(\frac{2^{-2\epsilon}}{2\epsilon}+O(1)\right)+O(1)}
    {\frac{2^{-2\epsilon}}{2\epsilon}+O(1)}
    \to 1,
    \end{align}
    as $\epsilon\to 0$.

 \vskip0.25cm
\noindent{\bfseries Step \uppercase\expandafter{\romannumeral 3}: Lower bound of optimal constant}
 \vskip0.25cm

    We will make usage of Step \uppercase\expandafter{\romannumeral 1} and Step \uppercase\expandafter{\romannumeral 2} when comparing both integral in \eqref{hriw}. First, we split the term on the right hand side in \eqref{sc2} into the sum $I_1+I_2$ where
    \begin{align*}
    I_1=\sum^\infty_{k=0}\left[
    \int^\infty_0\left(f_k''\right)^2 r^{N-1-2a}\mathrm{d}r
    +(1+2a)(N-1)\int^\infty_0\left(f_k'\right)^2 r^{N-3-2a}\mathrm{d}r
    \right]
    \end{align*}
    denotes the radial part of expansion in \eqref{hriw}, whereas
    \begin{align*}
    I_2=\sum^\infty_{k=1}\left[
    2c_k\int^\infty_0\left(f_k'\right)^2 r^{N-3-2a}\mathrm{d}r
    +c_k[c_k+2(1+a)(N-4-2a)]\int^\infty_0f_k^2 r^{N-5-2a}\mathrm{d}r
    \right]
    \end{align*}
    is its spherical part. Then, from \eqref{odwri} we have
    \begin{align}\label{I1b}
    I_1\geq \left(\frac{N+2a}{2}\right)^2\sum^\infty_{k=0}
    \int^\infty_0\left(f_k'\right)^2 r^{N-3-2a}\mathrm{d}r,
    \end{align}
    and from \eqref{odwhi},
    \begin{align}\label{I2b}
    I_2
    & \geq \sum^\infty_{k=0} c_{k}
    \left[
    c_{k}+\frac{(N-4-2a)^2}{2}+2(1+a)(N-4-2a)\right]
    \int^\infty_0f_{k}^2 r^{N-5-2a}\mathrm{d}r
    \nonumber \\
    & = \sum^\infty_{k=0} c_{k}
    \left[
    c_{k}+\frac{(N-4-2a)(N+2a)}{2}\right]
    \int^\infty_0f_{k}^2 r^{N-5-2a}\mathrm{d}r.
    \end{align}

    Let us first consider $N\geq 2$.
    If
    \[
    N-1+\frac{(N-4-2a)(N+2a)}{2}\geq 0,
    \]
    which implies
    \begin{align}\label{cra1}
    \frac{-2-\sqrt{N^2-2N+2}}{2}\leq a\leq \frac{-2+\sqrt{N^2-2N+2}}{2},
    \end{align}
    then
    \[
    c_k\left[c_k+\frac{(N-4-2a)(N+2a)}{2}\right]\geq 0,\quad \forall k\geq 0,
    \]
    which indicates $I_2\geq 0$ then we always have $C(N,a)\geq \left(\frac{N+2a}{2}\right)^2$. Otherwise, if
    \begin{align}\label{cra2}
    a<\frac{-2-\sqrt{N^2-2N+2}}{2}\quad \mbox{or}\quad \frac{-2+\sqrt{N^2-2N+2}}{2}<a<\frac{N-2}{2},
    \end{align}
    then
    \begin{align*}
    c_{\bar{k}}\left[c_{\bar{k}}\frac{(N-4-2a)(N+2a)}{2}\right]<0,\quad \mbox{for some}\ \bar{k}\geq 1.
    \end{align*}
    Note that
    \[
    \frac{-2-\sqrt{N^2-2N+2}}{2}< \frac{N-4}{2}
    <\frac{-2+\sqrt{N^2-2N+2}}{2}< \frac{N-2}{2},
    \]
    then the assumption \eqref{cra2} implies $a\neq \frac{N-4}{2}$, and from \eqref{odwhi} we have
    \begin{align*}
    & 2c_{\bar{k}}\int^\infty_0\left(f_{\bar{k}}'\right)^2 r^{N-3-2a}\mathrm{d}r
    +c_{\bar{k}}[c_{\bar{k}}+2(1+a)(N-4-2a)]\int^\infty_0f_{\bar{k}}^2 r^{N-5-2a}\mathrm{d}r
    \\
    & \geq \left(\frac{2}{N-4-2a}\right)^2c_{\bar{k}}
    \left[
    c_{\bar{k}}+\frac{(N-4-2a)(N+2a)}{2}\right]
    \int^\infty_0\left(f_{\bar{k}}'\right)^2 r^{N-3-2a}\mathrm{d}r,
    \end{align*}
    thus combining with \eqref{I1b}-\eqref{I2b} we have
    \begin{align*}
    C(N,a)
    & \geq  \left(\frac{N+2a}{2}\right)^2
    +\left(\frac{2}{N-4-2a}\right)^2
    \inf_{k\in\mathbb{N}}
    c_k\left[c_k+\frac{(N-4-2a)(N+2a)}{2}\right]
    \\
    & = \left(\frac{2}{N-4-2a}\right)^2
    \inf\limits_{k\in\mathbb{N}}
    \left(k+\frac{N}{2}+a\right)^2\left(k+\frac{N-4}{2}-a\right)^2.
    \end{align*}

    Then, let us consider the case $N=1$ which is similar to the case $N\geq 2$ with minor changes. If
    \[
    2+\frac{(-3-2a)(1+2a)}{2}\geq 0,\quad a<-\frac{1}{2},
    \]
    which implies
    \begin{align}\label{cra11}
    \frac{-2-\sqrt{5}}{2}\leq a<-\frac{1}{2},
    \end{align}
    then
    \[
    c_k\left[c_k+\frac{(-3-2a)(1+2a)}{2}\right]\geq 0,\quad \forall k\geq 0,
    \]
    thus combining with \eqref{I1b}-\eqref{I2b} we also always have $C(1,a)\geq \left(\frac{1+2a}{2}\right)^2$. Otherwise, if
    \begin{align}\label{cra22}
    a<\frac{-2-\sqrt{5}}{2},
    \end{align}
    then
    \begin{align*}
    c_{\bar{k}}\left[c_{\bar{k}}+\frac{(-3-2a)(1+2a)}{2}\right]<0,\quad \mbox{for some}\ \bar{k}\geq 2,
    \end{align*}
    thus we also have
    \begin{align*}
    C(1,a)
    \geq \left(\frac{2}{3+2a}\right)^2
    \inf\limits_{k\in\mathbb{N}}
    \left(k+\frac{1}{2}+a\right)^2\left(k-\frac{3}{2}-a\right)^2.
    \end{align*}

 \vskip0.25cm
\noindent{\bfseries Step \uppercase\expandafter{\romannumeral 4}: Explicate form of optimal constant}
 \vskip0.25cm

    Now, we consider the sharpness of $C(N,a)$. We claim that for $N\geq 2$,
    \begin{align}\label{defsc1}
    C(N,a)=\left(\frac{N+2a}{2}\right)^2\quad\text{if }\ \frac{-2-\sqrt{N^2-2N+2}}{2}\leq a\leq \frac{-2+\sqrt{N^2-2N+2}}{2},
    \end{align}
    and
    \begin{align}\label{defsc2}
    C(N,a)=\left(\frac{2}{N-4-2a}\right)^2\inf\limits_{k\in\mathbb{N}}
    \left(k+\frac{N}{2}+a\right)^2\left(k+\frac{N-4}{2}-a\right)^2
    \end{align}
    if $a<\frac{-2-\sqrt{N^2-2N+2}}{2}$ or $\frac{-2+\sqrt{N^2-2N+2}}{2}<a<\frac{N-2}{2}$, furthermore,
    \begin{align}\label{defsc11}
    C(1,a)=\left(\frac{1+2a}{2}\right)^2\quad\text{if }\ \frac{-2-\sqrt{5}}{2}\leq a< -\frac{1}{2},
    \end{align}
    and
    \begin{align}\label{defsc21}
    C(1,a)
    = \left(\frac{2}{3+2a}\right)^2
    \inf\limits_{k\in\mathbb{N}}
    \left(k+\frac{1}{2}+a\right)^2\left(k-\frac{3}{2}-a\right)^2
    \end{align}
    if $a<\frac{-2-\sqrt{5}}{2}$.

    We first verify the cases \eqref{defsc2} for $N\geq 2$ and \eqref{defsc21} for $N=1$. Let $k_*\in\mathbb{N}$ such that
    \[
    \left(k_*+\frac{N}{2}+a\right)^2\left(k_*+\frac{N-4}{2}-a\right)^2
    =\inf_{k\in\mathbb{N}}
    \left(k+\frac{N}{2}+a\right)^2\left(k+\frac{N-4}{2}-a\right)^2.
    \]
    In fact, from the above analysis in Step \uppercase\expandafter{\romannumeral 3}, we know $k_*\geq 1$ for $N\geq 2$ and $k_*\geq 2$ for $N=1$.
    It is enough to show that the family of functions
    \[
    \{u^*_\epsilon(x):=f_\epsilon(|x|)
    \Phi_{k_*}(\sigma)\}_{\epsilon>0}
    \]
    is a minimizing sequence for $\left(\frac{2}{N-4-2a}\right)^2
    \left(k_*+\frac{N}{2}+a\right)^2\left(k_*+\frac{N-4}{2}-a\right)^2$,
    where $f_\epsilon$ is given as in \eqref{deftf}, and $\Psi_{k_*}(\sigma)$ denotes the $k_*$-th spherical harmonic satisfying $-\Delta_{\mathbb{S}^{N-1}}\Psi_{k_*}=c_{k_*}\Psi_{k_*}$.
    Then, from \eqref{sc1}-\eqref{sc2}, combining with \eqref{fe1}-\eqref{fehi} we deduce
    \begin{align}\label{scl}
    \frac{\int_{\mathbb{R}^{N}}|x|^{-2a}|\Delta u^*_\epsilon|^2\mathrm{d}x}
    {\int_{\mathbb{R}^N}|x|^{-2a-4}(x\cdot \nabla u^*_\epsilon)^2 \mathrm{d}x}
    & =\frac{\int^\infty_0\left(f_\epsilon''\right)^2 r^{N-1-2a}\mathrm{d}r}
    {\int^\infty_0(f_\epsilon')^2 r^{N-3-2a}\mathrm{d}r}
    +[(1+2a)(N-1)+2c_{k_*}]
    \nonumber\\& \quad +c_{k_*}[c_{k_*}+2(1+a)(N-4-2a)]\frac{\int^\infty_0f_\epsilon^2 r^{N-5-2a}\mathrm{d}r}
    {\int^\infty_0(f_\epsilon')^2 r^{N-3-2a}\mathrm{d}r}
    \nonumber\\ & \underrightarrow{\epsilon\to 0} \left(\frac{N-2-2a}{2}\right)^2
    +[(1+2a)(N-1)+2c_{k_*}]
    \nonumber\\
    &\quad +c_{k_*}[c_{k_*}+2(1+a)(N-4-2a)]\left(\frac{2}{N-4-2a}\right)^2
    \nonumber\\
    & = \left(\frac{2}{N-4-2a}\right)^2\left(\frac{(N-4-2a)(N+2a)}{4}
    +c_{k_*}\right)^2
    \nonumber\\
    & =\left(\frac{2}{N-4-2a}\right)^2
    \left(k_*+\frac{N}{2}+a\right)^2\left(k_*+\frac{N-4}{2}-a\right)^2.
    \end{align}
    Thus the claims \eqref{defsc2} for $N\geq 2$ and \eqref{defsc21} for $N=1$ hold.

    Next, we verify the cases \eqref{defsc1} for $N\geq 2$ and \eqref{defsc11} for $N=1$. We split it into two cases: $a=\frac{N-4}{2}$ and $a\neq \frac{N-4}{2}$. If $a\neq \frac{N-4}{2}$, it is enough to show that the family of functions
    \[
    \{\hat{u}_\epsilon(x):=f_\epsilon(|x|)\}_{\epsilon>0}
    \]
    is a minimizing sequence for $\left(\frac{N+2a}{2}\right)^2$, where $f_\epsilon$ is given as in \eqref{deftf}. Putting the test functions sequence $\{\hat{u}_\epsilon\}_{\epsilon>0}$ into formulas, from \eqref{fe1} we have
    \begin{align*}
    \frac{\int_{\mathbb{R}^{N}}|x|^{-2a}|\Delta \hat{u}_\epsilon|^2\mathrm{d}x}
    {\int_{\mathbb{R}^N}|x|^{-2a-4}(x\cdot \nabla \hat{u}_\epsilon)^2 \mathrm{d}x}
    & =\frac{\int^\infty_0\left(f_\epsilon''\right)^2 r^{N-1-2a}\mathrm{d}r}
    {\int^\infty_0(f_\epsilon')^2 r^{N-3-2a}\mathrm{d}r}
    +(1+2a)(N-1)
    \\
    & \to \left(\frac{N-2-2a}{2}\right)^2 +(1+2a)(N-1)=\left(\frac{N+2a}{2}\right)^2,
    \end{align*}
    as $\epsilon\to 0$.
    However, the case $a=\frac{N-4}{2}$ is not covered for the test functions sequence $\{\hat{u}_\epsilon\}_{\epsilon>0}$ because of the nontermination $0/0$, thus we consider another test functions sequence
    $$
    \left\{\bar{u}_\epsilon(x):=
    \int^\infty_{|x|}h_\epsilon(t)\mathrm{d}t\right\}_{\epsilon>0},
    $$
    where $h_\epsilon(t)$ is given as in \eqref{deftfe}, and it is enough to show that the family of functions $\left\{\bar{u}_\epsilon\right\}_{\epsilon>0}$
    is a minimizing sequence for $\left(\frac{N+2a}{2}\right)^2=(N-2)^2$. It is obvious that $\bar{u}_\epsilon(x)=\bar{u}_\epsilon(|x|)$ is radially symmetry, then from \eqref{sc1}-\eqref{sc2} and \eqref{fe1e} we obtain
    \begin{align}\label{scle}
    \frac{\int_{\mathbb{R}^{N}}|x|^{4-N}|\Delta \bar{u}_\epsilon|^2\mathrm{d}x}
    {\int_{\mathbb{R}^N}|x|^{-N}(x\cdot \nabla \bar{u}_\epsilon)^2 \mathrm{d}x}
    & =\frac{\int^\infty_0\left(h_\epsilon'\right)^2 r^{3}\mathrm{d}r}
    {\int^\infty_0h_\epsilon^2 r\mathrm{d}r}
    +(N-3)(N-1)
    \nonumber\\
    & \to (N-2)^2,\quad \mbox{as}\ \epsilon\to 0.
    \end{align}
    Thus the claims \eqref{defsc1} for $N\geq 2$ and \eqref{defsc11} for $N=1$ hold.

    Now, the proof of Theorem \ref{thmhriw} is completed.
\qed

\begin{remark}\label{remwfhri}\rm
    For each $a\neq \frac{N-4}{2}$, from \eqref{scl} we can deduce the sharp constant in \eqref{hriw} satisfying
    \begin{align*}
    C(N,a)\leq \left(\frac{2}{N-4-2a}\right)^2\inf\limits_{k\in\mathbb{N}}
    \left(k+\frac{N}{2}+a\right)^2\left(k+\frac{N-4}{2}-a\right)^2
    \leq \left(\frac{N+2a}{2}\right)^2.
    \end{align*}
    While, for the cases \eqref{defsc1} when $N\geq 2$ and \eqref{defsc11} when $N=1$,
    \begin{align*}
    C(N,a)=\left(\frac{N+2a}{2}\right)^2,
    \end{align*}
    therefore, with addition condition $a\neq \frac{N-4}{2}$ we have
    \begin{align*}
    \left(\frac{2}{N-4-2a}\right)^2\inf\limits_{k\in\mathbb{N}}
    \left(k+\frac{N}{2}+a\right)^2\left(k+\frac{N-4}{2}-a\right)^2
    =\left(\frac{N+2a}{2}\right)^2.
    \end{align*}
    To sum up, combining with the cases \eqref{defsc2} for $N\geq 2$ and \eqref{defsc21} for $N=1$, we have
    \begin{align}\label{defcnas2}
    C(N,a)=\left(\frac{2}{N-4-2a}\right)^2\inf\limits_{k\in\mathbb{N}}
    \left(k+\frac{N}{2}+a\right)^2\left(k+\frac{N-4}{2}-a\right)^2
    \end{align}
    if $a<\frac{N-4}{2}$ or $\frac{N-4}{2}<a<\frac{N-2}{2}$, and $C\left(N,\frac{N-4}{2}\right)=(N-2)^2$, thus $C(N,a)>0$ if and only if $a$ satisfies \eqref{defac} or $a=\frac{N-4}{2}$. But, it should be emphasized that for the cases \eqref{defsc2} when $N\geq 2$ and \eqref{defsc21} when $N=1$, from the previous analysis in Step \uppercase\expandafter{\romannumeral 3},
    \begin{align*}
    \left(\frac{2}{N-4-2a}\right)^2
    \inf\limits_{k\in\mathbb{N}}
    \left(k+\frac{N}{2}+a\right)^2\left(k+\frac{N-4}{2}-a\right)^2
    <\left(\frac{N+2a}{2}\right)^2.
    \end{align*}
    \end{remark}

\medskip

\noindent{\bfseries Statements and Declarations}
The authors declare that they have no conflict of interest.

\medskip

\noindent{\bfseries Acknowledgements}
The research has been supported by National Natural Science Foundation of China (No. 12371121).


\begin{thebibliography}{99}

\bibitem{Be08}
Beckner, W.: Weighted inequalities and Stein-Weiss potentials. {\em Forum Math.} {\bf 20}(4), 587--606 (2008)

\bibitem{BDMN17}
Bonforte, M., Dolbeault, J., Muratori, M., Nazaret, B.: Weighted fast diffusion equations (Part \uppercase\expandafter{\romannumeral 1}): Sharp asymptotic rates without symmetry and symmetry breaking in Caffarelli-Kohn-Nirenberg inequalities. {\em Kinet. Relat. Models} {\bf 10}(1), 33--59 (2017)

\bibitem{CKN84}
Caffarelli, L., Kohn, R., Nirenberg, L.: First order interpolation inequalities with weights. {\em Compos. Math.} {\bf 53}, 259--275 (1984)

\bibitem{CC16}
Caldiroli, P., Cora, G.: Entire solutions for a class of fourth-order semilinear elliptic equations with weights. {\em Mediterr. J. Math.} {\bf 13}(2), 657--675 (2016)

\bibitem{CM11}
Caldiroli, P., Musina, R.: On Caffarelli-Kohn-Nirenberg-type inequalities for the weighted biharmonic operator in cones. {\em Milan J. Math.} {\bf 79}(2), 657--687 (2011)

\bibitem{CM12}
Caldiroli, P., Musina, R.: Rellich inequalities with weights. {\em Calc. Var. Partial Differential Equations} {\bf 45}(1-2), 147--164 (2012)

\bibitem{CW01}
Catrina, F., Wang, Z.-Q.: On the Caffarelli-Kohn-Nirenberg inequalities: sharp constants, existence (and nonexistence), and symmetry of extremal functions. {\em Comm. Pure Appl. Math.} {\bf 54}(2), 229--258 (2001)

\bibitem{Ca20}
Cazacu, C.: A new proof of the Hardy-Rellich inequality in any dimension. {\em Proc. Roy. Soc. Edinburgh Sect. A} {\bf 150}(6), 2894--2904 (2020)

\bibitem{CL08}
Chen, W., Li, C.: The best constant in a weighted Hardy-Littlewood-Sobolev inequality. {\em Proc. Amer. Math. Soc.} {\bf 136}(3), 955--962 (2008)

\bibitem{CLO06}
Chen, W., Li, C., Ou, B.: Classification of solutions for an integral equation. {\em Comm. Pure Appl. Math.} {\bf 59}(3), 330--343 (2006)

\bibitem{CC93}
Chou, K.S., Chu, C.W.: On the best constant for a weighted Sobolev-Hardy inequality. {\em J. Lond. Math. Soc. (2)} {\bf 48}(1), 137--151 (1993)

\bibitem{DM22}
D'Ambrosio, L., Mitidieri, E.: Entire solutions of certain fourth order elliptic problems and related inequalities. {\em Adv. Nonlinear Anal.} {\bf 11}(1), 785--829 (2022)

\bibitem{DMY20}
Dan, S., Ma, X., Yang, Q.: Sharp Rellich-Sobolev inequalities and weighted Adams inequalities involving Hardy terms for bi-Laplacian. {\em Nonlinear Anal.} {\bf 200}, 112068, 18 pp (2020)

\bibitem{DH98}
Davies, E.B., Hinz, A.M.: Explicit constants for Rellich inequalities in $L^p(\Omega)$. {\em Math. Z.} {\bf 227}(3), 511--523 (1998)

\bibitem{DGT23-jde}
Deng, S., Grossi, M., Tian, X.: On some weighted fourth-order equations. {\em J. Differ. Equations} {\bf 364}, 612--634 (2023)

\bibitem{DT23-jfa}
Deng, S., Tian, X.: Some weighted fourth-order Hardy-H\'{e}non equations.  {\em J. Funct. Anal.} {\bf 284}(1), Paper No. 109745, 28 pp (2023)

\bibitem{DT24}
Deng, S., Tian, X.: Classification and non-degeneracy of positive radial solutions for a weighted fourth-order equation and its application. {\em Nonlinear Anal.} {\bf 240}, Paper No. 113468, 15 pp (2024)

\bibitem{DT23-f}
Deng, S., Tian, X.: Symmetry breaking of extremals for the high order Caffarelli-Kohn-Nirenberg type inequalities. Preprint, \url{https://arxiv.org/abs/2308.07568} (2024)

\bibitem{DT24-2}
Deng, S., Tian, X.: Symmetry breaking of extremals for the high order Caffarelli-Kohn-Nirenberg type inequalities: the singular case.  Preprint, https://arxiv.org/abs/2409.18154 (2024)

\bibitem{DEL16}
Dolbeault, J., Esteban, M.J., Loss, M.: Rigidity versus symmetry breaking via nonlinear flows on cylinders and Euclidean spaces. {\em Invent. Math.} {\bf 206}(2), 397--440 (2016)


\bibitem{DELM17}
Dolbeault, J., Esteban, M.J., Loss, M., Muratori, M.: Symmetry for extremal functions in subcritical Caffarelli-Kohn-Nirenberg inequalities. {\em C. R. Math. Acad. Sci. Paris} {\bf 355}(2), 133--154 (2017)

\bibitem{DELT09}
Dolbeault, J., Esteban, M.J., Loss, M., Tarantello, G.:
On the symmetry of extremals for the Caffarelli-Kohn-Nirenberg inequalities. {\em Adv. Nonlinear Stud.} {\bf 9}(4), 713--726 (2009)

\bibitem{EFJ90}
Edmunds, D.E., Fortunato, D., Janelli, E.: Critical exponents, critical dimensions, and the biharmonic operator. {\em Arch. Rational Mech. Anal.} {\bf 112}, 269--289 (1990)

\bibitem{FS03}
Felli, V., Schneider, M.: Perturbation results of critical elliptic equations of Caffarelli-Kohn-Nirenberg type. {\em J. Differ. Equations} {\bf 191}, 121--142 (2003)

\bibitem{GM11}
Ghoussoub, N., Moradifam, A.: Bessel pairs and optimal Hardy and Hardy-Rellich inequalities. {\em Math. Ann.} {\bf 349}(1), 1--57 (2011)

\bibitem{GG22}
Guan, X., Guo, Z.: New types of Caffarelli-Kohn-Nirenberg inequalities and applications. {\em Commun. Pure Appl. Anal.} {\bf 21}(12), 4071--4087 (2022)

\bibitem{HL28}
Hardy, G.H., Littlewood, J.E.: Some properties of fractional integrals. \uppercase\expandafter{\romannumeral 1}. {\em Math. Z.} {\bf 27}(1), 565--606 (1928)

\bibitem{Ho97}
Horiuchi, T.: Best constant in weighted Sobolev inequality with weights being powers of distance from the origin. {\em J. Inequal. Appl.} {\bf 1}(3), 275--292 (1997)

\bibitem{JL06}
Jin, C., Li, C.: Symmetry of solutions to some systems of integral equations. {\em Proc. Amer. Math. Soc.} {\bf 134}(6), 1661--1670 (2006)

\bibitem{LL17}
Lam, N., Lu, G.: Sharp constants and optimizers for a class of Caffarelli-Kohn-Nirenberg inequalities. {\em Adv. Nonlinear Stud.} {\bf 17}(3), 457--480 (2017)

\bibitem{Lieb83}
Lieb, E.: Sharp constants in the Hardy-Littlewood-Sobolev and related inequalities. {\em Ann. of Math. (2)} {\bf 118}(2), 349-374 (1983)

\bibitem{Li86}
Lin, C.-S.: Interpolation inequalities with weights. {\em Comm. Partial Differential Equations} {\bf 11}(4), 1515--1538 (1986)

\bibitem{Li98}
Lin, C.-S.: A classification of solutions of a conformally invariant fourth order equation in $\mathbb{R}^N$. {\em Comment. Math. Helv.} {\bf 73}(2), 206--231 (1998)

\bibitem{Li85-1}
Lions, P.-L.: The concentration-compactness principle in the calculus of variations. The limit case. \uppercase\expandafter{\romannumeral 1}. {\em Rev. Mat. Iberam.} {\bf 1}(1), 145--201 (1985)

\bibitem{MNSS21}
Metafune, G., Negro, L., Sobajima, M., Spina, C.: Rellich inequalities in bounded domains. {\em Math. Ann.} {\bf 379}(1-2), 765--824 (2021)

\bibitem{Mi00}
Mitidieri, E.: A simple approach to Hardy inequalities. {\em Math. Notes} {\bf 67}, 479--486 (2000)

\bibitem{MS14}
Musina, R., Sreenadh, K.: Radially symmetric solutions to the H\'{e}non-Lane-Emden system on the critical hyperbola. {\em Commun. Contemp. Math.} {\bf 16}(3), 1350030, 16 pp (2014)

\bibitem{So63}
Sobolev, S.L.: Some generalizations of imbedding theorems. {\em Amer. Math. Soc. Transl. (2)} {\bf 30}, 295--344 (1963)

\bibitem{SW58}
Stein, M., Weiss, G.: Fractional integrals on $n$-dimensional Euclidean space. {\em J. Math. Mech.} {\bf 7}, 503--514 (1958)

\bibitem{SW12}
Szulkin, A., Waliullah, S.: Sign-changing and symmetry-breaking solutions to singular problems. {\em Complex Var. Elliptic Equ.} {\bf 57}(11), 1191--1208 (2012)

\bibitem{Ta76}
Talenti, G.: Best constant in Sobolev inequality. {\em Ann. Mat. Pura Appl. (4)} {\bf 110}, 353--372 (1976)

\bibitem{TZ07}
Tertikas, A., Zographopoulos, N.B.: Best constants in the Hardy-Rellich inequalities and related improvements. {\em Adv. Math.} {\bf 209}, 407--459 (2007)


\bibitem{Va93}
Van der Vorst, R.: Best constant for the embedding of the space $H^2\cap H^1_0(\Omega)$ into $L^{2N/(N-4)}(\Omega)$. {\em Differ. Integral Equ.} {\bf 6}(2), 259--276 (1993)

\bibitem{Ya21}
Yan, Y.: Classification of positive radial solutions to a weighted biharmonic equation. {\em Commun. Pure Appl. Anal.} {\bf 20}(12), 4139--4154 (2021)

\end{thebibliography}
    \end{document}